\documentclass[11pt]{article}
\usepackage{amsmath, amssymb, amsthm}

\usepackage{enumitem}                                 
\usepackage[margin=1.0in]{geometry}

\AtEndDocument{\bigskip{\footnotesize%
  \textsc{Einstein Institute of Mathematics, Edmond J. Safra Campus, The Hebrew University of Jerusalem,
Givat Ram. Jerusalem, 9190401, Israel.} \par  
  \textit{E-mail address} \texttt{ amir.algom@mail.huji.ac.il
} 
}}

\newtheorem{theorem}{Theorem}[section]

\newtheorem{Conjecture} [theorem]{Conjecture}

\newtheorem{Counter-example}[theorem]{Counter example}
\newtheorem{Claim}[theorem]{Claim}

\newtheorem{Lemma}[theorem]{Lemma}
\newtheorem{Proposition}[theorem]{Proposition}

\newtheorem{Corollary}[theorem]{Corollary}

\newtheorem*{theorem*}{Theorem}

\newcommand{\image}{{\rm im}}

\newcommand{\ignore}[1]{}

\newcommand{\id}{{\rm Id}}

\DeclareMathOperator*{\cmpct}{cpct}

\DeclareMathOperator*{\diag}{diag}

\newcommand\blfootnote[1]{%
  \begingroup
  \renewcommand\thefootnote{}\footnote{#1}%
  \addtocounter{footnote}{-1}%
  \endgroup
}

\title{Affine embeddings of Cantor sets in the plane}

\author{Amir Algom}

\date{}

\begin{document}
\maketitle

\begin{abstract}
Let \blfootnote{Supported by ERC grant 306494}  $F,E\subseteq \mathbb{R}^2$  be two self similar sets. First, assuming $F$ is generated by an IFS $\Phi$ with strong separation, we characterize the affine maps $g:\mathbb{R}^2 \rightarrow \mathbb{R}^2$ such that $g(F)\subseteq F$. Our analysis depends on the cardinality of the group $G_\Phi$ generated by the orthogonal parts of the similarities in $\Phi$. When $|G_\Phi|=\infty$ we show that any such self embedding must be a similarity, and so (by the results of Elekes, Keleti and M\'ath\'{e} \cite{elekes2010self})  some power of its orthogonal part lies in $G_\Phi$. When $|G_\Phi| < \infty$ and $\Phi$ has a uniform contraction $\lambda$, we show that the linear part of any such embedding is diagonalizable, and the norm of each of its eigenvalues is a rational power of $\lambda$.

 We also study the existence and properties of affine maps $g$ such that $g(F)\subseteq E$, where $E$ is generated by an IFS $\Psi$. In this direction,  we provide  more evidence for a Conjecture of Feng, Huang and Rao \cite{feng2014affine}, that such an embedding exists only if the contraction ratios of the maps in $\Phi$ are algebraically dependent on the contraction ratios of the maps in $\Psi$. Furthermore, we show that, under some conditions, if $|G_\Phi|=\infty$ then $|G_\Psi|=\infty$ and if $|G_\Phi|<\infty$ then $|G_\Psi|<\infty$.
\end{abstract}

\tableofcontents

\section{Introduction}
Let $F,E \subset \mathbb{R}^2$ be two Cantor sets. We say that $F$ may be affinely embedded into $E$ if there exists an affine map $g:\mathbb{R}^2 \rightarrow \mathbb{R}^2$,  $g(z)= A\cdot z+t$ where $A\in GL(\mathbb{R}^2),t\in \mathbb{R}^2$ such that $g(F)\subseteq E$. When $F=E$ we say that $g$ is a self embedding of $F$.

In this paper we pursue two main objectives. The first is to exhibit how the microscopic structure of $F$ imposes severe restrictions on its self embeddings. Specifically, for dynamically defined sets one expects that any self-symmetry should "come from" the generating dynamics. For example, we expect that the eigenvalues of the linear part of a self embedding of $F$ to be algebraically related to the scales appearing in the multiscale structure of $F$. In particular,  there should not be too many self embeddings of $F$, where too many can mean e.g. positive Hausdorff dimension or even uncountable cardinality.  The second objective is to show that if $F$ can be affinely embedded into $E$ then their respective microscopical structures must be compatible in some sense. For example, following a conjecture of  Feng, Huang and Rao \cite{feng2014affine}, if $F$ and $E$ are self similar sets we expect  that the contraction ratios of the IFS's defining $F$ and $E$ should have some algebraic relations (see Conjecture \ref{Fengs conjecture} for a precise statement).

 We devote the subsequent sections to stating our results and discussing the relevant literature. We focus our attention solely on self similar sets in the plane. However,  we note that the case of affine embeddings of self affine sets is also tractable in some cases; see our recent joint work with Hochman \cite{algom2016self} about self-embeddings of strictly self affine Bedford-McMullen carpets. 
 
\textbf{Standing assumptions} Throughout this paper, unless stated otherwise, all self similar sets are assumed to be in $\mathbb{R}^2$. We shall always assume that they do  not lie on a $1$-dimensional affine line, and that their Hausdorff dimension is strictly between $0$ and $2$.

\subsection{Self embeddings of self similar sets} We begin with the case of self embeddings. Recall that a similarity map $g: \mathbb{R}^2 \rightarrow \mathbb{R}^2$ is an affine map $g(z)= \alpha O\cdot z+t$ such that $O\in O(\mathbb{R}^2)$, where $O(\mathbb{R}^2)$ is the orthogonal group of the Euclidean space $\mathbb{R}^2$, $\alpha \neq 0$ is a scalar and $t\in \mathbb{R}^2$. We call $O$ the orthogonal part of $g$.

Let $\Phi = \lbrace \phi_i \rbrace_{k=1} ^l , l\in \mathbb{N}, l\geq 2$ be a family of contractions $\phi_i : \mathbb{R}^d \rightarrow \mathbb{R}^d, d\geq 1$. The family $\Phi$ is called an iterated function system, abbreviated IFS, the term being coined by Hutchinson \cite{hutchinson1981fractals}, who defined them and studied some of their fundamental properties. In particular, he proved that there exists a unique compact $\emptyset \neq F \subset \mathbb{R}^d$ such that $F = \bigcup_{i=1} ^l \phi_i (F)$. $F$ is called the attractor of $\Phi$, and $\Phi$ is called a generating IFS for $F$.  A set $F \subset \mathbb{R}^d$ will be called self similar if there exists a generating IFS $\Phi$ for $F$ such that $\Phi$ consists only of similarity mappings.

Let $F$ be a self similar set with a generating IFS $\Phi$. We shall say that $\Phi$   satisfies the strong separation condition, abbreviated SSC, if $i \neq j \Rightarrow \phi_i (F) \cap \phi_j (F) = \emptyset$.  We shall say that it satisfies the open set condition, hereafter abbreviated OSC, if there exists an open set $\emptyset \neq U \subset \mathbb{R}^d$ such that $i \neq j \Rightarrow \phi_i (U) \cap \phi_j (U) = \emptyset$, and $\phi_i (U) \subseteq U$ for any $\phi_i \in \Phi$. Also, For a multi-index $I=(i_1,...,i_k) \in \lbrace 1,..,l \rbrace^*$, we define
\begin{equation*}
\phi_I = \phi_{i_1} \circ ... \circ \phi_{i_k}.
\end{equation*}

 Suppose $F \subseteq \mathbb{R}^2$ is a self similar set, generated by an IFS $\Phi = \lbrace \phi_i \rbrace_{i=1} ^l$ such that all the maps in $\Phi$ are similarities. Let $\mathcal{E}(F)$ denote the set of affine self embeddings of $F$, that is
\begin{equation} \label{equation set of self embeddings}
\mathcal{E}(F)=\lbrace g:\mathbb{R}^2\rightarrow \mathbb{R}^2 : g \text{ is affine  and } g(F)\subseteq F \rbrace.
\end{equation}
Let $G_\Phi \leq  O (\mathbb{R}^2) $ be the group generated by the orthogonal parts of the similarities in the IFS $\Phi$. We first discuss our results on the nature of the maps in $\mathcal{E}(F)$, depending on weather $G_\Phi$ is finite or infinite.

\begin{theorem} \label{Theorem embeddings}
Let $F \subseteq \mathbb{R}^2$ be a self similar set such that $0<\dim_H F<2$, that is generated by an IFS $\Phi$ satisfying the SSC, and let  $g:\mathbb{R}^2 \rightarrow \mathbb{R}^2$ be an affine map such that $g(F) \subseteq F$.  
\begin{enumerate}
\item \label{Theorem embeddings of irrational type} If $|G_\Phi| = \infty$ then $g$ is a similarity, and there exists some $k\in \mathbb{N}$ and $I,J \in \lbrace 1,...,l\rbrace^*$ such that
\begin{equation*}
g^k \circ \phi_I = \phi_J.
\end{equation*}

\item \label{Theorem embeddings of rational type} If $|G_\Phi| < \infty$, and if all the maps in $\Phi$ share the same contraction ratio $\lambda>0$, then the linear part of $g$ is diagonalizable over $\mathbb{C}$, with eigenvalues $\gamma_1,\gamma_2 \in \mathbb{C}$ satisfying that $|\gamma_1| = \lambda^p, |\gamma_2| = \lambda^q$ for some $p,q \in \mathbb{Q}$.
\end{enumerate}
\end{theorem}

Part (1) of Theorem \ref{Theorem embeddings} is in fact a generalization of a result of Elekes, Keleti, and M\'ath\'{e}. In \cite{elekes2010self}, they proved that for a self similar set $F \subset \mathbb{R}^d$ with the SSC, if $g$ is a similarity such that $g(F)\subseteq F$ then there exists some $k\in \mathbb{N}$ and $I,J \in \lbrace 1,...,l \rbrace^*$ such that
\begin{equation*}
g^k \circ \phi_I = \phi_J.
\end{equation*}
The arguments used by them to prove this result are based on the scaling nature of Hausdorff measures with respect to similarity maps. Thus, the main novelty of part (1) of Theorem \ref{Theorem embeddings} is in its treatment of general affine map (i.e. not necessarily similarity maps).  

Other results in this direction include the work of Feng and Wang \cite{feng2009structures}, that treat self embeddings of attractors of homogeneous (i.e. with uniform contraction ratio) IFS's in dimension $1$.  Generalizing this work, in \cite{Deng2013equivalence} Deng and Lau proved that for self similar sets in $\mathbb{R}^2$, if $\Phi$ and $\Psi$ are two homogeneous generating IFS's (i.e. with uniform linear parts) for $F$ each satisfying the SSC, then there exists some $m,n\in \mathbb{N}$ such that $\Phi^m = \Psi^n$, where $\Phi^m$ is the generating IFS for $F$ consisting of $m$-fold compositions of maps from $\Phi$. Recently, in \cite{Deng2017structure}, the same authors  generalized this result by replacing the SSC with the OSC. Although this is a related line of research, it is also rather different: while Deng and Lau search for algebraic relations between different self similar generating IFS's of the same attractor, we are more interested in affine non-similarity maps that preserve the attractor.

Theorem \ref{Theorem embeddings} is also related to the  work of Bonk and Merenkov \cite{bonk2013quasi}. In this paper, the authors show that any quasi-symmetric self map of the standard Sierpi\'{n}ski carpet with respect to base $3$ must be an isometry (they also obtain results about self maps of corresponding standard Sierpi\'{n}ski carpets for odd integers $p\geq 3$). In general, the class of quasi-symmetric maps is much more general than the class of affine maps. However, there are two main differences between our results and those of Bonk and Merenkov: First, quasi-symmetric maps are by definition homeomorphisms, and in particular are onto. In our setting, the self  maps do not necessarily have to be onto (they are onto as maps $g:\mathbb{R}^2 \rightarrow \mathbb{R}^2$, but the containment $g(F) \subseteq F$ may be strict). Secondly, Bonk and Merenkov consider Sierpi\'{n}ski carpets, which are endowed with a special grid structure, as well as being self similar sets generated by IFS's with $G_\Phi = \lbrace \id \rbrace$. We allow our self similar IFS's to have rotations, and also do not require the special grid structure present in Sierpi\'{n}ski carpets.

We proceed to discuss part (2) of Theorem \ref{Theorem embeddings}: the case when $|G_\Phi| <\infty$. We note that in this setting, it is no longer true that if $g$ is an affine map such that $g(F)\subseteq F$ then $g$ is a similarity. Indeed, let $F=C\times C \subset \mathbb{R}^2$, where $C\subset \mathbb{R}$ is the classic middle thirds Cantor set. Then $F$ admits a generating  IFS 
\begin{equation*}
\Phi = \lbrace \phi_{i,j} (x,y) = (\frac{x+i}{3},\frac{y+j}{3})\rbrace_{(i,j)\in \lbrace (0,0),(2,0),(0,2),(2,2)\rbrace}.
\end{equation*}
Note that $\Phi$ has the SSC, $G_\Phi = \lbrace \id \rbrace$ and all the maps in $\Phi$ share  a uniform contraction ratio $\frac{1}{3}$. For every $m\neq n \in \mathbb{N}$ define the linear maps
\begin{equation*}
g_1 ^{n,m} (x,y) = (\frac{x}{3^n},\frac{y}{3^m}), \quad g_2 ^{n,m} (x,y) = (\frac{y}{3^n},\frac{x}{3^m}).
\end{equation*}
This defines a countable family of linear maps, and every map in this family is an affine self-embedding of $F$ that is not a similarity. Note, however, that as predicted by part (2) of Theorem \ref{Theorem embeddings}, the eigenvalues of every map in this family are rational powers of $\frac{1}{3}$, the uniform contraction ratio of $\Phi$.

We also notice that part (2) of Theorem \ref{Theorem embeddings} fails if we replace the SSC with the OSC. Indeed, consider the set $E= C \times [0,1]$. Then $\dim_H E = \frac{\log 6}{\log 3}$, and $E$  is in fact the attractor of the IFS
\begin{equation*}
\lbrace x\mapsto \frac{x}{3}, x\mapsto \frac{x}{3}+(0, \frac{1}{3}), x\mapsto \frac{x}{3}+(0, \frac{2}{3}), x\mapsto \frac{x}{3}+(\frac{2}{3},0), x\mapsto \frac{x}{3}+( \frac{2}{3},\frac{1}{3}), x\mapsto \frac{x}{3}+(\frac{2}{3},\frac{2}{3}) \rbrace.
\end{equation*}
One easily sees that for any $a\in [0,1]$, the linear map 
\begin{equation*}
(x,y)\mapsto (\frac{x}{3}, y\cdot a)
\end{equation*}
defines a self embedding of $E$, so the characterization of the eigenvalues of self embeddings in part (2) of Theorem \ref{Theorem embeddings} is no longer true. Moreover, consider the linear map $g:\mathbb{R}^2 \rightarrow \mathbb{R}^2$ defined by the matrix
\begin{equation*}
A = \begin{pmatrix}
\frac{1}{3} & 0 \\
a & \frac{1}{3} \\
\end{pmatrix} \text{ where } 0< a \leq  \frac{2}{3}.
\end{equation*}
Then it is readily seen that $g(E)\subseteq E$, but $A$ is not diagonalizable.

Let us make one final remark about part (2) of Theorem \ref{Theorem embeddings}. Let  $g$ be a self embedding of $F$, where $g$ is an affine map and $F$ is as in part (2) of Theorem \ref{Theorem embeddings}. Suppose that the linear part of $g$ has two complex non-real eigenvalues, so  they are complex conjugates of each other and have the same norm. It is well known that this implies that the linear part of $g$ is similar to a similarity matrix (a scaled rotation matrix). In this context, we believe  that this should imply that $g$ is already a similarity map. However, we do not have a proof of this claim.

We proceed to state our general results about the set $\mathcal{E}(F)$, defined in equation \eqref{equation set of self embeddings}, which we can naturally identify with a subset of $\mathbb{R}^6$.  Note that, under the conditions of part (1) of Theorem \ref{Theorem embeddings}, $\mathcal{E}(F)$ must be countable. This is because  Elekes, Keleti, and M\'ath\'{e} proved in \cite{elekes2010self} that the set of similarity self-embeddings of a self similar set with the SSC is countable, and by part (1) of Theorem \ref{Theorem embeddings} all the  maps in $\mathcal{E}(F)$ are similarities.  We would like, however, to say something about $\mathcal{E}(F)$ without these assumptions.

 Let $G(2,1)$ denote the collection of all $1$-dimensional linear spaces in $\mathbb{R}^2$. For $V\in G(2,1)$ let $P_V :\mathbb{R}^2 \rightarrow V$ denote the orthogonal projection from $\mathbb{R}^2$ onto  $V$.  Let $\mu \in P(\mathbb{R}^2)$ be a self similar measure with respect to an IFS $\Phi = \lbrace \phi_i \rbrace_{i=1} ^l$. That is, for a non-degenerate probability vector $(p_1,..,p_l)$,  $\mu$ is the unique measure satisfying
\begin{equation*}
\mu = \sum_{i=1} ^l p_i \cdot \phi_i \mu.
\end{equation*}
We say that $\mu$ admits a $1$-slicing if there exists some line $V\in G(2,1)$  such that:
\begin{enumerate}

\item $V$ is $D \Phi$-invariant, i.e. for every $A\in O(\mathbb{R}^2)$ that is a linear part of some $g\in \Phi$, the line $A(V)$ is equal to $V$.

\item The conditional measures of $\mu$ on translates of $V$, obtained by disintegrating $\mu$ according to the projection $P_{V^\perp}$, are $P_{V^\perp} \mu = \mu \circ P_{V^\perp} ^{-1}$-a.s. exact dimensional and of dimension $1$.

\end{enumerate}
Otherwise, we say that $\mu$ does not admit a $1$-slicing. The following Proposition provides many examples of self similar measures that do not admit a $1$-slicing:
\begin{Proposition} \label{Theorem saturation for self similar measures}
Let $\mu$ be a self similar measure with respect to an IFS $\Phi$ on $\mathbb{R}^2$. 
\begin{enumerate}
\item Suppose either $5\leq |G_\Phi|$ or  $|G_\Phi \cap SO(\mathbb{R}^2)|>2$. Then $\mu$  does not admit a $1$-slicing.

\item If $\Phi$ satisfies the SSC then $\mu$ is does not admit a $1$-slicing.
\end{enumerate}
\end{Proposition}

Our most general result on the dimension of $\mathcal{E}(F)$ is the following. By a self similar measure of maximal dimension we mean a self similar measure that has  Hausdorff dimension $\dim F$.
\begin{theorem} \label{Theorem zero dimension of self embeddings}
Let $F\subseteq \mathbb{R}^2$ be a self similar set generated by an IFS $\Phi$, such that $0<\dim F <2$. Suppose $F$ supports a self similar measure of maximal dimension, and that this measure does not admit a $1$-slicing. Then $\dim \mathcal{E}(F) =0$.
\end{theorem}

In the above Theorem, and in this paper in general, unless stated otherwise $\dim$ refers to Hausdorff dimension. The requirement that the aformentioned measure does not admit a $1$-slicing is needed in order to apply Hochman's two-dimensional inverse theorem for entropy, which is discussed in  section \ref{Section two dimensional inv Thm}.  We note that  there is a possibility that the conclusion can  be strengthened to "$\mathcal{E}(F)$ must be countable" (this happens e.g.  under the conditions of part (1) of Theorem \ref{Theorem embeddings} as discussed earlier).  Finally, we remark that Hochman proved in \cite{hochman2016some} a similar  result for one dimensional attractors of (possibly infinite) compact IFS's.

\subsection{Affine embeddings of one self similar set into another one}

To begin our discussion of affine (not necessarily self) embeddings of self similar sets, we recall a recent Conjecture of  Feng,  Huang, and  Rao, stated in \cite{feng2014affine}. Throughout this section,  let $F$ and $E$ be two self similar sets in $\mathbb{R}^d$, generated by IFS's $\Phi = \lbrace \phi_i \rbrace_{i=1} ^l$ and $\Psi = \lbrace \psi_j  \rbrace_{j=1} ^m$, respectively. Denote, for $1\leq i \leq l$ and $1 \leq j \leq m$, the contraction ratio of $\phi_i$ by $\alpha_i \in (0,1)$ and of $\psi_j$ by $\beta_j \in (0,1)$. 
\begin{Conjecture} \cite{feng2014affine} \label{Fengs conjecture}
Suppose that $E$ and $F$ are totally disconnected and that $F$ can be affinely embedded into $E$. Then for every $1 \leq i \leq l$ there exists $t_{i,j} \in \mathbb{Q}, t_{i,j} \geq 0$ such that 
\begin{equation*}
\alpha_i = \prod_{j=1} ^m \beta_j ^{t_{i,j}} 
\end{equation*}
In particular, if $\beta_j = \beta$ for all $1\leq j \leq m$ then for every $1 \leq i \leq l$
\begin{equation*}
\frac{\log \alpha_i}{\log \beta} \in \mathbb{Q}.
\end{equation*}
\end{Conjecture}
Note that these arithmetic conditions on $\alpha_i,\beta_j$ hold true when $\Phi$ and $\Psi$ satisfy the strong separation condition, and $E$ and $F$ are Lipschitz equivalent, by the results of Falconer and Marsh \cite{falconer1992lipschitz}. However, for Lipschitz embeddings no arithmetic conditions are required. Indeed,  in  \cite{deng2011bilipschitz}, Deng, Wen, Xiong, and Xi proved that if $F$ and $E$ are attractors of IFS's $\Phi$ and $\Psi$ that satisfy the strong separation condition and $\dim F < \dim E$ then $F$ can be Lipschitz embedded into $E$.

Recent works provide strong evidence for this Conjecture's validity. Indeed, Feng,  Huang, and  Rao, proved it assuming $E$ admits an IFS with uniform contractions and the SSC, and $\dim E<\frac{1}{2}$. Later on, in \cite{algom2016affine}, we showed it to be true for $d=1$ assuming $\dim E - \dim F < \delta$ where $\delta = \delta (\dim F)>0$, and that $E$ admits an IFS with uniform contractions and the SSC. A major breakthrough for the case $d=1$ was obtained independently by Shmerkin \cite{shmerkin2016furstenberg} and Wu \cite{wu2016proof}. Shmerkin  proved the Conjecture holds assuming $E$ admits an homogeneous generating IFS with the OSC and $\dim E < 1$. Wu  proved the conjecture under almost the same assumptions, except for requiring that $E$ admits a generating IFS with the SSC. In the non-homogeneous setting there is the recent work of Feng and Xiong \cite{feng2016affine}, who proved some cases of the conjecture when $E$ does not have a uniform contraction, and has small dimension in a manner depending on these contractions.

We proceed to state results further supporting this conjecture. For $F$ and $E$, two self similar sets in $\mathbb{R}^2$, let $\mathcal{E}(F,E)$ denote the set of affine embeddings of $F$ into $E$. That is,
\begin{equation} \label{Equation E(F,E)}
\mathcal{E}(F,E)=\lbrace g:\mathbb{R}^2\rightarrow \mathbb{R}^2 : g \text{ is affine  and } g(F)\subseteq E \rbrace.
\end{equation}
Notice that we may always assume $F,E \subseteq [0,1]^2$. This does not make our results less general, since for sets $F,E$  in the plane we can always dilate and translate them to obtain sets $F',E'$ in the unit cube, with every embedding of $F$ into $E$ inducing an embedding of $F'$ into $E'$, and vice versa.

The following Theorem shows that if $\mathcal{E}(F,E) \neq \emptyset$, then there should be some compatibility between generating IFS's for these sets, not only for the contractions but also for the rotations  involved.  

\begin{theorem} \label{Theorem one into one}
Assume   $0<\dim F<2$, and that $E$ is generated by $\Psi$, an IFS with the SSC and a uniform contraction ratio $\beta$. Suppose that one of the following conditions holds:
\begin{enumerate}
\item \label{Theorem evidence for Feng in higher dimenion} $F$  supports a self similar measure of maximal dimension  that does not admit a $1$-slicing, and there exists some $\alpha_i$ such that $\frac{\log \alpha_i}{\log \beta} \notin \mathbb{Q}$.

\item \label{Theorem evidence B} $\Phi$ satisfies $|G_\Phi|=\infty$, has the OSC,  and  $|G_\Psi|<\infty$.

\item \label{Theorem evidence C} $\Phi$ satisfies $|G_\Phi|<\infty$ and  $F$ supports a self similar measure of maximal dimension  that does not admit a $1$-slicing.  In addition, all the maps in $\Psi$ have the same linear part and $|G_\Psi|=\infty$.
\end{enumerate}
Then there exists $\delta=\delta( F)$ such that if $\dim E - \dim F < \delta$ then $\mathcal{E}(F,E) =\emptyset$.

In particular, if $F$  supports a self similar measure of maximal dimension  that does not admit a $1$-slicing, then Conjecture \ref{Fengs conjecture} holds for $F$ and $E$ whenever $\dim E - \dim F <\delta$, for some $\delta(F)>0$.
\end{theorem}

Each one of the conditions stated in Theorem \ref{Theorem one into one} ensures that we can apply the 2-dimensional inverse Theorem for entropy (see Section \ref{Section two dimensional inv Thm}). It is clear that conditions (2) and (3) are mutually exclusive. Moreover, by Proposition \ref{Theorem saturation for self similar measures} and the OSC, condition (2) ensures that $F$ supports a self similar measure of maximal dimension  that does not admit a $1$-slicing.

Note that in condition (3) of Theorem \ref{Theorem one into one}, we cannot relax the assumption that all the maps in $\Psi$ have the same linear part. Indeed, let $\Phi$ be the usual generating IFS for the product set $F = C\times C$ where $C$ is the middle thirds Cantor sets. Since $\Phi$ has the SSC, by Proposition \ref{Theorem saturation for self similar measures} we see that $F$ supports a self similar measure of maximal dimension that does not admit a $1$-slicing. Let $R\in SO(\mathbb{R}^2)$ be an irrational rotation, and let $\Psi_k = \Phi\cup \lbrace z\mapsto \frac{1}{3^k}\cdot R \cdot z + (\frac{1}{2},\frac{1}{2})\rbrace$, and let $E_k$ be its attractor. Assuming  $k>3$, $\Psi_k$ has the SSC, and by taking $k$ large we can make $\dim E_k$ arbitrarily close to $\dim F$. However, $F\subset E_k$, but $G_\Phi$ is finite and $G_{\Psi_k}$ is infinite.

\subsection{Methods and tools}
 The following sections contain the various statements of the inverse Theorems for entropy we will be using in this paper, and roughly how they configure in our proofs. 
 \subsubsection{One dimensional inverse Theorem for entropy}
Let $\mu$ be a probability measure on $\mathbb{R}$; for the definition of the upper and lower entropy dimension of $\mu$, $\overline{\dim_e} \mu, \underline{\dim_e} \mu,$ as well as the definition of the Hausdorff dimension of $\mu$, $\dim \mu$, see  the preliminaries section. Denote the set of invertible affine maps $\mathbb{R}\rightarrow \mathbb{R}$ by $G_1$, that is,
\begin{equation} \label{equation for G1}
G_1 = \lbrace g:\mathbb{R} \rightarrow \mathbb{R}, \quad g(x)=a\cdot x + b, \quad a \neq 0,b\in \mathbb{R} \rbrace.
\end{equation}
Thus $G_1$ can be naturally identified with an open subset of $\mathbb{R}^2$. For a set $X\subset \mathbb{R}^d$, $P(X)$ denotes the set of Borel probability measures supported on $X$.  For any $\theta \in P(G_1)$ and $\mu \in P(\mathbb{R})$ that are compactly supported, define the convolution measure $\theta . \mu \in P(\mathbb{R})$ as the push-forward of the measure $\theta \times \mu$ on $G_1 \times \mathbb{R}$ via the map $f:G_1 \times \mathbb{R} \rightarrow \mathbb{R}$ defined by $f(\phi,x) = \phi (x)$.

\begin{theorem} \cite{hochman2016some} \label{Theorem 4.2}
Let $s\in (0,1)$, then there exists some $\delta=\delta(s) >0$ such that:

Let $\mu \in P(\mathbb{R})$ and $\nu \in P(G_1)$ be compactly supported measures. Suppose that $\mu$ is a  self similar measure of maximal dimension with respect to some IFS that satisfies the OSC, and that its attractor $F$ satisfies $0 < \dim_H F \leq  s $. Suppose in addition that $\overline{\dim}_e \nu \geq  1$.  Then
\begin{equation*}
\overline{\dim}_e \nu. \mu \geq \dim_H F + \delta.
\end{equation*} 
\end{theorem}

We shall apply Theorem \ref{Theorem 4.2} in the proof of part (2) of Theorem \ref{Theorem embeddings}. Specifically, let $F$ be a self similar set  as in Theorem \ref{Theorem embeddings} part (2), and assume $g:\mathbb{R}^2 \rightarrow \mathbb{R}^2$ is an affine self embedding. Let $A$ be the linear part of $g$.  Let $n\in \mathbb{N}$. Then $g^n$, the $n$-th iteration of $g$,   is also a self embedding of $F$. Since $\Phi$ has the SSC, $g^n(F)$ is contained within a cylinder $\phi_n (F)$ so that the two sets have comparable diameters  (the generation of $ \phi_n$  depends on the operator norm of $A$). We can now rescale $g^n$ by pre-composing with $\phi_n$, i.e. we consider the sequence of self embeddings  $\lbrace \phi_n ^{-1} \circ g^n \rbrace$.  Now, if $A$ is not diagonalizable,
by taking any converging subsequence, we show that the linear part of this subsequence of affine maps converges to a projection matrix (i.e. a matrix with one dimensional kernel).  Thus, we see that for some projection $P$, the set $P (F)$ can be affinely embedded into $F$. 

Recall that $P (F)$ is the attractor of a graph directed IFS by e.g.  the results of Farkas \cite{Farkas2014projections}. If $A$ is not diagonalizable, then in fact one can manufacture via  the  procedure we have just described, a set of dimension $1$ of embeddings of $P (F)$ into $F$, and so of $P (F)$ into $P (F)$. We obtain a contradiction via approximating $P (F)$ by a self similar set, and then applying Theorem \ref{Theorem 4.2}. We thus see that $A$ must be diagonalizable. 

By a similar argument, we see that if $A$ has two eigenvalues of different norms then the norm of the larger one is algebraically dependent on the uniform contraction of $\Phi$.  .

\subsubsection{Inverse Theorems for product measures, the weak separation condition and microset  analysis}
We continue outlining  the argument proving part (2) of Theorem \ref{Theorem embeddings} from the last section. The most difficult part of the proof is showing that if both eigenvalues of $A$ have distinct norms, then  the norm of the eigenvalue  with the smallest norm is also algebraically dependent on the uniform contraction ratio of $\Phi$. For this we require a version of the inverse Theorem for entropy that works for certain product measures. We first formulate the relevant Theorem, and then proceed to explain its role in the proof. We note that to see how this Theorem follows from Hochman's work (Corollary 2.15 in \cite{hochman2015self}) is not trivial, and is shown in section \ref{Section inverse Theorems for entropy}.

Let $F$ be a self homothetic set (i.e. it is generated by $\Phi = \lbrace \phi_i \rbrace_{i=1} ^l$ such that $G_\Phi= \lbrace \id \rbrace$) with the SSC. Let $P_1 :\mathbb{R}^2 \rightarrow \mathbb{R}$ be the coordinate projection $P_1 (x,y)=x$.   For $x\in \mathbb{R}$ let $F^x$ denote the vertical slice through $F$, i.e. 
\begin{equation*}
F^x = \lbrace (x,y): (x,y)\in F \rbrace.
\end{equation*}
Let $\mu$ be the self similar measure on $F$ with respect to the probability vector $(\alpha_1 ^{\dim F},...,\alpha_l ^{\dim F})$, where $\alpha_i$ is the contraction ratio of $\phi_i$. Then  $\dim \mu = \dim F$. Let $[x] = P_1 ^{-1} (\lbrace x \rbrace)$, and let $\lbrace \mu_{[x]} \rbrace$ denote the disintegration of $\mu$ with respect to $P_1$,  so that $P_1 \mu$ almost surely $\mu_{[x]}$  is supported on the slice $F^x$. Recall that by the dimension conservation formula for self homothetic sets with the SSC, proved  in \cite{furstenberg2008ergodic},  $P_1 \mu$ almost every $\mu_{[x]}$ is exact dimensional and
\begin{equation} \label{Equation DC 1}
\dim P_1 \mu + \dim \mu_{[x]} = \dim \mu, \text{ for } P_1 \mu \text{ almost every } x \in P_1 (F).
\end{equation}
Let
\begin{equation} \label{Equation for DC(X,P)}
\text{DC}(F,P_1)= \lbrace x\in P_1 (F): x \text{ satisfies equation } \eqref{Equation DC 1} \rbrace
\end{equation}
Define the family of product measures
\begin{equation} \label{Equation for Pro(F)}
\text{Pro}(F) = \lbrace P_1 \mu \times P_2 (\mu_{[x]}|_S) : x\in \text{DC}(F,P_1), \quad S \subseteq F^x \text{ is such that } \mu_{[x]} (S) >0 \rbrace,
\end{equation}
where $P_2 :\mathbb{R}^2 \rightarrow \mathbb{R}^2$ is the map $P_2 (x,y)= y$.

Let $G_2$ denote the space of invertible affine maps $\mathbb{R}^2 \rightarrow \mathbb{R}^2$, that is,
\begin{equation} \label{equation for G2}
G_2 = \lbrace g:\mathbb{R}^2 \rightarrow \mathbb{R}^2 : \quad g(z)=A\cdot z +t, \quad  A\in GL(\mathbb{R}^2), \quad  t\in \mathbb{R}^2 \rbrace.
\end{equation}
This set can be identified with an open subset of $\mathbb{R}^6$. For any $\theta \in P(G_2)$ and $\mu \in P(\mathbb{R}^2)$ that are compactly supported define the convolution measure $\theta . \mu \in P(\mathbb{R})$ as the push-forward of the measure $\theta \times \mu$ on $G_2 \times \mathbb{R}^2$ via the map $f:G_2 \times \mathbb{R}^2 \rightarrow \mathbb{R}^2$ defined by $g(\phi,x) = \phi (x)$.

\begin{theorem} \label{Theorem Inverse Theorem for product measures - finite group}
Let $X\subset G_2$ be a bounded set, and let $F\subseteq \mathbb{R}^2$ be a self similar set generated by an IFS $\Phi$ with a uniform contraction ratio, $|G_\Phi|< \infty$ and the SSC. Assume $\gamma\cdot (P_1 (F),0)+t \subseteq F$ for some $\gamma > 0, t\in \mathbb{R}^2$. Then there exist $\delta(F)>0$ such that:

Let $F' \subseteq F$ be a self homothetic set. Let $\theta \in \text{Pro}(F')$ be such that it is not supported on any affine line, and let $\nu \in P(X)$ be such that $\overline{\dim}_e (\nu) \geq 1$. Then 
\begin{equation*}
\overline{\dim}_e (\nu. \theta) \geq \overline{\dim}_e (\theta) +\delta.
\end{equation*} 
\end{theorem}

Let $g$ be  a self embedding of $F$  with linear part $A$, as in the previous section. Denote by $\gamma_1,\gamma_2$ the eigenvalues of $A$, and suppose that $|\gamma_1| > |\gamma_2|$. Denoting the uniform contraction ratio of $\Phi$ by $\lambda$, we can use Theorem \ref{Theorem 4.2} as explained before to show that $\frac{\log |\gamma_1|}{\lambda}\in \mathbb{Q}$. We aim at showing that this is also true for $|\gamma_2|$. Note that by using the iteration and rescaling procedure outlined in the previous section,  we see that there is an affine projection mapping $F$ into itself. Without the loss of generality, assume this projection is $P_1$. We now restrict our discussion to self homothetic sets, so that now projections of $F$ are themselves self similar sets;   the finite group case follows from the self homothetic case by an approximation argument. 

Let $x\in DC(F,P_1)$ (recall its definition from \eqref{Equation for DC(X,P)}). The general strategy of the proof is to show that $F$ must contain $Y$, an affine image of a product set of the form $P_1 (F) \times P_2 (S)$, where $S \subseteq F^x$. This is done by "blowing up" balls in $g^n (F)$ of radius $|\gamma_2|^{-n}$ around a point in this slice $F^x$, and analysing their Hausdorff-metric limits as $n\rightarrow \infty$ (these limits  are called microsets). If $\frac{\log |\gamma_2|}{\log \lambda} \notin \mathbb{Q}$ we use a similar argument to find that $\dim \mathcal{E}(Y,F) \geq 1$ (recall the definition of $\mathcal{E}(F,E)$ from \eqref{Equation E(F,E)}).

To get a contradiction, we argue that we can find such a set $S \subseteq F^x$ that satisfies $\mu_{[x]} (S)>0$. If this can be done, we obtain our contradiction by applying Theorem \ref{Theorem Inverse Theorem for product measures - finite group}. The key observation for the construction of such a set  $S$ is that the self similar set $P_1 (F)$ has the weak separation condition. This separation condition, originally introduced by Lau and Ngai in \cite{Lau1999weak} as a weaker version of the open set condition, allows us to find a nice finite partition of the slice $F^x$. The set $S$ can be chosen as one of the sets in this partition.

To see why $P_1 (F)$ has the Weak separation condition, we make use of the Fraser-Henderson-Olson-Robinson dichotomy \cite{Fraser2015Assouad} for real self similar sets. This dichotomy roughly states that if the Assouad dimension of $P_1 (F)$ is smaller than $1$ then $P_1 (F)$ has the weak separation condition. In order to bound the Assouad dimension of $P_1 (F)$, we find a uniform bound on the  Hausdorff dimension of any microset of $P_1 (F)$. This is may be done  because $P_1 (F)$ embeds into a  slice of $F$, and since the original IFS $\Phi$ has the strong separation condition (see Corollary \ref{Corollary SSC implies dimension gap for slices}).

\subsubsection{Two dimensional inverse Theorem for entropy} \label{Section two dimensional inv Thm}

Recall that  a $1$-slicing of a self similar measure  was defined in the discussion preceding Theorem \ref{Theorem saturation for self similar measures}.

\begin{theorem} \cite{hochman2015self}  \label{Theorem Inverse Theorem main application}
Let $X\subseteq G_2$ be a bounded set, and let $\nu \in P(X)$ and $\mu \in P(\mathbb{R}^2)$ be compactly supported measures, and suppose $\overline{\dim}_e (\nu) \geq 1$. Suppose that $\mu$ is a self similar measure such that $0<\dim \mu <2$ and that $\mu$  does not admit a $1$-slicing.
Then there exists some $\epsilon (\mu,X)= \epsilon>0$ such that for every $\epsilon' \leq \epsilon$ there is a  $\delta(\epsilon')=\delta>0$ such that 
\begin{equation*}
\overline{\dim_e} (\nu) > \epsilon' \Rightarrow \overline{\dim_e} (\nu.\mu) \geq \overline{\dim_e} (\mu) + \delta
\end{equation*}
\end{theorem}

This Theorem follows from Corollary 2.15 in \cite{hochman2015self}.  It is used, for example, to prove part (1) of Theorem \ref{Theorem embeddings}: Denoting the self similar set $F$ and the self embedding $g$, we want to show that $A$, the linear part of $g$, is a similarity matrix. Recall that we are assuming that $\Phi$ has the SSC. By applying $g$ to cylinder sets of $F$ and then rescaling by an appropriate cylinder map, we obtain a sequence of  self-embeddings of $F$ of the form $\phi_I ^{-1} \circ g \circ \phi_J$ for some multi-indices $I,J$. If $A$ is not a similarity matrix, one uses the assumption that $G_\Phi$ is infinite (so it contains an irrational rotation), and the fact that $A(S^1)$ is an ellipse that is not a circle, to show that the set of accumulation points of this sequence  has Hausdorff dimension $\geq 1$. We then apply Theorem \ref{Theorem Inverse Theorem main application} to obtain a contradiction.

\textbf{Organization} In Section \ref{Preliminaries} we recall some basic relevant  definitions  and discuss the "restrict-map-rescale" heuristic that features prominently in this paper. Section \ref{Section proof of Theorem sauration} is devoted to the proof of Proposition \ref{Theorem saturation for self similar measures}. In Section \ref{Proof of self embeddings} we prove Theorem \ref{Theorem zero dimension of self embeddings}, and use it to prove  part (1) of Theorem \ref{Theorem embeddings}. In section \ref{Section proof of affine embeddings} we prove Theorem \ref{Theorem one into one}. Part (2) of Theorem \ref{Theorem embeddings} is proved in the subsequent section with some preliminaries regarding the weak separation condition. In section \ref{Section inverse Theorems for entropy} we discuss our versions of the inverse Theorem for entropy in two dimensions.

\textbf{Acknowledgements} This paper is part of the author's research towards a PhD dissertation, conducted at the Hebrew University of Jerusalem. I would like to sincerely thank
my adviser, Michael Hochman, for his continuous encouragement and support, and for many
useful comments and suggestions. I would also like to thank Ariel Rapaport and Shai Evra for many helpful discussions, and the anonymous referee for some useful suggestions (in particular, for pointing out the relation between the present work and \cite{bonk2013quasi}). 

\section{Preliminaries} \label{Preliminaries}
\subsection{Affine maps, similarities, and self similar sets} \label{Section self similar sets}
Let $SO(\mathbb{R}^2)$ denote the group of rotations in the plane, that is
\begin{equation*}
O \in SO(\mathbb{R}^2) \iff O\cdot O^T = \id \text{ and } \det(O)=1.
\end{equation*}
Then $SO(\mathbb{R}^2)$ is a subgroup of $O(\mathbb{R}^2)$, the orthogonal group of $\mathbb{R}^2$. The orthogonal group is in turn a subgroup of $GL(\mathbb{R}^2)$, the group of matrices with $\det \neq 0$.

Recall that we defined the sets of invertible affine maps $\mathbb{R}^i \rightarrow \mathbb{R}^i, i=1,2$ in \eqref{equation for G1}, \eqref{equation for G2}, and denoted them $G_1$ and $G_2$ respectively. Recall that $\phi \in G_2$ is called a similarity if $\phi(z) = \alpha\cdot O \cdot z +t, \alpha\neq 0, O\in O(\mathbb{R}^2), t\in \mathbb{R}^2$. We also recall that we may identify $G_i$  with an open subset of $\mathbb{R}^{i^2+i}$. 

Note that, with this identification, the transformation taking an affine map $\sigma\in G_2$ to its linear part is continuously differentiable (in fact, it is $C^\infty$).  Thus, if $\Sigma \subseteq G_2$ and $\Sigma'$ is the set of linear parts of the maps in $\Sigma$, one has $\dim_H \Sigma' \leq \dim_H \Sigma$. 

Also, when we work with $GL(\mathbb{R}^2)$ we shall usually use the metric induced by the operator norm (which is induced by $||\cdot ||_2$ on $\mathbb{R}^2$). This metric is equivalent to the metric induced on $GL(\mathbb{R}^2)$ via its identification with an open set in $\mathbb{R}^4$ with  the $||\cdot ||_2$ norm. Thus,  if a map defined on $GL(\mathbb{R}^2)$ is Lipschitz continuous in the operator norm then it is also Lipschitz continuous in the latter norm.

Recall that a set $F \subset \mathbb{R}^i$ is called self similar if there exists a generating IFS $\Phi$ for $F$ such that $\Phi$ consists only of similarities. Recall that we  denoted by $G_\Phi$ the group generated by the orthogonal parts of the maps in $\Phi$, so that $G_\Phi \leq O(\mathbb{R}^2)$. We remind the reader of our standing assumptions: unless stated otherwise, all self similar sets $F$ in this paper are assumed to be in $\mathbb{R}^2$. We shall always assume that $F$ does not lie on a $1$-dimensional affine line, and that $0<\dim F <2$. 

Note that $|G_\Phi|=\infty$ if and only if $G_\Phi$ contains an irrational rotation, that is, an element $O\in SO(\mathbb{R}^2)$ that has infinite order.  One also sees that for this to happen we must have either that some map in $\Phi$ has an irrational rotation as its orthogonal part, or there are two maps in $\Phi$ with linear parts being reflections such that their multiplication (or composition) yields an irrational rotation. Thus, either some map in $\Phi$ has an irrational rotation as its orthogonal part, or some map in the iterated IFS 
\begin{equation*}
\Phi^2 = \lbrace \phi_i \circ \phi_j : 1\leq i,j \leq l \rbrace
\end{equation*}
which is also a generating IFS for $F$,  has an irrational rotation as the orthogonal part of one of its maps.

For an IFS $\Phi = \lbrace \phi_i \rbrace_{i=1} ^l$ and its attractor $F$, a cylinder set is a set of the form $\phi_{i_1} \circ ... \circ \phi_{i_k} (F)$, where $\phi_i \in \Phi$ for all $i$ and $k\in \mathbb{N}$. Writing $I= (i_1,...,i_k)\in \lbrace 0,...,l \rbrace^k$, we use the notation $\phi_{i_1} \circ ... \circ \phi_{i_k} = \phi_I$. This composition of maps from the IFS shall be called a cylinder map. Thus, cylinder sets have the form $\phi_I (F), I\in \lbrace 0,...,l \rbrace^*$, and $|I|$ denotes the length of the word $I$.

\subsection{Hausdorff and entropy dimension of measures, and measure disintegration} \label{Section entropy}
Recall that $P(X)$ denotes the space of Borel probability measures supported on a Borel set $X\subseteq  \mathbb{R}^d$. We first recall the definition of entropy dimension of a measure.    Let 
\begin{equation} \label{Equastion dyadic}
D_n = \lbrace [\frac{k}{2^n}, \frac{k+1}{2^n}) \rbrace_{k\in \mathbb{Z}}
\end{equation}
denote the level $n$ dyadic partition of $\mathbb{R}$, and let
\begin{equation*}
D_n ^2 = \lbrace I_1 \times I_2 : I_i \in D_n \rbrace 
\end{equation*}
denote the level $n$ dyadic partition of $\mathbb{R}^2$. Let 
\begin{equation} \label{equation for entropy}
H(\theta, \mathcal{E}) =  -\sum_{E\in \mathcal{E}} \theta (E) \log \theta (E)
\end{equation}
denote the Shannon entropy of a probability measure $\theta \in P(X)$ with respect to a partition $\mathcal{E}$ of $X$.  Then the entropy dimension of $\theta \in P(\mathbb{R}^i),i=1,2,$ is defined as
\begin{equation*}
\dim_e \theta = \lim_{n\rightarrow \infty} \frac{1}{n} H(\theta, D_n ^i).
\end{equation*}
If the above limit does not exist, we define the upper entropy dimension $\overline{\dim}_e \theta$ by taking $\limsup$. Note  that if $\theta \in P(\mathbb{R}^i)$ is supported on a set $Y$ then
\begin{equation} \label{Equation dimension}
\overline{\dim}_B Y \geq \overline{\dim}_e \theta \geq \underline{\dim}_H \theta 
\end{equation}
where $\overline{\dim}_B Y$ is the upper box dimension of $Y$, and
\begin{equation*}
\underline{\dim}_H \theta  =  \inf \lbrace \dim A : A \text{ is Borel }, \theta (A)>0 \rbrace 
\end{equation*}

Note that if $F$ is the attractor of an IFS that satisfies the OSC then there exists an explicit self similar measure $\mu$ supported on $F$ such that  $\underline{\dim}_H \mu = \dim_H F$ (so it has maximal dimension). We shall call this measure the natural self similar measure on $F$. See \cite{falconer1986geometry},\cite{mattila1999geometry}, \cite{bishop2013fractal} for more details. 

Also, let $\mu \in P([0,1]^2)$ be a probability measure. Then $\mu$ admits a "Fubini" type decomposition
\begin{equation*}
\mu = \int_{x\in [0,1]} \delta_x \times \theta_x d P_1 \mu
\end{equation*}
where $\theta_x \in P([0,1])$ is defined almost surely with respect to $P_1 \mu$ (see e.g. Chapter 5 in \cite{einsiedler2011ergodic}). We shall denote the measures $\delta_x \times \theta_x$ by $\mu_{[x]}$, and refer to them as the disintegration, or conditional measures, of $\mu$ with respect to the projection $P_1$. Note that these measures are supported ($P_1 \mu$ almost surely) on vertical slices of the support of $\mu$. If we assume furthermore that $\mu$ is a self similar measure with respect to an IFS $\Phi$ with the SSC and $G_\Phi = \lbrace \id \rbrace$, then the conditional measures $\mu_{[x]}$ are almost surely exact dimensional, and we have the dimension conservation formula
\begin{equation*}
\dim \mu_{[x]} + \dim P_1 \mu = \dim \mu \quad \text{ for } P_1 \mu \text{ a.e. } x.
\end{equation*}
This was proved by Furstenberg in \cite{furstenberg2008ergodic}. In fact, this formula remains true for every projection (not only $P_1$). 

We end this subsection with by recalling the definition of our convolution measures. Let $v\in P(G_i)$ and $\mu \in P(\mathbb{R}^i)$ be compactly supported probability measures. Then $\nu.\mu \in P(\mathbb{R}^i)$ is defined as the push forward of $\nu \times \mu$ via the action map $(g,x) \mapsto g (x)$, from $G_i \times \mathbb{R}^i$ to $\mathbb{R}^i$. Note that this is a smooth map defined on an open subset of $\mathbb{R}^{i^2+i} \times \mathbb{R}^i$, so $\nu.\mu$ is a Borel probability measure on $\mathbb{R}^i$. 

\subsection{Graph directed self similar sets and projections of self similar sets} 
Let $G(\mathbf{V}, \mathbf{E})$ be a directed graph, where $\mathbf{V} = \lbrace 1, 2, . . . , q \rbrace$ is the set of vertices and $\mathbf{E}$ is the finite set of directed edges,  such that for each $i \in  \mathbf{V}$ there exists at least one $e \in  \mathbf{E}$ starting from i. Let $\mathbf{E}_{i,j}$ denote the set of edges from vertex $i$ to vertex $j$ and $\mathbf{E}^k _{i,j}$ denote the set of sequences of $k$ edges $(e_1, . . . , e_k)$ which form a directed path from vertex $i$ to vertex $j$. A graph directed iterated function system (GD-IFS) in $\mathbb{R}^d$ is a finite collection of maps $\lbrace g_e : e \in \mathbf{E} \rbrace$ from $\mathbb{R}^d$ to $\mathbb{R}^d$ such that every $g_e$ is a contracting similarity. The attractor of the GD-IFS is the unique q-tuple of nonempty compact sets $(K_1, . . . , K_q)$ such that
\begin{equation*}
K_i = \bigcup_{j=1} ^q \bigcup_{e \in \mathbf{E}_{i,j}} g_e (K_j).
\end{equation*}
The attractor of a GD-IFS is called a graph directed attractor, or graph directed set.

The directed graph $G (\mathbf{V}, \mathbf{E})$ is called strongly connected if for every pair of vertices $i$ and $j$ there exists a directed path from $i$ to $j$ and a directed path from $j$ to $i$. We say that the GD-IFS $\lbrace g_e : e\in \mathbf{E} \rbrace$ is strongly connected if $G (\mathbf{V}, \mathbf{E})$ is strongly connected. Thus, the attractor of a strongly connected GD-IFS satisfies $\dim_H K_i = \dim_H K_j, 1\leq i,j\leq q$. By the implicit methods\footnote{Note that Corollary 3.5 in \cite{falconer1997techniques} is stated for GD-IFS's with some separation conditions. However, for the proof of the equality $\dim_H K_i = \dim_B K_i$ and for the proof that the Hausdorff measure in the dimension is finite (which rely on Theorem 3.2 in \cite{falconer1997techniques}), no separation is required for the GD-IFS.} of Falconer (Corollary 3.5 in \cite{falconer1997techniques}), for each graph directed set such that $G(\mathbf{V},\mathbf{E})$  is strongly connected, the Hausdorff measure of $K =\bigcup_{i=1} ^q K_i$ in its dimension is  finite, and also  $\dim_H K = \dim_B K = \dim_B K_i = \dim_H K_i$ for every $i$.

In our treatment of Theorem \ref{Theorem embeddings}, we shall require the following results of {\'A}bel Farkas, regarding linear images of self similar sets in the plane, and dimension approximation of GD-IFS's. We group both statements in the following Theorem. We remark that we do not state the most general form of this Theorem, only what we require.

\begin{theorem}  \cite{Farkas2014projections}, \cite{farkas2015dimension} \label{Theorem farkas}
Let $F\subseteq \mathbb{R}^2$ be a self similar set generated by  an IFS $\Phi$ such that $|G_\Phi| < \infty$. Let  $L : \mathbb{R}^2 \rightarrow \mathbb{R}$ be a linear map. Let $G_\Phi = \lbrace O_1,...,O_l,O_{l+1},..,O_q \rbrace$ where $\lbrace O_1,..,O_l \rbrace$ is the set of linear parts of the maps in $\Phi$. Then the following two statements are valid:
\begin{enumerate}
\item   There exists a strongly connected GD-IFS  with attractor $(L \circ O_1(F), . . . , L \circ O_q(F))$.

\item For every $\epsilon>0$ there is some $K \subseteq L \circ O_1 (F)$ such that $ \dim L \circ O_1 (F) - \dim K < \epsilon$, and $K$ is the attractor of a self similar IFS with the SSC.
\end{enumerate}
In particular, part (1) implies that
\begin{equation*}
L(F) = \bigcup_{i=1} ^l L(\alpha_i O_i (F) +t_i), \text{ where } \Phi = \lbrace z\mapsto \alpha_i O_i (z) +t_i \rbrace_{i=1} ^l
\end{equation*}
has equal Hausdorff and Box dimension, and that $L(F)$ has finite Hausdorff measure in its dimension.
\end{theorem}

\subsection{The restrict - map - rescale heuristic} \label{Section restrict map resacles}
Suppose we have an affine embedding $g$ of a self similar set $F$ into a self similar set $E$, where the generating IFS of $E$ is assumed to have the SSC. As we mentioned in the introduction, there is a way to produce more embeddings of $F$ into $E$, by making use of the self similar structures of both $F$ and $E$. Namely, we can do this  by applying $g$  to cylinder sets of $F$ of large generation, and then "rescaling" by pre-composing with the inverse of a cylinder map of $E$, that corresponds to a cylinder set that contains this image set. We thus obtain a new embedding of $F$ into $E$ of the form $\psi_J ^{-1} \circ g \circ \phi_I$ where $\phi_I$ is a cylinder map of $\Phi$ and $\Psi_J$ is a cylinder map of $\Psi$.

The next two Lemmas are key in order to formally execute the above heuristic. First, Let us recall the following Lemma from \cite{feng2014affine}. Let $\Psi = \lbrace \psi_i \rbrace_{i=1} ^m$ be an IFS generating the set $E$, and denote its contraction ratios by $\beta_1,...,\beta_m$. Assume without the loss of generality that $\beta_1 \leq \beta_i$ for all $i$. For any $0<r<\beta_1$ define the family of $r$-cylinders by
\begin{equation} \label{equation for Ar}
\Psi_r = \lbrace \psi_I, I = (i_1,...,i_n) \in \lbrace 1,...,m\rbrace^*: \quad \beta_{i_1}\cdot \cdot \cdot \beta_{i_n} \leq r < \beta_{i_1} \cdot \cdot \cdot \beta_{i_{n-1}} \rbrace
\end{equation}
Then 
\begin{equation}
\bigcup_{\psi_I \in \Psi_r } \psi_I (E) = E. \label{Equation stopping time cover}
\end{equation}
For two non-empty sets $A,B \subseteq \mathbb{R}^2$ we define
\begin{equation*}
\text{dist} (A,B) = \inf \lbrace || a-b||: a\in A, b\in B \rbrace
\end{equation*}

\begin{Lemma} \cite{feng2014affine} \label{Lemma N0}
Suppose $\Psi$ satisfies the OSC. Then there exists $N_0 \in \mathbb{N}$ such that for any $0<r<\beta_1$ and $\psi_I \in \Psi_r$,
\begin{equation*}
| \lbrace \psi_J \in \Psi_r : \text{ dist } (\psi_I (E),\psi_J (E)) \leq r \rbrace| \leq N_0
\end{equation*}
\end{Lemma}

For an affine map $g: \mathbb{R}^{d_1} \rightarrow \mathbb{R}^{d_2}$ let $||g||$ denote the operator norm  of the linear part of $g$. 
\begin{Lemma} \label{Lemma Restrict map rescale}
Let  $F$ and $E$ be two self similar sets, generated by IFS's $\Phi = \lbrace \phi_i \rbrace_{i=1} ^l$ and $\Psi = \lbrace \psi_j  \rbrace_{j=1} ^m$, respectively. Denote, for $1\leq i \leq l$ and $1 \leq j \leq m$, the contraction ratio of $\phi_i$ by $\alpha_i \in (0,1)$ and of $\psi_j$ by $\beta_j \in (0,1)$. Suppose that $g$ is an affine embedding of $F$ into $E$. Then the following two statement are valid:

\begin{enumerate}
\item Suppose $\Psi$ has the OSC. Let $I \in \lbrace 1,..., l \rbrace^*$ . Define 
\begin{equation*}
r = ||g|| \cdot \alpha_I \cdot \text{diam} (F),
\end{equation*}
then
\begin{equation*}
| \lbrace \psi_J \in \Psi_{r} : g( \phi_I (F) ) \cap \psi_J (E) \neq \emptyset \rbrace | \leq N_0
\end{equation*}
where $\Psi_r$ was defined in \eqref{equation for Ar} and $N_0$ is the number from Lemma \ref{Lemma N0}.

\item Suppose  $\Psi$ has the SSC, and that it has a uniform contraction ratio $\lambda = \beta_1 =....=\beta_m$. Let 
\begin{equation*}
\rho = \min_{i\neq j} \text{dist} (\psi_i (E), \psi_j (E) ) >0.
\end{equation*} 
Let $n\in \mathbb{N}$. Let $I \in \lbrace 1,...,l \rbrace^*$ be such that $||g|| \cdot \alpha_I \cdot \text{diam} (F)< \rho \lambda^{n-1}$. Then $g(\phi_I (F))$ is contained within a unique $n$-generational cylinder $\psi_J (E), |J|=n$.
\end{enumerate}
\end{Lemma}

\begin{proof}
We have $g( \phi_I (F)) \subseteq g(F) \subseteq E$, and by equation \eqref{Equation stopping time cover}, there is some $\psi_J \in \Psi_{r}$ such that $\psi_J (F) \cap g( \phi_I (F)) \neq  \emptyset$. Therefore, 
\begin{equation} \label{Equation left hand side 1}
\lbrace \psi_U \in \Psi_{r} : \quad g( \phi_I (F)) \cap \psi_U (E) \neq \emptyset \rbrace \subseteq \lbrace \psi_U \in \Psi_{r}: \quad d(\psi_J (E),\psi_U (E))\leq r \rbrace
\end{equation}
This follows since if $\psi_U$ belongs to the set on the left hand side of equation \eqref{Equation left hand side 1} then by the definition of $r$
\begin{equation*}
d(\psi_U (E), \psi_J (E)) \leq \text{diam } g(\phi_I (F)) \leq ||g|| \cdot \alpha_I \cdot \text{diam} (F) = r.
\end{equation*}
Therefore, by Lemma \ref{Lemma N0} the set on the left hand side of equation \eqref{Equation left hand side 1} contains at most $N_0$ maps. This proves the first part of the Lemma. 

The second part follows by noting that by self similarity, the definition of $\rho$, and the fact that $\Psi$ has a uniform contraction ratio,
\begin{equation} \label{Equation contradiction}
\min_{U\neq J \in \lbrace 1,...,m\rbrace^n} \text{dist} (\psi_U (E), \psi_J (E) ) = \rho \cdot \lambda^{n-1}.
\end{equation}
Since $g( \phi_I (F)) \subseteq E$, it follows that $g( \phi_I (F) ) \cap \psi_J (E) \neq \emptyset$ for some $J\in \lbrace 1,..,m\rbrace^n$. If $g( \phi_I (F) ) \cap \psi_U (E) \neq \emptyset$ for some other $U \neq J\in \lbrace 1,..,m\rbrace^n$, then
\begin{equation*}
d(\psi_U (E), \psi_J (E)) \leq \text{diam } g(\phi_I (F)) \leq ||g|| \cdot \alpha_I \cdot \text{diam} (F) < \rho \lambda^{n-1} 
\end{equation*}
Contradicting equation \eqref{Equation contradiction}. It follows that $\psi_J (E)$ is the unique $n$-generational cylinder that intersects $g( \phi_I (F))$, and since $g( \phi_I (F)) \subseteq E$, we must have $g( \phi_I (F)) \subseteq \psi_J (E)$. 
\end{proof}

\subsection{The Hausdorff metric}
We recall the definition of the Hausdorff metric on compact subset of 
\begin{equation*}
Q=[-1,1]^2 .
\end{equation*}
Let $A,B \subseteq Q$ be non-empty and compact. For every $\epsilon >0$ define, using the Euclidean norm $|| \cdot ||_2$,
\begin{equation*}
A_\epsilon = \lbrace x \in Q : \exists a\in A, || x-a|| < \epsilon \rbrace.
\end{equation*}
The Hausdorff metric is then defined as the distance function
\begin{equation*}
d_H (A,B) = \inf \lbrace \epsilon >0 : A \subseteq B_\epsilon, B \subseteq A_\epsilon \rbrace. 
\end{equation*}
It is well known that the space  $\cmpct(Q)$ of  non-empty compact subsets of $Q$ is compact with this metric (see e.g. the appendix in \cite{bishop2013fractal}).

We proceed to collect a number of facts about the Hausdorff metric on compact subsets of $Q$. The following Proposition is standard, so we omit its proof. Recall that $P_2 (x,y) = y$.

\begin{Proposition} \label{Proposition - Hausdorff metric0}
In the following claims, convergence of sets is always in the Hausdorff metric on compact subsets of $Q$, and converges of points of $Q$ is in the standard Euclidean norm. 
\begin{enumerate}
\item Suppose $X_n \subseteq Q$ is a sequence of compact sets that converges  to $X \subseteq Q$. Then
\begin{equation*}
X = \lbrace x\in Q : \exists x_{n_k} \in X_{n_k} , \lim_{k} x_{n_k} = x\rbrace.
\end{equation*}
In particular, if a set $B \subseteq Q$ satisfies that for every $b \in B$ there is a sequence $x_{n_k} \in X_{n_k}$ such that $\lim x_{n_k} = b$ then $B \subseteq X$. On the other hand, if $C \subseteq Q$ is a compact set such that every converging subsequence $x_{n_k} \in X_{n_k}$ holds $\lim_{k} x_{n_k} \in C$, then $X \subseteq C$.

\item Suppose $X_n = \bigcup_{k=1} ^l X_n ^k$ where $l \in \mathbb{N}$ is fixed, $n\rightarrow \infty$, and  $X_n ^k \subseteq Q$ are compact sets for all $n,k$.  Suppose that $X_n ^k \rightarrow X^k$ as $n\rightarrow \infty$ for all $1 \leq k \leq l$. Then $X_n \rightarrow \bigcup_{k=1} ^l X^k$. 

\item Suppose the sequence $X_n \subseteq Q$ converges to $X\subseteq Q$. Then $P_2 (X_n)$ converges to $P_2 (X)$.
\end{enumerate}
\end{Proposition}

\section{Proof of Proposition \ref{Theorem saturation for self similar measures}}  \label{Section proof of Theorem sauration}
We now prove Proposition \ref{Theorem saturation for self similar measures}. Part (1) follows easily by observing that the only subgroups of $O(\mathbb{R}^2)$ that admit an invariant  line $V\in G(2,1)$ are $\lbrace \pm \id \rbrace$, $\lbrace  \id, R_1\rbrace$, and $\lbrace \pm \id, R_1, R_2 \rbrace$ where $R_1$ and $R_2$ are reflections about orthogonal lines.

For the proof of Part (2), it will be convenient to introduce the notation
\begin{equation*}
AG(2,1) = \lbrace L: L = V+t , \quad  V\in G(2,1), \quad  t\in \mathbb{R}^2 \rbrace
\end{equation*}
to denote the space of all affine lines in $\mathbb{R}^2$. We shall generally use the letter $L$ to denote elements in $AG(2,1)$, as opposed to using the letter $V$ to denote elements of $G(2,1)$.

Now, part (2) follows by proving that if $F$ is a self similar set with a generating IFS having the OSC, and $\dim F\cap L=1$ for some affine line $L\in AG(2,1)$, then $F\cap L'$ contains a non-degenerate interval for some affine line $L'\in AG(2,1)$. The proof of this observation  is  based on the approximation argument proving Proposition 2.2 in \cite{feng2014affine}. In fact, we closely follow the proof of  Theorem 1.1 from \cite{feng2014affine}, but we can't use it as it is stated there since the self similar set we embed sits on a line in $\mathbb{R}^2$ (so it is really one dimensional). In this section we always assume, without the loss of generality, that $F\subseteq [0,1]^2$.

\begin{Proposition} \label{Proposition full dimension of fiber implies connected}
Let $F \subseteq [0,1]^2$ be a self similar set generating by an IFS $\Phi$ with the OSC. Suppose $\dim F\cap L =1$ for some $L \in AG(2,1)$. Then there exists some line $L' \in AG(2,1)$ such that $F\cap L'$ contains a non-degenerate interval. 

Moreover, writing $L=V+t$ and $L' = V' +t'$ for $V,V' \in G(2,1), t,t' \in \mathbb{R}^2$ we have
\begin{equation*}
V' \in \lbrace W \in G(2,1): W = \lim_{n} O_{i_1} ^{-1} \cdot \cdot \cdot O_{i_n} ^{-1} (V), \quad O_{i_k} \text{ are  linear part of  maps in } \Phi \rbrace 
\end{equation*}
\end{Proposition}

Let $n\in \mathbb{N}$, and define (as in \cite{feng2014affine})  $s_n<1$ to be the unique positive number such that
\begin{equation} \label{Equasion the numbers sn}
(2^n -1)\cdot (\frac{1}{2^n})^{s_n} =1.
\end{equation}
Also, recall that definition of $D_n$, the $n$-level dyadic partition of $\mathbb{R}$, defined in equation \eqref{Equastion dyadic}.
\begin{Claim}
Let $F$ be a self similar set generated by an IFS $\Phi$ with the OSC, and let $N_0$ be the number from Lemma \ref{Lemma N0}. Let $L \in AG(2,1)$ and $n\in \mathbb{N}$, and suppose that $\dim F\cap L >s_n$.

Then there are $0 <  k_n \leq N_0$ affine maps $g_i :[0,1] \rightarrow L_i $, $1\leq i \leq k_n$ where $L_i \in AG(2,1)$ such that:

  for every $n$-level dyadic interval $D\in D_n$, $D\subseteq [0,1]$,
\begin{equation*}
\left( \bigcup_{i=1} ^{k_n} g_i (D) \right) \bigcap F \neq \emptyset.
\end{equation*} 
Moreover, the maps $g_i$ are uniformly bounded independently of $n$.
\end{Claim}
\begin{proof}
We  consider $[0,1]$ as the attractor of the IFS $\Psi = \lbrace \psi_i \rbrace_{i=1} ^2$, where $\psi_i (z)= \frac{z+i-1}{2}$. Let $f:[0,1]\rightarrow L\cap [0,1]^2$ be an affine similarity parametrization\footnote{This map is just $f(x)=\alpha\cdot O\cdot (x,0)+t$, where $\alpha=||f||$ is the length of $L\cap [0,1]^2$ and $O\in SO(\mathbb{R}^2)$,$t\in \mathbb{R}^2$ are parameters chosen according to the the affine line $L$.} of the segment $L\cap [0,1]^2$. Let $n\in \mathbb{N}$ be such that $\dim f([0,1])\cap F = \dim L\cap F > s_n$.  

 We first claim that for every $j \in \mathbb{N}$ there is a word $W_j \in \lbrace 1,2\rbrace^*$ with $|W_j| \geq j$ such that for every $D\in D_n$, $D\subseteq [0,1]$, $f(\psi_{W_j} (D))\cap F \neq \emptyset$. For a proof, which relies on the definition of $s_n$ and on the fact that the similarity dimension of $\Psi$ is $1=\dim [0,1]$, see the beginning of the proof of Proposition 2.2 in \cite{feng2014affine}.

 Now, let $p_1 \in \mathbb{N}$ satisfy $2^{-p_1 +1} < \alpha_1$, where $\alpha_1$ is the minimal contraction ratio of any map in $\Phi$. Let $j>p_1$, and let $W_j \in \lbrace 1,2\rbrace^*$ be the word we found in the previous paragraph. Fix $r_j = 2\cdot 2^{-|W_j|}$, then since $j>p_1$, we have $r_j < \alpha_1$. Next, by Lemma \ref{Lemma Restrict map rescale} we find that 
\begin{equation} \label{Equation left hand side}
| \lbrace \phi_I \in \Phi_{r_j} : f(\psi_{W_j} ([0,1]))\cap \phi_I (F) \neq \emptyset \rbrace | \leq N_0
\end{equation}
where $\Phi_{r_j}$ was defined in \eqref{equation for Ar} and  $N_0$ in the number from Lemma \ref{Lemma N0}.  Enumerate these maps from $1$ to $k_n \leq N_0$, and denote them $h_{i,n}$, $1\leq i \leq k_n$ (Note that the $j$ we are using only depends on $n$)

Finally, by our choice of $W_j$, the $k_n$ maps $g_i = h_{i,n} ^{-1} \circ f \circ \psi_{W_j}$ for $1 \leq i \leq k_n \leq N_0$ satisfy that for every $D\in D_n, D\subseteq [0,1]$,
\begin{equation*}
\left( \bigcup_{i=1} ^{k_n} g_{i} (D) \right) \bigcap F \neq \emptyset.
\end{equation*}
Moreover, since for every $i$
\begin{equation*}
||h_{i,n} || \leq r_j , \quad  r_j \leq ||h_{i,n} || \cdot \alpha_1 ^{-1}
\end{equation*}
then   we have (recalling that $h_{i,n}$ and $f$ are similarities))
\begin{equation*}
\frac{||f|||}{2} = r_j ^{-1} \cdot ||f|| \cdot  2^{-|W_j|} \leq || g_{i} ||\leq 2^{-|W_j|} \cdot ||f||\cdot  r_j ^{-1} \alpha_1 ^{-1} = \frac{||f||}{2\cdot \alpha_1},
\end{equation*}
which is uniformly bounded (note that $||f||$, the similarity ratio of $f$, is just the length of $L\cap [0,1]^2$). It also follows that the translation part of each one of the maps $g_{i}$ is uniformly bounded, since $g_{i} ([0,1])\cap F \neq \emptyset$. Finally, it is clear that the maps $g_{i}$ take $[0,1]$ into an affine line in the plane, since $F$ is self similar. 
\end{proof}  

$$ $$
\textbf{Proof of Proposition \ref{Proposition full dimension of fiber implies connected}} By our assumption, $\dim F\cap L =1, F\subseteq [0,1]^2$. Let $f:[0,1]\rightarrow [0,1]^2\cap L$ be an  affine similarity parametrization as in the previous proof, so that $f([0,1])=L\cap [0,1]^2$. Then for every $n\in \mathbb{N}$ we have $\dim f([0,1])\cap F >s_n$. Thus, for every $n$ let $\lbrace g_1 ^{n},...,g_{k_n} ^n \rbrace$ be the maps produced via the previous Claim.  Move to a sub sequence $\lbrace g_1 ^{n_l},...,g_{k_{n_l}} ^{n_l} \rbrace$ such that $k_{n_l} \equiv k$ for some $k$. Then move to another subsequence (without changing the notation) such that $g_i ^{n_l} \rightarrow g_i$ for every $1\leq i \leq k$, and $g_i : [0,1] \rightarrow L_i \cap B(0,\beta)$ are affine maps, where $L_i \in AG(2,1)$, and $\beta >0$. Here we are using the uniform bounds we established for both the translation and the contraction part of the $g_i ^n$'s, and the fact that each $g_i$ has its image in some affine line.

We now claim that $[0,1] \subseteq \bigcup_{i=1} ^k g_i ^{-1} (F)$. Indeed, let $x\in [0,1]$, and for every $n$ let $D_n (x) \in D_n$ be the dyadic interval containing $x$. Then $ \lbrace x \rbrace  = \bigcap_{l=1} ^\infty D_{n_l}(x)$. In addition, for every $n_l$ there is a map $g^{n_l} _{i_l}$ such that $d(g^{n_l} _{i_{l}} (x), F) \leq 2^{-n_l}\cdot c$, for some uniform constant $c>0$. Indeed, by the previous Claim, we may pick  $g^{n_l} _{i_l}$ as the map such that $g^{n_l} _{i_l} (D_{n_l} (x)) \cap F \neq \emptyset$. Moving to another sub sequence such that $i_{n_l} \equiv i$ is fixed, we see that both $g_i ^{n_l} (x) \rightarrow g_i (x)$, and that $d(g_i (x),F)=0$. Since $F$ is closed it follows that $g_i (x)\in F$. The claim is proved.

Finally, from Baire's Theorem it follows that there are $0\leq a < b \leq 1$ such that $(a,b) \subseteq g_i ^{-1} (F)$ for some $i$. Since $g_i$ is a map taking $[0,1]$ to an affine line, this finishes the proof. \hfill{$\Box$}

$$ $$
Part  (2) of Proposition \ref{Theorem saturation for self similar measures} is now an immediate consequence of the following Corollary.
\begin{Corollary} \label{Corollary SSC implies dimension gap for slices}
Let $F \subseteq \mathbb{R}^2$ be a self similar set generated by an IFS $\Phi$ with the SSC. Then for every line $L \in AG(2,1)$ we have $\dim F\cap L < 1$.
\end{Corollary}
$$ $$

\textbf{Remark} Note that we can in fact deduce a stronger result. Let $F \subseteq \mathbb{R}^2$ be a self similar set generating by an IFS $\Phi$ with the SSC, and let $s_n$ be the sequence of numbers from equation \eqref{Equasion the numbers sn}. Then, as a consequence of the proof of Proposition \ref{Proposition full dimension of fiber implies connected}  there exits some $m\in \mathbb{N}$ such that
\begin{equation*}
\sup_{V\in G(2,1), t\in \mathbb{R}^2} \dim (F\cap (V+t) ) < s_m<1
\end{equation*}

\section{Proof of Theorem \ref{Theorem zero dimension of self embeddings} and  of part (1) of Theorem \ref{Theorem embeddings} } \label{Proof of self embeddings}
\textbf{Proof of Theorem \ref{Theorem zero dimension of self embeddings}} Assume the contrary is true. Then $\dim_H \mathcal{E}(F)>0$. Apply Frostman's Lemma to obtain a compactly supported measure $\nu \in P(\mathcal{E}(F))$ such that $\overline{\dim_e} \nu \geq  \underline{\dim_H} \nu > \epsilon' >0$.  By our assumptions there exists a self similar measure $\mu \in P(F)$ such that $\underline{\dim_H} \mu = \dim_H F$, and such that $\mu$ does not admit a $1$-slicing. Denote $X = \text{supp}(\nu)$ and let $\epsilon_1 = \min \lbrace \epsilon', \epsilon(\mu,X) \rbrace$, where $\epsilon(\mu,X)>0$ is the number from the statement of Theorem \ref{Theorem Inverse Theorem main application}.  Therefore, we may apply Theorem \ref{Theorem Inverse Theorem main application} for $\epsilon_1$ to obtain that the measure $\nu.\mu$ satisfies that for some $\delta(\epsilon_1)=\delta>0$,
\begin{equation*}
\overline{\dim_e} \nu.\mu > \overline{\dim_e} \mu + \delta \geq  \underline{\dim_H} \mu + \delta = \dim F + \delta.
\end{equation*}
Note the use of equation \eqref{Equation dimension}.

However, $\nu.\mu$ is a measure that is supported, by definition, on $F$. Therefore,  by equation \eqref{Equation dimension} again we have
\begin{equation*}
\dim_B F \geq \overline{\dim_e} \nu.\mu >  \dim_H F + \delta.
\end{equation*}
However, $F$ is a self similar set, so $\dim_B F = \dim_H F$. Hence the above equation yields the desired contradiction. It follows that $\dim_H \mathcal {E}(F) =0$, as required. \hfill{$\Box$}

$$ $$

We now proceed to prove part (1) of Theorem \ref{Theorem embeddings}. Let us first give a general overview of the proof. Assume towards a contradiction that there exists an affine self embedding of $F$ that is not a similarity. As we are assuming $|G_\Phi|=\infty$,  we may assume one member of $\Phi$ has an irrational rotation as its orthogonal part (see Section \ref{Section self similar sets}). We now use the "restrict - map - rescale" heuristic (see Section \ref{Section restrict map resacles}), restricting to cylinders sets defined by iterating this specific cylinder map, mapping them using our non similarity self embedding into $F$, and rescaling. Since the linear part of the embedding is not an orthogonal matrix,  this procedure generates a large set (of Hausdorff dimension at least $\frac{1}{2}$) of self embeddings of $F$. Once this point is established (which requires some work), we apply Theorem \ref{Theorem zero dimension of self embeddings} and obtain a contradiction.

We will use the following Lemma in the proof. Define a map $h: GL(\mathbb{R}^2)\rightarrow \mathbb{R}$ by
\begin{equation} \label{Equation for h}
h(M) =  ||M(e_1)||^{-1} \cdot ||M||
\end{equation}
where the norm on $M$ is the operator norm (with respect to the $||\cdot ||_2$ norm), $e_1 = (1,0)\in \mathbb{R}^2$ and the norm on the vector $M(e_1)$ is the usual $||\cdot ||_2$ norm.
\begin{Lemma} \label{Lemma h is Lip}
Suppose $\Sigma \subset GL(\mathbb{R}^2)$ satisfies that for some $C,c>0$,
\begin{equation*}
 \min_{M\in \Sigma} ||M(e_1)||  > c ,\quad  \max_{M\in \Sigma} ||M|| < C.
\end{equation*}
Then $h|_\Sigma$ is locally a Lipschitz function, where $h$ was defined in equation \eqref{Equation for h}.
\end{Lemma}

\begin{proof}
First, we note that the function $s(M)= ||M(e_1)||^{-1}$ is continuously  differentiable on $\Sigma$ by the assumptions $\min_{M\in \Sigma} ||M(e_1)|| > c$ and $\max_{M\in \Sigma} ||M(e_1)||< C$. Therefore,  $s|_\Sigma$ is locally Lipschitz. We also note that our assumption mean that the functions $s|_\Sigma$ and $||\cdot ||$-restricted-to-$\Sigma$ are bounded. Therefore, the function $h$ is locally a Lipschitz function, as a product of two locally Lipschitz  bounded functions.
\end{proof}

\textbf{Proof of Theorem \ref{Theorem embeddings} part (1)} 

\textbf{Step 1 - assumption towards a contradiction}
Recall that we are assuming $|G_\Phi|=\infty$. So, we may assume, without the loss of generality, that $\phi_1 \in \Phi$, $\phi_1 (z)=\alpha\cdot O (z)+t'$, satisfies that $O \in SO(\mathbb{R}^2)$ is an irrational rotation, i.e $\lbrace O^n \rbrace_{n\in \mathbb{N}}$ is dense in $SO(\mathbb{R}^2)$ (see section \ref{Section self similar sets}). Write $g(z)=Az+t, A\in GL(\mathbb{R}^2)$, and suppose $g(F)\subseteq F$.

To prove the full statement of  Theorem,  it suffices to show that $A$ is a similarity, that is, that $A \in \lambda \cdot  O(\mathbb{R}^2)$ for some scalar $\lambda \neq 0$.  Indeed, if this is  the case,  we may  apply Theorem 4.9 from \cite{elekes2010self} to obtain the desired result. 

So, assume towards a contradiction that $A\notin \lambda \cdot  O(\mathbb{R}^2)$ for all $\lambda>0$. In particular,  denoting the unit sphere $S^1 \subset \mathbb{R}^2$, the ellipse $A(S^1)$ is not a circle.

\textbf{Step 2 - generating more self embeddings} Let $n\in \mathbb{N}$ and define 
\begin{equation} \label{Eq for rn}
r_n = ||A||\cdot \alpha^n \cdot \text{diam}(F),
\end{equation}
and consider the family of $r_n$-cylinders $\Phi_{r_n}$, which was defined in \eqref{equation for Ar}. By Lemma \ref{Lemma Restrict map rescale} we see that
\begin{equation} \label{Equatio LHS again}
| \lbrace \phi_I \in \Phi_{r_n} : g(\phi_{1^n} (F)) \cap \phi_I (F) \neq \emptyset \rbrace | \leq N_0.
\end{equation}
where $N_0$ is the number from Lemma \ref{Lemma N0}. So, for any $n$ we obtain that there are $1\leq k_n \leq N_0$ maps in the set in equation \eqref{Equatio LHS again}. Written explicitly, these are 
\begin{equation*}
\phi_{I_{i,n}}(\cdot ) = \alpha_{I_{i,n}} O_{I_{i,n}} (\cdot ) + t'_{I_{i,n}}, \quad 0<\alpha_{I_{i,n}}<1, \quad O_{I_{i,n}} \in G_\Phi, \quad t' _{I_{i,n}} \in \mathbb{R}^2, \quad  i=1,..,k_n
\end{equation*}
 so that 
\begin{equation*}
g(\phi_{1^n} (F)) \subseteq \bigcup_{i=1} ^{k_n} \phi_{I_{i,n}} (F)
\end{equation*}
which implies that
\begin{equation} \label{equation union}
F \subseteq \bigcup_{i=1} ^{k_n} \phi_{1^n} ^{-1} \circ g^{-1}  \circ \phi_{I_{i,n}} (F).
\end{equation}

\begin{Claim} \label{Claim maps}
For every $R \in SO(\mathbb{R}^2)$ there exists:
\begin{itemize}
\item  A number $k(R)\in \mathbb{N}$ such that $1\leq k(R) \leq N_0$.

\item  $k(R)$ matrices $O_i \in O(\mathbb{R}^2)$ for $1\leq i \leq k(R)$.

\item  $k(R)$ contractions $||A||\cdot \text{diam}(F)\cdot \alpha_{\min} \leq  \alpha_i \leq ||A|| \cdot  \text{diam}(F)$, $1\leq i \leq k(R)$ where $\alpha_{\min}$ is the minimal contraction ratio among the contraction ratios of the maps in $\Phi$,
\end{itemize}
such that
\begin{equation*}
F \subseteq \bigcup_{i=1} ^{k(R)} \left( \alpha_i \cdot R^{-1} \cdot A^{-1} \cdot O_i (F) + t_i \right), \quad \text{ for some translations } t_i \in \mathbb{R}^2, 1\leq i \leq k(R).
\end{equation*}
\end{Claim}

\begin{proof}
Let $R\in SO(\mathbb{R}^2)$, and recall that $O$ is an irrational rotation. Therefore, there is some subsequence $n_l$ such that $O^{n_l}$ converges to $R$. Move to another subsequence so that $k_{n_l}$ is constant for all large enough $l$ (without the loss of generality this already happens for $n_l$), and denote this constant $k(R)$. For every large enough  $l \in \mathbb{N}$ and $1\leq i \leq k(R)$,  note that the linear part of the affine map $\phi_{1^{n_l}} ^{-1} \circ g^{-1}  \circ \phi_{I_{i,n_l}} $ is $\alpha_{I_{i, n_l}} \cdot \alpha^{-n_l} \cdot O^{-n_l} \cdot A^{-1} \cdot O_{I_{{i,n_l}}}$. Move to another subsequence so that $O_{I_{i,n_l}}$ converges to some $O_i \in O(\mathbb{R}^2)$ for every $1\leq i \leq K(R)$ (again suppose this happens for our subsequence $n_l$). Also, for every $1\leq i \leq k(R)$ and every $l\in \mathbb{N}$ we have (by the choice of $r_n$ in equation \eqref{Eq for rn})
\begin{equation*}
||A||\cdot \text{diam}(F)\cdot \alpha_{\min} \leq  \alpha_{I_{i, n_l}} \cdot \alpha^{-n_l} \leq ||A|| \cdot  \text{diam} (F).
\end{equation*}
So by moving to another subsequence we can make sure that $\alpha_{I_{i, n_l}} \cdot \alpha^{-n_l}$ converges for every $1 \leq i \leq k(R)$. Since by definition (recall \eqref{Equatio LHS again}) $\phi_{1^{n_l}} ^{-1} \circ g^{-1}  \circ \phi_{I_{i,n_l}} (F)\cap F \neq \emptyset$ for every $1\leq i \leq k(R)$, we see that  the translation part of $\phi_{1^{n_l}} ^{-1} \circ g^{-1}  \circ \phi_{I_{i,n_l}}$ is also uniformly bounded. Thus, by moving to another subsequence we can make sure that for every $1\leq i \leq k(R)$ the affine map $\phi_{1^{n_l}} ^{-1} \circ g^{-1}  \circ \phi_{I_{i,n_l}}$ converges pointwise to an affine map. Furthermore, its linear part has the form $\alpha_i \cdot  R^{-1} \cdot  A^{-1} \cdot O_i$ for parameters $\alpha_i, R,O_i,t_i$ as in the statement of the claim. Note that we may assume the convergence is uniform on compact sets, by an application of the Arzel\'{a}-Ascoli Theorem.  Finally, since  equation \eqref{equation union} holds for every $n_l$ for these $K(R)$ maps, it also holds for their limits (using the uniform convergence on $F$ of each sequence of maps), concluding the proof of the Claim.
\end{proof}

By Claim \ref{Claim maps} and using Baire's Theorem, we obtain that for every $R \in SO(\mathbb{R}^2)$ there exist some  $O \in O(\mathbb{R}^2)$, a contraction $||A||\cdot \text{diam}(F)\cdot \alpha_{\min} \leq  \alpha \leq ||A|| \cdot  \text{diam}(F)$, some $t \in \mathbb{R}^2$, and a cylinder set $\phi_I (F)$  such that
\begin{equation} \label{Equation for gamma}
 \phi_I (F) \subseteq \alpha \cdot R^{-1} A^{-1} O (F) +t.
\end{equation}
Denote the set of maps we have thus obtained on the right hand side of equation \eqref{Equation for gamma} (i.e. the maps $\alpha \cdot R^{-1} A^{-1} O (\cdot ) +t$) by $\Gamma$. Note that we  have manufactured, for every $R\in SO(\mathbb{R}^2)$, an affine self  embedding of $F$, by moving sides in equation  \eqref{Equation for gamma}. 

\textbf{Step 3 - $\Gamma$ has dimension at least 1} Note that $\Gamma$ is contained in the set
\begin{equation} \label{Equation for Gamma}
\lbrace \alpha \cdot R^{-1} A^{-1} O (\cdot ) +t : ||A||\cdot \text{diam}(F)\cdot \alpha_{\min} \leq  \alpha \leq ||A|| \cdot  \text{diam}(F), \quad R\in SO(\mathbb{R}^2),O \in O(\mathbb{R}^2),  t\in \mathbb{R}^2 \rbrace.
\end{equation}
Let $\Sigma' = \Gamma^{-1} = \lbrace \gamma^{-1} : \gamma\in \Gamma\rbrace$. Then $\Sigma'$  is contained within the set
\begin{equation*}
\lbrace \alpha \cdot O A R (\cdot ) +t : \frac{1}{||A|| \cdot  \text{diam}(F)} \leq  \alpha \leq \frac{1}{||A||\cdot \text{diam}(F)\cdot \alpha_{\min}}, \quad R\in SO(\mathbb{R}^2),O \in O(\mathbb{R}^2),  t\in \mathbb{R}^2 \rbrace.
\end{equation*}
Denote by $\Sigma$ the set of all linear parts of the maps in $\Sigma'$. Then
\begin{equation} \label{Equation for Sigma}
\Sigma \subseteq \lbrace \alpha \cdot O A R  : \frac{1}{||A|| \cdot  \text{diam}(F)} \leq  \alpha \leq \frac{1}{||A||\cdot \text{diam}(F)\cdot \alpha_{\min}} , \quad R\in SO(\mathbb{R}^2) ,O \in O(\mathbb{R}^2) \rbrace.
\end{equation}

Therefore, for every $M\in \Sigma$ we have $M = \alpha \cdot O A R$ for $\alpha,O,R$ as in Equation \eqref{Equation for Sigma}, so
\begin{equation*}
|| M|| \leq ||A||\cdot \alpha \leq ||A|| \cdot \frac{1}{||A||\cdot \text{diam}(F)\cdot \alpha_{\min}}.
\end{equation*}
We also claim that $\min_{M\in \Sigma} ||M(e_1)|| \geq c>0$ for some $c>0$. Indeed, let $c$ denote the length of the minor semi axis of the ellipse $A(S^1) \subset \mathbb{R}^2$. Let $M\in \Sigma$. Then $M=\alpha OAR$ for  $\alpha,O,R$ as in equation  \eqref{Equation for Sigma}. Thus, $AR(e_1)\in A(S^1)$ and therefore $||AR(e_1)|| \geq c$. Since $O\in O(\mathbb{R}^2)$ we have $||OAR(e_1)|| = ||AR(e_1)||  \geq c$. Finally, 
\begin{equation*}
||M(e_1)|| = \alpha \cdot ||OAR(e_1)|| \geq \frac{1}{||A||\cdot \text{diam} (F)} \cdot ||OAR(e_1)|| \geq \frac{1}{||A||\cdot \text{diam} (F)} \cdot c >0.
\end{equation*}

We may now apply Lemma \ref{Lemma h is Lip} and conclude that the map $h:GL(\mathbb{R}^2) \rightarrow \mathbb{R}$ defined by
\begin{equation*}
h(M) = ||M(e_1)||^{-1} \cdot ||M||
\end{equation*}
is locally Lipschitz on $\Sigma$. We now use this fact to show that $\dim h(\Sigma) \geq 1$.

By the definition of $\Gamma$, for every $R\in SO(\mathbb{R}^2)$ there are  $\alpha>0, O\in O(\mathbb{R}^2)$ as in equation \eqref{Equation for Gamma} such that $\alpha R^{-1} A^{-1} O^{-1}$ is a linear part of some map in $\Gamma$. Therefore, by definition, $M = \alpha^{-1} OAR \in \Sigma$. Thus,
\begin{equation*}
h(M)= ||M(e_1)||^{-1} \cdot ||M|| = \alpha \cdot ||AR(e_1)||^{-1} \cdot \alpha^{-1} \cdot ||A|| = ||A|| \cdot ||AR(e_1)||^{-1}.
\end{equation*}
It follows that
\begin{equation} \label{Equation h sigma}
h(\Sigma) \supseteq \lbrace ||A||\cdot ||AR(e_1)||^{-1} : R\in SO(\mathbb{R}^2 \rbrace.
\end{equation}
We now show that the set on the right hand side of equation \eqref{Equation h sigma} is a connected interval. This is the only place where the assumption towards a contradiction that $A$ is not a similarity is used: by this assumption, $A(S^1)$ is an ellipse that is not a circle. In particular, the function $||\cdot ||:A(S^1)\rightarrow \mathbb{R}$ is continuous, non-constant, and strictly positive. Since for every $v\in A(S^1)$ there is some $R\in SO(\mathbb{R}^2)$ with $A(R(e_1))=v$, we see that  the set on the right hand side of equation \eqref{Equation h sigma} equals the non-degenerate interval $||A||\cdot \lbrace ||v||^{-1} :v\in A(S^1)\rbrace$. Thus, $h(\Sigma)$ contains an interval and therefore $\dim h(\Sigma) \geq 1$.

Finally, since $h$ is locally Lipschitz we see that $\dim \Sigma \geq 1$. Since $\Sigma$ is the set of linear parts of the maps in $\Sigma'$, we have $\dim \Sigma \geq \dim \Sigma' \geq 1$ (see section \ref{Section self similar sets}). Since inversion is a smooth operation on $GL(\mathbb{R}^2)$ (a polynomial in several variables, also note that $\det M >c'>0$ for some $c'$ for every $M\in \Sigma'$), $\Gamma = (\Sigma')^{-1}$ has dimension $\dim \Gamma = \dim \Sigma ' \geq 1$.

\textbf{Step 4 - $\mathcal{E}(F)$ has positive Hausdorff dimension}   By the definition of $\Gamma$ (recall equation \eqref{Equation for gamma}) we have
\begin{equation*}
\Gamma = \bigcup_{I \in \lbrace 1,...,l \rbrace^*} \lbrace \gamma \in \Gamma: \phi_I (F) \subseteq \gamma (F)   \rbrace.
\end{equation*}
Since $\dim \Gamma \geq 1$, it follows that for some cylinder $\phi_I$ we have
\begin{equation*}
\dim \lbrace \gamma \in \Gamma: \phi_I (F) \subseteq \gamma (F)   \rbrace = \dim \lbrace \gamma \in \Gamma: \gamma^{-1} \circ \phi_I (F) \subseteq F   \rbrace  > \frac{1}{2}.
\end{equation*}
Denote $\Lambda = \lbrace \gamma \in \Gamma: \gamma^{-1} \circ \phi_I (F) \subseteq F   \rbrace$. Then since inversion is a smooth operation in $G_2$  we see that $\dim \Lambda^{-1} = \dim \Lambda$. Since composition with $\phi_I$ is a diffeomorphism of $G_2$, we see that 
\begin{equation*}
\dim \Lambda^{-1} \circ \phi_I = \dim \Lambda^{-1} = \dim \Lambda > \frac{1}{2},
\end{equation*}
where $\Lambda^{-1} \circ \phi_I  = \lbrace g^{-1} \circ \phi_I : g\in \Lambda \rbrace$. Now, by definition $\mathcal{E}(F) \supset \Lambda^{-1} \circ \phi_I$. That is, every element in $\Lambda^{-1} \circ \phi_I$ is a self embedding of $F$. Therefore,
\begin{equation*}
\dim \mathcal{E}(F) \geq  \dim \Lambda^{-1} \circ \phi_I   > \frac{1}{2}.
\end{equation*} 

However, $\Phi$ has the SSC, so $F$ admits a self similar measure of maximal dimension that does not admit a $1$-slicing, by Theorem \ref{Theorem saturation for self similar measures}.  So, by Theorem \ref{Theorem zero dimension of self embeddings}  we must have $\dim \mathcal{E}(F)=0$. This is our desired contradiction.   \hfill{$\Box$}

\section{Proof of Theorem \ref{Theorem one into one}} \label{Section proof of affine embeddings} \label{Proof of affine embeddings}
In this section we prove our results concerning situations when one self similar set $F \subset [0,1]^2$ embeds into another self similar set $E \subset [0,1]^2$. The end game in all our arguments here is to employ the inverse Theorem \ref{Theorem Inverse Theorem main application}. To apply  this Theorem, we must prescribe a bounded set $X\subset G_2$ where we allow such embeddings to live. Note that this set may only depend (in our setting) on $F$. This is where the assumptions that both $F$ and $E$ are in $[0,1]^2$, and that $F$ does not sit on any affine line, come into play. Thus, since $\mathcal{E}(F,E) \subseteq \mathcal{E}(F,[0,1]^2)$ and $F\subset [0,1]^2$ does not sit on an affine line, it is not hard to see that $\mathcal{E}(F,[0,1]^2)$ is a bounded subset of $G_2$ (see \eqref{Equation E(F,E)} for the definition of $\mathcal{E}(F,E)$). So, this will always be our $X$ when we use Theorem \ref{Theorem Inverse Theorem main application}.

We begin  by proving that for a nice self similar set $F$, $\dim \mathcal{E}(F,E)$ is continuous in $\dim E-\dim F$ at $0$.  
\begin{theorem} \label{Theorem contiouty of dimesnion of embeddings}
Let $F$ be a self similar set that supports a self similar measure $\mu$ of maximal dimension, and suppose $0<\dim F <2$.  Suppose that $\mu$ does not admit a $1$-slicing. Then for every $\epsilon>0$ there exists $\delta=\delta(\mu,\epsilon)>0$ such that:

 For every self similar set $E$, if $\dim E - \dim F <\delta$ then $\dim \mathcal{E}(F,E) \leq \epsilon$. 
\end{theorem}

\begin{proof}
Let $\epsilon>0$. Since $\mu$ does not admit a $1$-slicing, we can produce $\epsilon_0 =\epsilon(\mu)$ from Theorem \ref{Theorem Inverse Theorem main application}. Let $\epsilon_1 = \min \lbrace \epsilon,\epsilon_0 \rbrace$. Produce the $\delta(\epsilon_1)=\delta >0$ from Theorem \ref{Theorem Inverse Theorem main application} (using $X = \mathcal{E}(F,[0,1]^2)$).  Let $E$ be a self similar set, and suppose that $\dim E - \dim F < \delta$. We prove that $\dim \mathcal{E}(F,E) \leq \epsilon_1 \leq \epsilon$. 

Suppose towards a contradiction that $\dim \mathcal{E}(F,E) > \epsilon_1$. It follows that we can find a measure $\nu \in P(G_2)$ of (upper) entropy dimension $>\epsilon_1$ that is supported on $X$ (since $\mathcal{E}(F,E) \subseteq X$). Recall that $\mu$ is a self similar measure on $F$ with $\underline{\dim_H} \mu = \dim_H F$. By Theorem \ref{Theorem Inverse Theorem main application} we have $\overline{\dim_e} \nu.\mu \geq \underline{\dim_H} \mu + \delta = \dim F + \delta$. Note that $\nu.\mu$ is supported on $E$, so $\dim_B E \geq \overline{\dim_e} \nu.\mu$.  Finally, $E$  is a self similar set, so $\dim_H E = \dim_B E \geq \dim F + \delta$, a contradiction. We conclude that $\dim \mathcal{E}(F,E) \leq \epsilon$.
\end{proof}

$$ $$
\textbf{Proof of Theorem \ref{Theorem one into one} assuming condition (1)} Recall that we are assuming $F$ admits a self similar measure $\mu \in P(F)$ of maximal dimension, that does not admit a $1$-slicing.  Produce  $\delta(\mu, \frac{1}{2}) = \delta>0$ using Theorem \ref{Theorem contiouty of dimesnion of embeddings}. Let $E$ be a self similar set generated by an IFS $\Psi$ that has the SSC and has a uniform contraction $\lambda>0$, and such that $\dim E - \dim F < \delta$.   Let $\phi_i \in \Phi$ be any map, $\phi_i(z)=\alpha O(z)+t$, where $\alpha>0, O\in O(\mathbb{R}^2)$ and $t\in \mathbb{R}^2$.  We prove that if $\frac{\log \lambda}{\log \alpha}\notin \mathbb{Q}$, then $\mathcal{E}(F,E)=\emptyset$. Suppose towards a contradiction that this is not the case, and without the loss of generality suppose $i=1$. Let  $g(z)=Az+t$ be an affine map such that $g(F)\subseteq E$.

Let $\rho = \min_{i\neq j} d( \psi_i (E),\psi_j (E))>0$.  Let $k\in \mathbb{N}$ be such that $\alpha^k \leq \frac{\rho}{||A||\cdot \text{diam} (F)}$. Thus, by Lemma \ref{Lemma Restrict map rescale}, it follows that for all large $n\in \mathbb{N}$, $g(\phi_{1^n} (F))$ is contained within a unique cylinder $\psi_{I_n} (F)$ of generation $[(n-k) \frac{\log \alpha}{\log \lambda}]+1$, where $[\cdot ]$ stands for integer (floor) value. Then the map $\psi_{I_n} ^{-1} \circ g \circ \phi_{1^n}$ defines an  embedding of $F$ into $E$ by an affine map. Furthermore, the linear part of this affine map has norm $\lambda^{-1} \cdot \alpha^{k} \cdot \lambda^{\lbrace (n-k)\cdot\frac{\log \alpha}{\log \lambda}\rbrace}\cdot ||A||$, where $\lbrace x \rbrace = x -[x]$. Let $\Sigma'$ denote the closure of  the sequence of all affine embeddings $\psi_{I_n} ^{-1} \circ g \circ \phi_{1^n}$ obtained this way. Let $\Sigma$ denote the set of linear parts of the maps in $\Sigma'$.

The map taking $\sigma \in \Sigma$ to $||\sigma||$ is a Lipschitz continuous map from $\Sigma$ to $\mathbb{R}$. Since $\frac{\log \alpha}{\log \lambda} \notin \mathbb{Q}$, it follows that the image of $\Sigma$ under this map is the non degenerate interval 
\begin{equation*}
[ \alpha^{k} \cdot  ||A||, \lambda^{-1} \cdot \alpha^k \cdot ||A||].
\end{equation*}
 Thus, $\dim \Sigma \geq 1$.  Therefore, $\dim \mathcal{E}(F,E)\geq \dim \Sigma' \geq \dim \Sigma \geq 1$. This contradicts the fact that Theorem \ref{Theorem contiouty of dimesnion of embeddings} guarantees that $\dim \mathcal{E}(F,E)\leq \frac{1}{2}<1$, by our choice of $\delta$. \hfill{$\Box$}

$$ $$

\textbf{Proof of Theorem \ref{Theorem one into one} assuming condition (2)}  Recall that we are assuming $F$ is generated by an IFS $\Phi$ with the OSC,  such that $|G_\Phi|=\infty$. Thus, we may assume $\Phi$ contains a similarity with orthogonal part an irrational rotation (see Section \ref{Section self similar sets}). Let $\mu$ be a self similar measure on $F$ of dimension $\dim F$. Then $\mu$ does not admit a $1$-slicing according to Proposition \ref{Theorem saturation for self similar measures}. We may thus produce  $\delta=\delta(\mu,\frac{1}{2})>0$ as in Theorem \ref{Theorem contiouty of dimesnion of embeddings}.

Let $E$ be generated by an IFS $\Psi$ that has the SSC, $|G_\Psi|<\infty$, and $\Psi$ has a uniform contraction $\lambda$. Suppose $\dim E - \dim F < \delta$. We prove that $\mathcal{E}(F,E)=\emptyset$. Let $g(z)=Az+t'$, where $A\in GL(\mathbb{R}^2)$ and $t'\in \mathbb{R}^2$, and suppose towards a contradiction that $g(F)\subseteq E$.

 Write $\phi(z)=\alpha O(z)+t$ for the map $\phi \in \Phi$ such that $O\in SO(\mathbb{R}^2)$ is an irrational rotation, and as usual $\alpha>0,t\in \mathbb{R}^2$ . Since $F$ supports a self similar measure of maximal dimension that does not admit a $1$-slicing , by the previous proof  we may assume that $\alpha^m = \lambda^n$ for some $m,n\in \mathbb{N}$. We may thus assume without the loss of generality that $\alpha = \lambda$ (by taking the iterated IFS's $\Phi^m$ and $\Psi^n$ instead of $\Phi$ and $\Psi$. For the definition of the iterated IFS $\Phi^n$ see e.g. Section \ref{Section self similar sets} or \cite{feng2009structures}).

Let $\rho = \min_{i\neq j} d( \psi_i (E),\psi_j (E))>0$.  Let $k\in \mathbb{N}$ be such that $\lambda^k \leq \frac{\rho}{||A||\cdot \text{diam} (F)}$. Let $n\in \mathbb{N}$ be such that $n>k$, then by Lemma \ref{Lemma Restrict map rescale} it follows that there is a unique $n-k+1$-th generation cylinder $\psi_{I_n}$ such that $g(\phi_{1^n} (F)) \subseteq \psi_{I_n} (E)$. Thus, $\psi_{I_n} ^{-1} \circ g \circ \phi_{1^n} $ defines an affine embedding of $F$ into $E$ for every $n>k$. The norm of each such embedding is $\lambda^{k-1} \cdot ||A||$. The linear part of each such embedding has the form
\begin{equation*}
\lambda^{k-1} \cdot U_n ^{-1} \cdot A \cdot O^n , \quad U_n \in G_\Psi .
\end{equation*}

Let $\Sigma'$ be the closure of the collection $\lbrace \psi_{I_n} ^{-1} \circ g \circ \phi_{1^n} \rbrace $ of embeddings of $F$ into $E$, and let $\Sigma$ denote the set of linear parts of $\Sigma'$. Then, as $O$ is an irrational rotation, for every $R\in SO(\mathbb{R}^2)$ there is a map $\sigma$ in $\Sigma'$ such that its linear part is $\lambda^{k-1} UAR$, for some $U \in G_\Psi$. Therefore, we may write
\begin{equation*}
SO(\mathbb{R}^2) = \bigcup_{U\in G_\Psi} \lbrace R\in SO(\mathbb{R}^2): \lambda^{k-1} UAR \text{ is the linear part of some } \sigma \in \Sigma' \rbrace. 
\end{equation*}
So, since $A$ is fixed and $|G_\Psi|<\infty$, it follows that there is a closed subset $\Gamma$ of dimension $1$ (the dimension of $SO(\mathbb{R}^2)$) in $SO(\mathbb{R}^2)$, such that: for a fixed $U\in G_\Psi$, for every $R\in \Gamma$, $\lambda^{k-1} UAR$ is the linear part of some map in $\Sigma'$ (so $\lambda^{k-1} UAR\in \Sigma$). Define $T_U : \Gamma \rightarrow GL(\mathbb{R}^2)$ by $T_U( R) = \lambda^{k-1} UAR$. Then $T_U (\Gamma) \subseteq \Sigma$, and since $T_U$ is a bi-Lipschitz map between $\Gamma$ and $T_U (\Gamma)$, we find that $\dim \Sigma \geq \dim T_U (\Gamma) \geq 1$. It follows that $\dim \Sigma' \geq \dim \Sigma \geq 1$. However, by our choice of $\delta$ we have
\begin{equation*}
1 \leq \dim \Sigma \leq \dim \mathcal{E} (F,E) \leq \frac{1}{2}<1
\end{equation*}
a contradiction. We conclude that $\mathcal{E}(F,E)=\emptyset$. \hfill{$\Box$}
$$ $$

The proof of Theorem \ref{Theorem one into one} assuming condition (3) follows along similar lines with minor modifications, so we just sketch the proof: produce $\delta$ using Theorem \ref{Theorem contiouty of dimesnion of embeddings}, and suppose  that there exists $g\in \mathcal{E}(F,E)$. Retaining the same notations as above, since $G_\Phi$ is  finite, we can assume without the loss of generality that $\phi_1$ is a homothety. We may thus produce an affine embedding $\psi_{I_n} ^{-1} \circ g \circ \phi_{1^n}$ of $F$ into $E$ as done above. Its linear part now has the form $\lambda^{k-1} U^{-n+k-1} \cdot A$, where $U$ is the  irrational rotation that is the orthogonal part of all the similarities in $G_\Psi$. Since $U$ is an irrational rotation this implies that $\dim \mathcal{E}(F,E) \geq 1$, which is impossible if $\dim E - \dim F$ is small enough. 

\section{Proof of part (2) of Theorem \ref{Theorem embeddings}}

Recall that now we are assuming that $F$ is generated by an IFS $\Phi$ with the SSC, a uniform contraction and such that $G_\Phi$ is finite. We begin the proof of part (2) of Theorem \ref{Theorem embeddings} by showing  that if $g(z)=A(z)+t,A\in GL(\mathbb{R}^2),t\in \mathbb{R}^2$ and $g(F)\subseteq F$ then $A$ is diagonalizable over $\mathbb{C}$.

\subsection{Proof that the linear part is diagonalizable} \label{Section part 1}
\textbf{Overview} We make use (again) of the "restrict - map - rescale" heuristic. We do this by considering the  iterated map $g^n$, and rescaling to get a new embedding using a cylinder of $F$ containing $g^n (F)$ (of generation approximately $\log ||g^n||$). Assuming $A$ is not diagonalizable,  its powers approach a rank 1 matrix after being normalized. Considering all these (re-scaled) embeddings, we find that there is a projection $P$ to a line such that a large set of scalings of $P(F)$ embed into $F$, and so a large set of scalings  of $P(F)$ embedd into $P(F)$ (the fact that we get a large set scalings uses the non-diagonalizable assumption). An application of the one dimensional inverse Theorem \ref{Theorem 4.2} shows this is impossible unless the projection $P(F)$ has dimension one. Since the projection embeds into a slice of F and all slices have dimension smaller than one (because of strong separation by Corollary \ref{Corollary SSC implies dimension gap for slices}), we obtain a contradiction. 
$$ $$
\textbf{Proof} Suppose towards a contradiction that $A$ is not diagonalizable (over $\mathbb{C}$).

\textbf{Step 1 - The Jordan decomposition of $A$} We first claim that  the Jordan decomposition of $A$ satisfies  $A=UJU^{-1}$ where $U \in GL(\mathbb{R}^2)$ and 
\begin{equation*}
J =  \begin{pmatrix}
\gamma & 1 \\
0 & \gamma \\
\end{pmatrix} , \quad \gamma\in \mathbb{R},  |\gamma| <1.
\end{equation*}
The assertion about $\gamma \in \mathbb{R}$ and $U\in GL(\mathbb{R}^2)$ follows since $A$ may have only real eigenvalues (if it has a complex eigenvalue then its complex conjugate is also an eigenvalue so $A$ is diagonalizable and we are done). We also claim that $|\gamma|<1$. First,  note that for every $n\in \mathbb{N}$ we have
\begin{equation*}
A^n = U J^n U^{-1} = U \begin{pmatrix}
\gamma^n & n\gamma^{n-1} \\
0 & \gamma^n \\
\end{pmatrix} U^{-1}.
\end{equation*}
Also, since $g$ is affine
\begin{equation*}
g^n ( \widetilde{F} ) \subseteq \widetilde{F}
\end{equation*}
where $\widetilde{F}$ is the closure of the convex hull of $F$. Since $F$ is a compact set that does not sit on an affine line, we see that $\widetilde{F}$ has positive and finite  volume. So,  we have
\begin{equation*}
|\gamma|^{2n} \cdot \text {vol} (\widetilde{F}) = \text{vol} (g^n (\widetilde{F})) \leq \text{ vol} (\widetilde{F}).
\end{equation*}
Thus, $|\gamma| \leq 1$. To see why we cannot have $|\gamma|=1$, let $w,z\in F$ be such that  
\begin{equation*}
P_2 (U^{-1} (w)) \neq P_2 (U^{-1} (z))
\end{equation*}
(recall that $F$ does not sit on any affine line and that $P_2(x,y)=y$).  Thus,  if $|\gamma| = 1$ then for any $n\in \mathbb{N}$,
\begin{equation*}
\text{diam} (F) \geq || g^n (w) - g^n (z)|| = || A^n (w) - A^n (z)||    
\end{equation*}
\begin{equation*}
 \geq || n \cdot  \left( P_2 (U^{-1} (w))  - P_2 (U^{-1} (z)) \right)+ D||\cdot ||U^{-1}||^{-1},
\end{equation*}
 where  $D = P_1 (U^{-1} (w)) - P_1 (U^{-1} (z))$,  a contradiction.

 Since $|\gamma|<1$ we have
\begin{equation*}
||J^n|| \leq 2\cdot ||J^n||_{\max} = 2\cdot n|\gamma|^{n-1},
\end{equation*} 
 where $||\cdot||_{\max}$ of a matrix is just the maximum of the absolute values of its entries. Since  $||A^n|| \leq ||U||\cdot ||J^n|| \cdot ||U^{-1}||$, we have $||A^n|| \leq ||U||\cdot ||U^{-1}||\cdot (2n|\gamma|^{n-1})$, using the operator norm on matrices as usual.

\textbf{Step 2 - generating more embeddings} We may assume without the loss of generality that $\gamma>0$ (otherwise we work with the self embedding $g^2$. Note that if the linear part of $g$ is not $\mathbb{C}$-diagonalizable, then neither is the linear part of $g^2$). Recall that we are assuming that $\Phi$ satisfies the SSC and has uniform contraction ratio $\lambda$. Let $\rho = \min_{i\neq j} (\phi_i (F),\phi_j (F))>0$, so that the distance between $n$-generation cylinders is $\rho\cdot \lambda^{n-1}$ (see Lemma \ref{Lemma Restrict map rescale}). Find $k\in \mathbb{N}$ such that 
\begin{equation*}
\gamma^{k} < \frac{\rho}{2||U||\cdot ||U||^{-1}\cdot \text{diam}(F)}.
\end{equation*}
Since the linear part of $g^n$ is $A^n$ for every $n\in \mathbb{N}$ such that $n>k$, we have
\begin{equation*}
\text{diam}(g^{n}(F)) \leq \text{diam}(F)\cdot ||A^{n}|| \leq ||U||\cdot ||U^{-1}||\cdot (2\cdot n\gamma^{n-1})\cdot \text{diam} (F) \leq \rho \cdot  n\cdot \gamma^{n-k-1} 
\end{equation*}
\begin{equation*}
 = \rho\cdot \lambda^{\frac{\log n\gamma^{n-k-1}}{\log \lambda}} \leq \rho\cdot \lambda^{[\frac{\log n\gamma^{n-k-1}}{\log \lambda}]},
\end{equation*}
It follows (by Lemma \ref{Lemma Restrict map rescale}) that $g^{n} (F)$ intersect a unique $[\frac{\log n\gamma^{n-k-1}}{\log \lambda}]+1$ generation cylinder $\phi_{I_n} (F)$, and is therefore contained in it.

Consider the affine self embedding $\phi_{I_n} ^{-1} \circ g^{n}$ of $F$.   Its linear part has the form, for $O_n \in G_\Phi$ that is the orthogonal part of $\phi_{I_n}$, and since $\Phi$ has a uniform contraction ratio,
\begin{equation*}
\lambda^{-1} \cdot \lambda^{-[\frac{\log n\gamma^{n-k-1}}{\log \lambda}]} \cdot O_n ^{-1} \cdot A^{n} = \lambda^{-1} \cdot O_n ^{-1} \cdot U \cdot \lambda^{-[\frac{\log n\gamma^{n-k-1}}{\log \lambda}]}\cdot J^{n} \cdot U^{-1}.
\end{equation*}
Note that:
\begin{itemize}
\item  The diagonal entries of $\lambda^{-[\frac{\log n\gamma^{n-k-1}}{\log \lambda}]}\cdot J^{n}$ are both $\gamma^{n} \cdot \lambda^{-[\frac{\log n\gamma^{n-k-1}}{\log \lambda}]}=O(\frac{1}{n})$.

\item The lower off diagonal entry of $\lambda^{-[\frac{\log n\gamma^{n-k-1}}{\log \lambda}]}\cdot J^{n}$ is $0$.

\item  The upper off diagonal entry of $\lambda^{-[\frac{\log n\gamma^{n-k-1}}{\log \lambda}]}\cdot J^{n}$ is 
\begin{equation} \label{equation off diagonal}
n\gamma^{n-1} \cdot \lambda^{-[\frac{\log n\gamma^{n-k-1}}{\log \lambda}]} =    \lambda^{\frac{\log n\gamma^{n-1}}{\log \lambda} -[\frac{\log n\gamma^{n-k-1}}{\log \lambda}]} = \lambda^{\frac{\log \gamma^{k}}{\log \lambda}+ \frac{\log n\gamma^{n-k-1}}{\log \lambda} -[\frac{\log n\gamma^{n-k-1}}{\log \lambda}]} 
\end{equation}
\begin{equation*}
= \lambda^{ \frac{\log \gamma^{k}}{\log \lambda} + \lbrace \frac{\log n\gamma^{n-k-1}}{\log \lambda}\rbrace}.
\end{equation*}
\end{itemize}

As a result,  the affine self embedding $\phi_{I_n} ^{-1} \circ g^n$ of $F$ has uniformly bounded norm, so  its translation part lies in a uniformly bounded neighbourhood of $F$. So, $\phi_{I_n} ^{-1} \circ g^n$ has converging sub-sequences. Moreover, every such converging subsequence  converges to an affine  map with linear part
\begin{equation*}
\lambda^{-1} \cdot O^{-1} \cdot U \begin{pmatrix}
0 &  \lambda^{ \frac{\log \gamma^{k}}{\log \lambda}}\cdot a \\
0 & 0\\
\end{pmatrix}  \cdot  U^{-1} = O^{-1} \cdot U \begin{pmatrix}
0 &  \lambda^{ \frac{\log \gamma^{k}}{\log \lambda}-1}\cdot a \\
0 & 0\\
\end{pmatrix}  \cdot  U^{-1} ,
\end{equation*}
where $a$ is an accumulation point of the sequence $ \lambda^{\lbrace \frac{\log n\gamma^{n-k-1}}{\log \lambda}\rbrace}$ and $O\in G_\Phi$ is the linear part of some cylinder map (note that here we use the fact that $|G_\Phi| < \infty$). We  claim that the set of such $a$'s we may obtain is in fact a connected interval of non zero length. This follows by verifying that the sequence $a_n = \lbrace \frac{\log n\gamma^{n-k-1}}{\log \lambda}\rbrace$ is dense in the circle
. Moreover, note that every affine map that arises as a limit of a subseqeunce of $\phi_{I_n} ^{-1} \circ g^n$ is a self embedding of $F$, since it is a limit of self embeddings of $F$.

\textbf{Step 3 - many scalings of a projection of $F$ embed into $F$} By the previous step (its last part),   every accumulation point of $\lbrace \phi_{I_n} ^{-1} \circ g^n \rbrace_{n\in \mathbb{N}}$ is a self embedding of $F$, and this set of limit maps has the following property. Fix the linear projection $L:\mathbb{R}^2 \rightarrow \mathbb{R}^2$ such that $L(z) = U \begin{pmatrix}
0 & 1 \\
0 & 0
\end{pmatrix} U^{-1} (z)$ . Then for every  $a$ in  a non degenerate interval $(b,c)$ there exists some $t_a\in \mathbb{R}^2$ and $O\in G_\Phi$ that is the orthogonal part of a cylinder map, such that 
\begin{equation*}
a \cdot O^{-1} \cdot L(\cdot ) +t_a
\end{equation*}
is a self embedding of $F$, that is,
\begin{equation*}
a\cdot O^{-1} \cdot  L (F)+t_a \subseteq F.
\end{equation*} 
This implies that 
\begin{equation} \label{Equation L(F)}
a\cdot  L (F)+O(t_a) \subseteq O (F).
\end{equation}
Recall that, by the construction in the previous step, $O$ is the orthogonal part of some cylinder map of $\Phi$.  Thus we can write, for some $n\in \mathbb{N}$ 
\begin{equation*}
O (F) = \lambda^{-n} \cdot \phi_I (F)-t' \quad \text{ for }  t' \in \mathbb{R}^2, I \in \lbrace 1,...,l\rbrace^n,
\end{equation*}
so
\begin{equation*}
O(F)=\lambda^{-n} \cdot \phi_I (F)-t' \subset \lambda^{-n} \cdot F-t'
\end{equation*}
Combining this with  equation \eqref{Equation L(F)} we have
\begin{equation*}
a\cdot \lambda^n \cdot L(F) +t_a ' \subseteq F,
\end{equation*}
for some translate $t_a ' \in \mathbb{R}^2$. Note that $\dim \ker (L)=1$. Also, note that since $|G_\Phi|<\infty$, we may assume that there are finitely many possible values of $n$ appearing in the above equation (because there is some $m$ such that every possible orthogonal part of a cylinder appears as an orthogonal part of a cylinder in $\lbrace \phi_I : 1 \leq |I| \leq m \rbrace$, since $|G_\Phi| < \infty$).

\textbf{Step 4 - an inverse Theorem for graph directed sets}  By Step 3, for every $a\in (b,c)$ there is some $1\leq n \leq m$ (where $m$ is global) and $t_a\in \mathbb{R}^2$ such that
\begin{equation*}
a\cdot \lambda^n \cdot L(F) +t_a  \subseteq F.
\end{equation*}
So,  $a\cdot \lambda^n \cdot L(F) +t_a $ is contained within the slice $F\cap ( \image (L) +t_a)$. Since  $F$ is the attractor of an IFS with the SSC, it follows by Corollary \ref{Corollary SSC implies dimension gap for slices} that 
\begin{equation*}
\dim L (F) \leq  \dim F\cap (\image (L)+t_a) <1.
\end{equation*}

We also claim that $\dim L (F) >0$. We first recall that $L(F)$ has finite Hausdorff measure in its dimension, by Theorem \ref{Theorem farkas}.  Now, if $\dim L(F)=0$ then $L(F)$ is finite, as it has positive and finite $0$-$\dim$ Hausdorff measure. Bearing in mind that $F$ is not supported on any affine line, and that this implies that $F$ is supported on a translate of $\ker (L)$, one readily sees that this is impossible.  

Next, recall that for every $a\in (b,c)$  there is a translation $t_a$ and some $1\leq n \leq m$ such that
\begin{equation*}
a\cdot \lambda^n \cdot  L (F)+t_a \subseteq F,
\end{equation*} 
which implies
\begin{equation*}
a\cdot \lambda^n \cdot  L (F)+L(t_a) \subseteq L(F),
\end{equation*} 
Noting that $L^2 (F)=L\circ L (F)$ equals $L(F)$. It follows that there is a set $\Sigma$ of affine self embeddings of $L (F)$ such that $\Sigma$ has dimension at least $1$.

Now, produce the $\delta=\delta(\dim L(F))>0$ guaranteed by Theorem \ref{Theorem 4.2}. Since $|G_\Phi|<\infty$,  by Theorem \ref{Theorem farkas}\footnote{
Note that, while formally $L:\mathbb{R}^2 \rightarrow \mathbb{R}^2$, we can clearly identify it with a map $L:\mathbb{R}^2 \rightarrow \mathbb{R}$, since $\dim \image (L)=1$. Thus, we identify $L(F)$ with the corresponding set in $\mathbb{R}$.} it follows that  we can find a self similar set $K\subset L (\phi_1 (F))$ that has a generating IFS with the SSC, such that $\dim L (\phi_1 (F)) -  \dim K <\delta$. Then for every $\sigma \in \Sigma$, $\sigma (K) \subseteq L (F)$. Let $\mu$ be a self similar measure of maximal dimension on $K$. Let $\nu \in P(G_1)$ be a measure supported on $\Sigma$ of (entropy) dimension $\geq 1$ \footnote{To find such a measure, we note $\Sigma \subset G_1$ can be identified with a compact subset of $\mathbb{R}^2$ such that $P_1 (\Sigma)$ contains an interval. Find a measure $\mu$ on $\Sigma$ such that $P_1 \mu$ is the Lebesgue measure on this interval. Then the entropy dimension of this measure is at least $1$.}. Then we know that the convolution measure $\nu.\mu$ (defined in the discussion before Theorem \ref{Theorem 4.2}) is supported on $L (F)$,  $\dim_B L (F) = \dim_H L (F)$ (by Theorem \ref{Theorem farkas}), and by Theorem \ref{Theorem 4.2} and \eqref{Equation dimension}, 
\begin{equation*}
\dim_B L (F) \geq \overline{\dim_e} \nu.\mu > \dim \mu +\delta = \dim K +\delta > \dim_H L (\phi_1 (F)) = \dim_H L (F).
\end{equation*}
 This is a contradiction. We conclude that $A$ must be diagonalizable.

\subsection{The largest eigenvalue is a rational power of lambda} \label{SEction largest eigenvalue}

Let $g(x)=Ax +t$, and recall that we are assuming $g(F)\subseteq F$. Denote the eigenvalues of $A$ as $\gamma_1, \gamma_2$. Note that either both $\gamma_1, \gamma_2 \in \mathbb{R}$ or both $\gamma_1, \gamma_2 \in \mathbb{C} - \mathbb{R}$, and in the latter case $|\gamma_1|=|\gamma_2|$ since they are complex-conjugates of each other. Thus, we have three cases to treat:
\begin{enumerate}
\item Suppose $|\gamma_1| = |\gamma_2|$ and $\gamma_1,\gamma_2 \in \mathbb{R}$. Then we must have $\gamma_2 = \pm \gamma_1$. Since $A$ is  diagonalizable (over $\mathbb{R}$ in this case), we see that $A^2$ is also diagonalizable with a unique real eigenvalue $\gamma_1 ^2$. Thus, $A^2 = \gamma_1 ^2 \cdot \id$. It follows that $A^2$ is a similarity matrix, so $g^2$ is a similarity map. By the results in \cite{elekes2010self} for similarity maps, $\gamma_1 ^2$  is a rational power of $\lambda$. Since $|\gamma_1| = |\gamma_2|$, the result follows.

\item Suppose $|\gamma_1| = |\gamma_2|$ and $\gamma_1,\gamma_2 \in \mathbb{C}-\mathbb{R}$. Write $\gamma_1 = a-bi, b\neq 0, a \in \mathbb{R}$. Then it is known that $A$ is similar to a similarity-rotation (i.e. a dilated rotation) matrix, namely
\begin{equation*}
A = T \cdot \begin{pmatrix}
a& -b \\
b & a \\
\end{pmatrix} \cdot T^{-1}, \quad T\in GL(\mathbb{R}^2).
\end{equation*}
Then for some  $R \in SO(\mathbb{R}^2)$ we can write
\begin{equation} \label{Equation similarity of A}
A = r\cdot T\cdot R \cdot T^{-1}, 
\end{equation}
for $r=|\gamma_1|=|\gamma_2|=\sqrt{a^2+b^2}>0$. Note that, as in the proof of step 1 in Section \ref{Section part 1}, we must have $r\leq 1$ since $g$ preserves the closure of the convex hull of $F$. If $r=1$ then $r=\lambda^0$, so we may assume $r<1$.

From here, we follow a similar argument as in Section \ref{Section part 1}, only we aim at using the two dimensional inverse Theorem \ref{Theorem Inverse Theorem main application} (or rather its consequence Theorem \ref{Theorem zero dimension of self embeddings}). By equation  \eqref{Equation similarity of A} we see that $||A||$ is dominated by $r\cdot ||T||\cdot ||T^{-1}||$. Next, as in step 2 of Section \ref{Section part 1}, for every $n\in \mathbb{N}$ we rescale the map $g^n$ using a cylinder $\phi_{I_n}$ of generation $[\frac{n-k \log r}{\log \lambda}]+1$, where $k\in \mathbb{N}$ is such that $r^k < \frac{\rho}{||T||\cdot ||T^{-1}||}$, and $\rho$ is the distance between first generation cylinders of $F$.  The resulting self embedding of $F$, $\phi_{I_n} ^{-1} \circ g^n$, has linear part, for some $O_n \in G_\Phi$,
\begin{equation} \label{equation limits of maps}
r^n \lambda^{-[\frac{n-k \log r}{\log \lambda}]-1} \cdot O_n ^{-1} T R^n T^{-1}= \lambda^{-1} \cdot  r^k \lambda^{ \lbrace \frac{n-k \log r}{\log \lambda} \rbrace} \cdot O_n ^{-1} T R^n T^{-1}
\end{equation}

Now, suppose that $\frac{\log r}{\log \lambda}\notin \mathbb{Q}$. Then, by taking limits of the maps $\phi_{I_n}^{-1} \circ g^n$ (which are all self embeddings of $F$),  and following arguments similar to the ones in step 3 of Section \ref{Section part 1} (based this time on \eqref{equation limits of maps}), we see that there exists a non-degenerate interval $(b,c)$ (where $b>0$)  such that for every $a\in (b,c)$ there is some rotation $R_a \in SO(\mathbb{R}^2)$, an orthogonal matrix $O_a \in G_\Phi$, and a translation $t_a \in \mathbb{R}^2$ such that 
\begin{equation} \label{self embedding}
a\cdot \lambda^{-1} \cdot r^k \cdot O_a \cdot T\cdot R_a \cdot T^{-1} (F) + t_a \subseteq F.
\end{equation}

Finally, we claim that this implies that $\dim \mathcal{E}(F)>0$ (where $\mathcal{E}(F)$ is the set of affine self embeddings of $F$). Indeed, for every $a\in (b,c)$ we may find  a self embedding of $F$ as in equation \eqref{self embedding}. Note that 
\begin{equation*}
\det (a\cdot \lambda^{-1} \cdot r^k \cdot O_a \cdot T\cdot R_a \cdot T^{-1}) = a^2 \cdot  \lambda^{-2} \cdot  r^{2k}  \cdot \det (O_a \cdot T\cdot R_a \cdot T^{-1}) = a^2 \cdot  \lambda^{-2}  \cdot r^{2k} .
\end{equation*}
Denote the set of linear parts of the maps in $\mathcal{E}(F)$ by $L\mathcal{E}(F)$.  It follows that the image of the set $L\mathcal{E}(F)$ under the map $\det:GL(\mathbb{R}^2)\rightarrow \mathbb{R}$ contains the non-degenerate interval $r^{2k}\cdot \lambda^{-2} \cdot (b^2,c^2)$. Since $\det$ is continuously differentiable (a polynomial in the entries of the matrix) it is locally Lipschitz, therefore
\begin{equation*}
\dim \mathcal{E} (F) \geq \dim L\mathcal{E}(F) \geq \dim \det (L\mathcal{E}(F)) \geq \dim \left( r^{2k}\cdot \lambda^{-2} \cdot (b^2,c^2) \right) = 1.
\end{equation*}
This contradicts Theorem \ref{Theorem zero dimension of self embeddings}, recalling that the natural self similar measure on $F$ does not admit a $1$-slicing since $\Phi$ has the SSC, by Theorem \ref{Theorem saturation for self similar measures}.

\item We remain with the case $|\gamma_1| > |\gamma_2|$, so in particular $\gamma_1, \gamma_2 \in \mathbb{R}$. By following a complete analogue of the argument given in Section \ref{Section part 1}, we may deduce that $|\gamma_1|$ must be a rational power of $\lambda$. Specifically, as in Step 1, we consider the Jordan decomposition of $A$ (and now we know that the Jordan form of $A$ is diagonal). Since $|\gamma_1 | > |\gamma_2|$, the operator norm of $A^n$ is dominated by $|\gamma_1|^n$. Next, as in step 2, we rescale the map $g^n$ using a  cylinder of generation $||g^n|| \approx [ n \frac{\log |\gamma_1|}{\log \lambda}]$. We thus obtain a sequence of self embeddings of $F$ that has linear parts that converge (upon moving to a subsequence)  to a rank $1$ matrix (regardless of the algebraic relation between $\gamma_1$ and $\lambda$).

 If $\frac{\log |\gamma_1|}{\log \lambda} \notin \mathbb{Q}$ we will thus obtain a large set (of dimension at least $1$) of affine embeddings of the projection $L(F)$ into $F$, where
\begin{equation*} L(z) = 
U \begin{pmatrix}
1 & 0 \\
0 & 0
\end{pmatrix} U^{-1} (z), \text{ where } A = U \begin{pmatrix}
\gamma_1 & 0 \\
0 & \gamma_2
\end{pmatrix} U^{-1}, \text{ and } U\in GL(\mathbb{R}^2).
\end{equation*} 
Here we employ arguments as in Steps 2 and 3. We obtain a contradiction as in Step 4.

Therefore, we may find some $k, k' \in \mathbb{N}$ such that $\gamma_1 ^{k'} = \lambda^k$. Since $|\gamma_2|^{k'} < |\gamma_1|^{k'}$, one sees that (as in the first paragraph of this case)  for some $m\in \mathbb{N}$ and some $t\in \mathbb{R}^2$, the projection $L(F)$ (defined above) satisfies $\lambda^{m} \cdot L(F) +t  \subset F$ for some $t$ (this only relies on the fact that the eigenvalues of $A$ don't have equal norms).
\end{enumerate}
\subsection{The smallest eigenvalue is also a rational power of lambda} \label{Section part 3}

Before proving the remaining assertion of part (2)  of Theorem \ref{Theorem embeddings} we require some preliminaries about  the notion of weak separation  for self similar sets. We also study the structure  of slices of a self homothetic set such that the corresponding projection admits the weak separation condition.
\subsubsection{The weak separation condition}
Let $K \subseteq \mathbb{R}$ be a self similar set, generated by an IFS $\Psi = \lbrace \psi_i \rbrace_{i=1} ^m$. Let 
\begin{equation*}
\mathcal{T} = \lbrace \psi_I ^{-1} \circ \psi_J : I\neq J \in \lbrace 1,.., m \rbrace^* \rbrace.
\end{equation*}
equip the group of all similarities on $\mathbb{R}$ with the topology induced by pointwise
convergence. We say that $\Psi$ has the weak separation condition (WSC) if 
\begin{equation*}
\id \notin \overline{\mathcal{T} \smallsetminus \lbrace \id \rbrace}.
\end{equation*}

This condition has a very useful implication: For any $r>0, x\in K$ we define 
\begin{equation} \label{Equation number of maps}
\Psi (x,r) = \lbrace \psi_I: \quad I \in \Psi_r , \quad  \psi_I (K) \cap B(x,r) \neq \emptyset \rbrace,
\end{equation}
where $\Psi_r$ was defined in \eqref{equation for Ar}. Note that $\Psi(x,r)$ only takes into account different maps  in $\Psi$ (i.e. it ignores exact overlaps if there are any). The following Lemma, proved e.g. by  Fraser,  Henderson,  Olson, and Robinson in (\cite{Fraser2015Assouad}, part of the proof of Theorem 2.1) is key for our analysis :

\begin{Lemma} \cite{Fraser2015Assouad} \label{Lemma Fraser}
Let $K \subseteq \mathbb{R}$ be a self similar set generated by an $\Psi$ with the WSC. Then 
\begin{equation*}
\sup_{x\in K, r>0}  | \Psi (x,r) | < \infty
\end{equation*}
\end{Lemma}

\subsubsection{The weak separation condition for projections of self homothetic sets}
Let $F \subset \mathbb{R}^2$ be a self similar set generated by an IFS $\Phi = \lbrace \phi_i \rbrace_{i=1} ^m$ such that $G_\Phi = \lbrace \id \rbrace$ (i.e. $F$ is self-homothetic), and suppose $\Phi$ satisfies the strong separation condition. We shall also assume that $\Phi$ has a uniform contraction ratio $\lambda>0$. Recall that $P_1 :\mathbb{R}^2 \rightarrow \mathbb{R}, P_1(x,y)=x$, and define  the projected IFS
\begin{equation*}
P_1 \Phi = \lbrace x \mapsto \lambda \cdot x + t_1 : \quad  \exists 1 \leq i \leq m \text{ and } t_2 \text{ such that } \phi_i (x,y) = (\lambda\cdot x + t_1, \lambda \cdot y +t_2 ) \rbrace
\end{equation*}
which generates the real self similar set $P_1 (F)$.  For every $x\in P_1 (F)$, let $F^x$ denote the vertical slice through $F$ above $x$, i.e. 
\begin{equation*}
F^x = \lbrace (x,y) : (x,y)\in F \rbrace.
\end{equation*}

In this section we show that if the projected IFS $P_1 \Phi$ has the WSC, then $F^x$ can be approximated by  finite unions of (not necessarily disjoint)  sets, with a uniform bound on the number of the sets in these unions. For this purpose, denote $\Psi = \lbrace \psi_i \rbrace_{i=1} ^m = P_1 \Phi$ and $K=P_1 (F)$, and assume $\Psi$ has the WSC. Moreover, since $\Phi$ has a uniform contraction ratio $\lambda$, then so does $\Psi$. 

Let $x\in K$ and $n\in \mathbb{N}$ and define
\begin{equation*}
\Psi_n (x) = \lbrace \psi_I : \quad  |I|=n, \quad  x\in \psi_I (K) \rbrace.
\end{equation*}
Then it is clear that $\Psi_n (x) \subseteq \Psi(x,\lambda^n)$ (recall equation \eqref{Equation number of maps}). Therefore, by Lemma \ref{Lemma Fraser}, we have
\begin{Lemma} \label{Finite maps} $\sup_{x\in K, n\in \mathbb{N}}  | \Psi_n (x) | < \infty$
\end{Lemma}

Fix $x\in K$ and $n\in \mathbb{N}$. Let 
\begin{equation} \label{Equation Psi n x}
\Psi_n (x) = \lbrace \psi_{I^n _1} ,..., \psi_{I^n _{p(n,x)}} \rbrace
\end{equation}
where $p(n,x):=|\Psi_n (x) | $ is uniformly bounded across $x$ and $n$ by Lemma \ref{Finite maps}, and we order the maps in $\Psi_n (x)$ by lexicographic order\footnote{That is, the order defined as follows: let $I,J \in \lbrace 1,...,m \rbrace^*$. Let $I \wedge J$ denote their largest common ancestor. Then $I = (I\wedge J,i,U), J=(I\wedge J,j,U')$ where $1\leq i,j \leq m-1$ and $U,U' \in \lbrace 1,...m\rbrace^*$. Then $I>J$ if $i>j$.} .  For every $1 \leq i \leq p(n,x)$ define the approximate vertical slice

\begin{equation} \label{Equation approx slice}
S^n _i = \bigcup_{I: |I|=n\text{ and } P_1 \phi_I = \psi_{I^n _i}} \phi_{I} (F)
\end{equation}
where the map $P_1 \phi_I : K\rightarrow K$ is just $P_1 \phi_I (x) = P_1 \circ \phi_I (x,y)$ for some $y\in F^x$ (recall that $\phi_I$ is a homothety so this is well defined).  Note that these approximate slices have the following nice property (which does not rely on the WSC):
\begin{Lemma} \label{Lemma approx slice are uniform}
Let $I \in \lbrace 1,...,m \rbrace^*$ and let $\psi_I$ be any cylinder of $\Psi$. Let $c \in \psi_I (K)$. Then for any cylinder $\phi_J$ of $\Phi$ such that $P_1 \phi_J = \psi_I$ there exists some $d\in P_2 (F)$ such that $(c,d)\in \phi_J (F)$.
\end{Lemma}

\begin{proof}
By our assumptions, there is some $w \in K$ such that $c = \psi_I (w)$. Note that since $K=P_1 (F)$ and $w \in K$, there is some $z\in P_2 (F)$ such that $(w,z)\in F$. Let $\phi_J$ be any cylinder of $\Phi$ such that $P_1 \phi_J = \psi_I$. Let $P_2 \phi_J = \eta$ and let $\eta (z)=d$. Then
\begin{equation*}
\phi_J (w,z) = (\psi_I (w), \eta (z)) = (c,d)
\end{equation*}
so $(c,d)\in \phi_J (F)$, as required.
\end{proof}

In all the Lemmas below, and in this paper in general, limits of sets are taken with respect to the Hausdorff metric $d_H (\cdot,\cdot )$ on compact subsets of $Q = [-1,1]^2$. Note that in the following Lemma we do not need the WSC.
\begin{Lemma} \label{Lemma slice structure}
$ F^x = \lim_{n \rightarrow \infty} \bigcup_{i=1} ^{p(n,x)} S^n _i$.
\end{Lemma}
\begin{proof}
First, we show that for every $n$, $ F^x \subseteq \bigcup_{i=1} ^{p(n,x)} S^n _i$. Indeed, let $(x,y)\in F$ and let $n\in \mathbb{N}$, then there exists a unique $\phi_I$ such that $|I|=n$ and $(x,y) \in \phi_I (F)$. Let $\psi_I = P_1 \phi_I$. Then $x\in \psi_I (K)$ and therefore $\psi_I \in \Psi_n (x)$. It follows that $\psi_I = \psi_{I^n _j}$ for some $1\leq j \leq p(n,x)$. Since $(x,y) \in \phi_I (F) \subseteq S^n _j$, the claim follows.

Next, let $n\in \mathbb{N}$. Then by Lemma \ref{Lemma approx slice are uniform} $F^x$ intersects every cylinder $\phi_I (F)$ that appears in the union $\bigcup_{i=1} ^{p(n,x)} S^n _i$. That is, for every $1 \leq i \leq p(n,x)$ and every $\phi_I$ such that $P_1 \phi_I = \psi_{I^n _i}$ and $|I| =n$, there is some $(x,y')$ such that $(x,y') \in \phi_I (F)$. This is because, by definition, $x\in \psi_{I^n _i} (K)$. Therefore, since the diameter of each $\phi_I (F)$ is at most $\lambda^n \cdot  \text{diam}(F)$, we see that $\bigcup_{i=1} ^{p(n,x)} S^n _i \subseteq (F^x)_{\lambda^n \cdot  \text{diam}(F)}$ (the $\lambda^n \cdot \text{diam}(F)$ neighbourhood of $F^x$). Thus, $d_H (F^x , \bigcup_{i=1} ^{p(n,x)} S^n _i) \leq \lambda^n \cdot  \text{diam}(F)$. 
\end{proof}

The following Corollary is where the WSC comes into play:

\begin{Corollary} \label{Corollary}
For every $x\in K$ and every sequence of natural numbers $n_l$ there is a subsequence $n_{l_q}$ and some $k\in \mathbb{N}$ such that:
\begin{enumerate}
\item $k = \lim_{q\rightarrow \infty} p(n_{l_q}, x)$, where $p(n,x)$ was defined in (and after) equation \eqref{Equation Psi n x}.

\item For every $1 \leq i \leq k$, $\lim_{q\rightarrow \infty} S_i ^{n_{l_q}}$ exists and equals a compact set $S_i \subseteq F^x$.

\item $F^x = \bigcup_{i=1} ^k S_i$.
\end{enumerate}
\end{Corollary}
\begin{proof}
By Lemma \ref{Finite maps} there is some $C>0$ such that $1\leq |\psi_{n_l} (x)| =p(n_l,x) \leq C$. Thus, $p(n_l,x)$ admits a subsequence $p(n_{l_q},x)$ converging  to some $k\in \mathbb{N}$, and since these are natural numbers,  for every large enough $q$, $p(n_{l_q},x) = k$. This is the first part of the Corollary.

For the second and third parts of the Corollary, note that $\Psi(n_{l_q} ,x) =  \lbrace \psi_{I^{n_{l_q}} _1} ,..., \psi_{I^{n_{l_q}} _{k}} \rbrace$, arranged by lexicographic order, for every $q$ large enough (since by the previous argument $p(n_{l_q},x)=k$ for all large $q$). Now, fix some $1\leq i \leq k$. Recall that $S^{n_{l_q}} _i$, defined in \eqref{Equation approx slice}, is a compact subset of $[0,1]^2 \subset [-1,1]^2$ for every $q$, and therefore it has a convergent subsequence (without the loss of generality, assume this happens for our subseqeunce already). Let $S_i$ denote its limit. Thus, we may assume that $S^{n_{l_q}} _i$ converges to $S_i$ for every $1\leq i \leq k$ by taking further subseqeunces (To simplify notation we assume this subsequence is $n_{l_q}$). Thus, by Lemma \ref{Lemma slice structure} and Proposition \ref{Proposition - Hausdorff metric0} we have
\begin{equation*}
F^x = \lim_{q\rightarrow \infty} \bigcup_{i=1} ^{k} S^{n_{l_q}} _i =  \bigcup_{i=1} ^{k} S_i.
\end{equation*}
\end{proof}

We end this section by stating a condition ensuring that a projection of a self homothetic set has the WSC. 

\begin{Lemma} \label{WSC for projection}
Let $F\subset \mathbb{R}^2$ be self homothetic, generated by an IFS $\Phi$ with a uniform contraction ratio $\lambda$ and the SSC. Suppose that for  some $\gamma > 0,t\in \mathbb{R}^2$ we have $\gamma \cdot (P_1 (F),0)+t \subseteq F$. Then the projected IFS $\Psi = P_1 \Phi$ has the WSC.
\end{Lemma}
Though this Lemma is stated for a specific projection ($P_1$), it is in fact true for all projections of $F$. We now prove Lemma \ref{WSC for projection}.  For this end, let us recall some of the terminology introduced by Furstenberg in \cite{furstenberg2008ergodic}. Let $K\subseteq [0,1]$ be a compact set. A set $A$ such that $A \subseteq [-1,1]$ is called a (Furstenberg) miniset of $K$ if $A \subseteq (\gamma \cdot K + t)\cap [-1,1]$ for some $\gamma \geq 1, t\in \mathbb{R}$. A set $M$ is called a (Furstenberg) microset of $K$ if $M$ is a limit in the  Hausdorff metric on subsets of $[-1,1]$ of minisets of $K$.   Finally, the set $K$ is called Furstenberg homogeneous if every microset of $K$ is a miniset of $K$. 

\textbf{Proof of Lemma \ref{WSC for projection}} Recall that we are assuming $F \subseteq [0,1]^2$. We rely on  (\cite{kaenmaki2015weak}, Theorem 5.5). By this Theorem, our self similar set $K = P_1 (F) \subseteq [0,1]$ has the WSC if it satisfies the following two conditions:
\begin{enumerate}
\item $\dim_H K <1$.

\item $K$ is Furstenberg homogeneous.
\end{enumerate}
We begin by showing that item (1) above holds true. By our assumptions, there are some $\gamma>0, t\in \mathbb{R}^2$ such that $\gamma \cdot (K,0) +t \subseteq F$. Now, $\gamma \cdot (K,0) +t$ is contained within some horizontal slice of $F$, which we denote by $S$. Therefore, by Corollary \ref{Corollary SSC implies dimension gap for slices}
\begin{equation*}
\dim K = \dim (\gamma \cdot (K,0) +t) \leq \dim S < 1.
\end{equation*} 

For the second item above, let $M \subset [-1,1]$ be a microset of $K$. We show that $M$ is a miniset of $K$. Let $M = \lim_n A_n$ where $A_n$ are minisets of $K$, so that in particular
\begin{equation*}
(A_n,0) \subseteq \left( \gamma_n \cdot (K,0) + (t_n,0) \right) \cap ([-1,1] \times \lbrace 0 \rbrace), \quad \text{ for } \gamma_n \geq 1, t_n\in \mathbb{R}.
\end{equation*}
Note that if $\gamma_n$ is bounded from above then it is clear that $M$ is a miniset of $K$. We thus assume that $\gamma_n \rightarrow \infty$. Then we have, for every $n\in \mathbb{N}$
\begin{equation*}
(A_n,0)  \subseteq (\gamma_n \cdot (K,0) + (t_n,0) )\cap ([-1,1] \times \lbrace 0 \rbrace) \subseteq (\gamma_n \cdot (K,0) + (t_n,0) )\cap [-1,1]^2  
\end{equation*}
\begin{equation*}
= (\frac{\gamma_n}{\gamma} \cdot (\gamma \cdot (K,0)+t)-\frac{\gamma_n}{\gamma}\cdot t  + (t_n,0) )\cap [-1,1]^2 \subseteq (\frac{\gamma_n}{\gamma} \cdot F - \frac{\gamma_n}{\gamma} \cdot t +(t_n,0)) \cap [-1,1]^2  : = R_n .
\end{equation*}

Now, the sequence of sets $R_n$ are, by definition, minisets of $F$. Find some converging subsequence of $R_n$ (recall that the Hausdorff metric on compact subsets of $[-1,1]^2$ is compact). Without the loss of generality, assume $R_n$ converges. Then its limit, denoted by $R$, is a micorset of $F$.

It is well known that $F$, a self homothetic set with the SSC, is Furstenberg homogeneous. Therefore, there exists $\beta \geq 1, v \in \mathbb{R}^2$ such that
\begin{equation*}
R \subseteq (\beta \cdot F + v)\cap [-1,1]^2.
\end{equation*}
Thus, by e.g. Proposition \ref{Proposition - Hausdorff metric0} we may deduce that
\begin{equation*}
(M,0) = \lim_n (A_n,0) \subseteq \lim_n R_n = R \subseteq (\beta \cdot F + v)\cap [-1,1]^2 \subseteq \beta \cdot F + v.
\end{equation*}
Thus,
\begin{equation*}
M = P_1 (M,0) \subseteq P_1 (\beta \cdot F + v) = \beta P_1 (F) + P_1 (v) = \beta \cdot K + P_1 (v).
\end{equation*}
Therefore,
\begin{equation*}
M \subseteq ( \beta \cdot K + P_1 (v)) \cap [-1,1]
\end{equation*}
It follows that $M$ is a miniset of $K$. Thus, $K$ is Furstenberg homogeneous. Since $\dim_H K <1$, we deduce that $K$ has the WSC, as claimed. \hfill{$\Box$}
$$ $$

\textbf{Remark} 
We note that Lemma \ref{WSC for projection} extends to the following situation: let $F$ be a self similar set generated by an IFS $\Phi$ with a uniform contraction ratio $\lambda$ that satisfies the SSC and $|G_\Phi| < \infty$. It is known that for every $\delta >0$ we may produce $F' \subseteq F$ that satisfies $\dim F - \dim F' < \delta$, such that  $F'$ is self homothetic, generated by $\Phi'$ that satisfies the SSC. See Corollary 1.4 in \cite{farkas2015dimension}. 

Now,  suppose that for  some $\gamma >0,t\in \mathbb{R}^2$ we have $\gamma \cdot (P_1 (F),0)+t \subseteq F$. Then for the self homothetic set mentioned above $F' \subseteq F$, the projected IFS $\Psi' = P_1  \Phi'$ has the WSC. To see this, we again employ Theorem 5.5 from \cite{kaenmaki2015weak}. In the proof of Theorem 5.5, it is shown that
\begin{equation} \label{Assouad dimension}
\dim_A (P_1 (F')) \leq \sup \lbrace \dim M : M \text { is a Furstenberg microset of } P_1 (F') \rbrace.
\end{equation}
Where $\dim_A$ denotes the Assouad dimension of a set\footnote{See e.g. \cite{fraser2014assouad} for a discussion and definition of the Assouad dimension.}.  Next, since $\gamma \cdot (P_1 (F),0)+t \subseteq F$ we have,  by Corollary \ref{Corollary SSC implies dimension gap for slices} and the remark following it, that $\dim P_1 (F)<1-c$ for some $c>0$, since $\gamma \cdot (P_1 (F),0)+t$ is contained within a horizontal slice of $F$. Also, since $F'\subseteq F$, $\gamma \cdot (P_1 (F'),0)+t \subseteq F$ . Therefore, by the argument given in Lemma \ref{WSC for projection}, we see that for any Furstenberg micorset $M$ of $P_1(F')$, $(M,0)$ is contained within a microset $R$ of $F$. 

Now, it is known that $F$ is also Furstenbeg homogeneous (since $\Phi$ has the SSC and $|G_\Phi
|<\infty$). Therefore,  if $R$ is a microset of $F$ then $R \subseteq (\beta F + t')\cap [-1,1]^2$ is a miniset of $F$.  So, for our microset $M$ of $P_1 (F')$, $(M,0)$ is contained within $\beta F + t'$.  Thus, $M=P_1 (M,0)$ is contained within $\beta P_1 (F) +P_1 (t')$. Since $\dim P_1 (F)<1-c$ we see that  $\dim M$ is bounded away from $1$. Therefore, by equation \eqref{Assouad dimension}
\begin{equation*}
\dim_A (P_1 (F')) < 1.
\end{equation*}
It now follows now from the Fraser, Henderson, Olson, Robinson dichotomy  for self similar sets  (\cite{fraser2014assouad}, Theorem 1.3) that $\Psi'$ has the WSC.

\subsubsection{Proof that the smallest eigenvalue is a rational power of lambda}
\textbf{Overview of the proof} Let us first explain the general lines of the proof. Recall that we are assuming $g(z)=Az+t$ is an affine self embedding of $F$, and that $A$ has eigenvalues $\gamma_1, \gamma_2$. As explained below, we may assume that $F$ is self homothetic (i.e. $|G_\Phi| = 1$) with the SSC, that $P_1 (F)$ admits an affine embedding into $F$, and that $|\gamma_1|> |\gamma_2|$. Let $\mu$ be the natural self similar measure on $F$. Let $\lbrace \mu_{[x]} \rbrace$ denote the decomposition of $\mu$ into conditional measures according to $P_1$, so that, in particular, for $P_1 \mu$ almost every $x\in P_1 (F)$, $\mu_{[x]}$ is well defined and supported on the slice above $x$ $F^x$.  Fix some conditional measure $\mu_{[x]}$ obtained this way.

The first instalment of the proof is about studying the geometry of self similar sets that are preserved by affine maps whose linear part is diagonalizable but not a similarity (in particular, here we don't assume anything about $\gamma_2$, apart from $|\gamma_2| < |\gamma_1|$).  We  show that since $P_1 (F)$  has the WSC (which may be deduced using Lemma \ref{WSC for projection}),  $F$ includes an affine image of a product $P_1 (F) \times P_2(S)$, where $S$ is a set that has positive measure with respect to $\mu_{[x]}$. This is done by analysing Hausdorff limits of blowups of the image set $g^n (F)$, where the blowup is by a factor of $\approx |\gamma_2|^{-n}$. On the one hand, every such limit is contained within an affine image of $F$ (since e.g. $F$ is Furstenberg homogeneous). On the other hand, using the WSC on $P_1 (F)$, we may decompose the support of $\mu_{[x]}$ into finitely many sets  (via Corollary \ref{Corollary}), so one of them has positive measure. Denote it by $S$. We then exploit the self homothetic nature of $F$ to  find subsets of these blowups of $g^n (F)$ that converge to (an affine image of) the product set $P_1 (F) \times P_2(S)$.

The second instalment of the proof is about showing that $|\gamma_2|^q = \lambda $ for some $q\in \mathbb{Q}$. The idea here is to repeat the argument of the first instalment,  only this time we study blowups of $g^n$ applied the affine image of $P_1 (F) \times P_2(S)$, where the blowup is again by a factor of $\approx |\gamma_2|^{-n}$ . The main point here is that if $\frac{\log |\gamma_2|}{\log \lambda} \notin \mathbb{Q}$, then we obtain a set of dimension $1$ of embeddings of $P_1 (F) \times P_2(S)$ into $F$. Since $P_1 (F) \times P_2(S)$ supports the nice measure $P_1 \mu \times P_2 (\mu_{[x]}|_S)$, we may apply the inverse Theorem \ref{Theorem Inverse Theorem for product measures - finite group} to obtain a contradiction.
$$ $$

Before we begin the proof, we first recall what we proved in Sections 1 and 2:

\textbf{What we already proved} \begin{itemize}
\item The linear part of $g$ is  diagonalizable. We denoted the eigenvalues by $\gamma_1, \gamma_2$.

\item By the argument appearing in Section \ref{SEction largest eigenvalue}, we may assume $|\gamma_1|> |\gamma_2|$. Therefore both eigenvalues are real, and we showed that $|\gamma_1|$ is a rational power of $\lambda$. 

\item Let $U$ be the matrix diagonalizing the linear part of $g$.  In Section \ref{SEction largest eigenvalue} part 3 we a showed that since $|\gamma_1|>|\gamma_2|$,  for some $m\in \mathbb{N}$ and $t\in \mathbb{R}^2$ 
\begin{equation*}
\lambda^m U \cdot \begin{pmatrix}
1 & 0 \\
0 & 0\\
\end{pmatrix}  \cdot U^{-1} (F) + t\subset F.
\end{equation*}
\end{itemize}

\textbf{What we can assume} \begin{itemize}
\item By iterating $g$ we may assume $\gamma_1 = \lambda^k$ for some $k\in \mathbb{N}$. Without the loss of generality, assume $k=1$ (otherwise we can just iterate the IFS $\Phi$ to obtain a generating IFS for $F$ with uniform contraction $\lambda^k$). We thus write the eigenvalues of $A$ as $\lambda,\gamma$ and assume (as we may) that $0<\gamma < \lambda$.

\item As usual, we assume that $F \subset [0,1]^2$.

\item We first assume that $F$ is self homothetic. The the general strategy of the proof is to first handle this case, and then,  using Theorem \ref{Theorem Inverse Theorem for product measures - finite group}, handle the general finite group case  approximating $F$ from the inside using a self homothetic set. See Step 15 below. 

\item We may assume that $P_1 (F)$ can be affinely embdded into $F$, by changing coordinates according to $U$ so that, in these coordinates, $P_1 (F)$ can be affinely embedded into $F$ by $\lambda^m (P_1 (F),0) + t \subset F$. Specifically, we  consider the set $U^{-1} (F)$ that is generated by the conjugated IFS $U^{-1} \circ \Phi \circ U$. It is not hard to see that this conjugated IFS retains the properties of $\Phi$: it has the SSC, it is self homothetic and has a uniform contraction ratio $\lambda$. 

\item With this change of coordinates, since our original map was
\begin{equation*}
g(z)= U \begin{pmatrix}
\lambda & 0 \\
0 & \gamma
\end{pmatrix}\cdot U^{-1} z +t'
\end{equation*}
we may now assume that the affine map $g(x,y)=(\lambda\cdot x, \gamma\cdot y)+t$ for $t=U^{-1} (t')$ is our self embedding of $F$ (after the change of coordinates), that is $g(F)\subseteq F$. Recall that $\lambda$ is the uniform contraction of $\Phi$ and that $0<\gamma<\lambda$.
 
\item As before, we denote $K=P_1 (F) \subseteq [0,1]$.

\end{itemize}

With these assumptions, we aim to prove that $\gamma$ must be a rational power of $\lambda$. Along the way, we prove the following Proposition, which is of independent interest.
\begin{Proposition} \label{Proposition}
Let $\mu$ be the natural self similar measure on $F$, and $K= P_1 (F)$. Then for $P_1 \mu$ almost every $x\in K$  the conditional measure $\mu_{[x]}$ in the decomposition of $\mu$ according to $P_1$ admits a compact set $S \subseteq [-1,1]^2$ with $\mu_{[x]} (S)>0$, such that $F$ contains an affine image of the product set $K\times P_2(S)$.
\end{Proposition}

The proof of the Proposition is outlined in steps 1 through  8. The rest of the steps contain a proof of part (2) of Theorem \ref{Theorem embeddings} using Proposition \ref{Proposition}. 

\textbf{Step 1 - choosing $x\in P_1 (F)$ via the dimension conservation formula} Let $\mu$ be the natural self similar measure on $F$ (so it has maximal dimension). By the dimension conservation formula for self homothetic sets with the SSC (see the discussion before equation \eqref{Equation DC 1}) we obtain that
\begin{equation} \label{DC}
\dim_H P_1 \mu + \dim_H \mu_{[x]} = \dim_H \mu, \text{ for } P_1 \mu \text{ almost every } x \in P_1 (F),
\end{equation}
where $[x] = \lbrace v\in \mathbb{R}^2: P_1 (v)=x\rbrace$, so that $\mu_{[x]}$ is supported on the slice $F^x$.  Let $x\in P_1 (F)$ be such that it satisfies equation \eqref{DC}, so $x\in DC(F,P_1)$ (recall equation \eqref{Equation for DC(X,P)}).

Note that since $\mu_{[x]}$ is supported on a parallel affine line to the $y$-axis, the projected measure $P_2 (\mu_{[x]})$ has essentially all the properties of the measure $\mu_{[x]}$, since $P_2$ acts as a translation in this context. In particular,
\begin{equation*}
\dim_H P_1 \mu + \dim_H P_2 (\mu_{[x]}) = \dim_H \mu.
\end{equation*}

\textbf{Step 2 - $P_1 (F)$ has the WSC} Since $F$ includes an affine homothetic image of $P_1 (F)$, $\lambda^m (P_1 (F),0) + t \subset F$, we see that the projected IFS $P_1 \Phi = \Psi$ has the WSC by Lemma \ref{WSC for projection}.

\textbf{Step 3 - The miniset $M_n$} Let $k(n) = [n\frac{\log \gamma}{\log \lambda}]$ (recall that $0<\gamma<\lambda$). Let $(x,y)\in F^x$ and define
\begin{equation} \label{The set Mn}
M_n = \left( \lambda^{-k(n)} \cdot [g^n (F) - g^n (x,y)] \right) \cap Q
\end{equation}
where $Q = [-1,1]^2$.
\begin{Claim} \label{Claim homogneous}
Let $M\subseteq Q$ be an accumulation point of the sequence $M_n$ with respect to the Hausdorff metric. Then  $M \subseteq \lambda^{-b} F +t$, for a fixed $b\in \mathbb{N}$ and some $t\in B(0,\lambda^{-b})$. 
\end{Claim}

\begin{proof}
Let $\rho = \min_{i\neq j} d(\phi_i (F), \phi_j (F))$ (the nearest point distance). Let $p \in \mathbb{N}$ be such that $\lambda^p < \rho$. Let $q\in \mathbb{N}$ be such that $\lambda^{-q} > \sqrt{2}$. Let $\phi_I$ be the unique cylinder such that $g^n (x,y) \in \phi_I (F)$ and $|I|=k(n)-p-q$. Then for any other $\phi_J$ such that $|J|=k(n)-p-q$ we have $d( \phi_I (F), \phi_J (F)) \geq \lambda^{k(n)-p-q} \rho$. This shows that $d( g^n (x,y), \phi_J (F) ) > \lambda^{k(n)-q}$ since otherwise
\begin{equation*}
\rho \lambda^{k(n)-p-q} \leq d( \phi_I (F), \phi_J (F)) \leq \lambda^{k(n)-q} \Rightarrow \rho \lambda^{-p} \leq 1
\end{equation*} 
contradicting the choice of $p$. It follows that for every $\phi_J$ such that $|J|=k(n)-p-q$ and $J\neq I$, and for every $(w,z) \in \phi_J (F)$,
\begin{equation*}
|| \lambda^{-k(n)} \cdot ( (w,z)-g^n (x,y) )|| = \lambda^{-k(n)} d( (w,z), g^n (x,y) ) \geq  \lambda^{-k(n)} d( \phi_J (F), g^n (x,y)) 
\end{equation*}
\begin{equation*}
> \lambda^{-k(n)} \cdot \lambda^{k(n)-q} = \lambda^{-q} > \sqrt{2}.
\end{equation*}
So, $\lambda^{-k(n)} \cdot ( (w,z)-g^n (x,y)) \notin Q$. Thus,
\begin{equation*}
M_n = \left( \lambda^{-k(n)} \cdot [g^n (F) - g^n (x,y)] \right) \cap Q \subseteq \left( \lambda^{-k(n)} \cdot [F - g^n (x,y)] \right) \cap Q
\end{equation*}
\begin{equation*}
\subseteq \left( \lambda^{-k(n)} \cdot [\phi_I (F) - g^n (x,y)] \right) \cap Q \subseteq \lambda^{-p-q} F +t_n, \quad \text{for some } t_n \in B(0,\lambda^{-q-p}).
\end{equation*}
This yields the claim, taking $b=p+q$.
\end{proof}

\textbf{Step 4 - decomposing the slice $F^x$} Now, find a subsequence $n_l$ of $n$ such that $\lbrace n_l \frac{\log \gamma}{\log \lambda} \rbrace$ converges, where $\lbrace x \rbrace = x - [x]$.  Using the sequence of natural numbers $k(n_l) - n_l$ (recall that $k(n)$ was defined in step 3) find a subsequence that satisfies the conclusion of Corollary \ref{Corollary} (Recall that $P_1 (F)=K$ has the WSC by Step 2). Assume without the loss of generality that this already happens for the the sequence  $k(n_l) - n_l$. So, there is some $k\in \mathbb{N}$ such that:
\begin{enumerate}
\item $k = \lim p(k(n_l) - n_{l}, x)$.

\item For every $1 \leq i \leq k$, $\lim_{l\rightarrow \infty} S_i ^{k(n_l) -n_{l}}$ exists and equals a compact set $S_i \subseteq F^x$.

\item $F^x = \bigcup_{i=1} ^k S_i$.
\end{enumerate}

Recall that $F^x$ supports a generic (in the sense of equation \eqref{DC}) conditional measure $\mu_{[x]}$ by step 1. Then by the third item above there is some $1\leq i \leq k$ such that $\mu_{[x]} (S_i) \geq \frac{1}{k}>0$. From this point forward, we fix this $i$. Denote by $\alpha$ the limit of $\lbrace n_l\frac{\log \gamma}{\log \lambda} \rbrace$.

\textbf{Step 5 - finding subsets of $M_{n_l}$} Recall the set $M_n$ from \eqref{The set Mn} and the sequence $k_n = [n \frac{\log \gamma}{\log \lambda}]$ were defined in step 3. For $a,b\in \mathbb{R}$ define
\begin{equation*}
\diag (a,b) := \begin{pmatrix}
a & 0\\
0 & b \\
\end{pmatrix}.
\end{equation*}
We now calculate for this subsequence $n_l$:
\begin{equation*}
M_{n_l} = \left( \diag (\lambda^{-k(n_l)},\lambda^{-k(n_l)}) \cdot ( \diag(\lambda^{n_l},\gamma^{n_l})\cdot  F - \diag(\lambda^{n_l},\gamma^{n_l}) \cdot (x,y)) \right)\cap Q =
\end{equation*}
\begin{equation*}
 = \left( \diag( \lambda^{n_l-k(n_l)}, \lambda^{\lbrace n_l \frac{\log \gamma}{\log \lambda} \rbrace})   \cdot (F - (x,y)) \right)\cap Q \supseteq 
\end{equation*}
\begin{equation} \label{Miniset phase set}
\diag( \lambda^{n_l-k(n_l)}, \lambda^{\lbrace n_l \frac{\log \gamma}{\log \lambda} \rbrace})   \cdot  \lbrace (c,d)-(x,y): (c,d)\in F , |c-x | \leq \lambda^{k(n_l)-n_l} \rbrace
\end{equation}
where we can omit the intersection with $Q$ in the equation above since  $F\subseteq [0,1]^2$ so that $|d-y| \leq 1$ for every $d\in P_2 (F)$, so the set above in contained in $Q=[-1,1]^2$.

Recall the $i$ we fixed in the end of Step 4, and that we denoted the projected IFS $P_1 \Phi$ by $\Psi$. Let 
\begin{equation*}
\psi_{I^{k(n_l)-n_l} _i} \in \Psi(k(n_l)-n_l,x) =  \lbrace \psi_I : |I|=k(n_l) -n_l, x\in \psi_I (K) \rbrace.
\end{equation*}
the element corresponding to $i$ of generation $k(n_l)-n_l$. Then since $x\in \psi_{I^{k(n_l)-n_l} _i} (K)$ and the diameter of $\psi_{I^{k(n_l)-n_l} _i} (K)$ is $\lambda^{k(l)-n_l}$ (since $\text{diam} (K)\leq 1$) we see that the set in equation \eqref{Miniset phase set} includes the set
\begin{equation*} 
\diag( \lambda^{n_l-k(n_l)}, \lambda^{\lbrace n_l \frac{\log \gamma}{\log \lambda} \rbrace})   \cdot  \lbrace (c,d)-(x,y): (c,d)\in F , c\in \psi_{I^{k(n_l)-n_l} _i} (K) \rbrace.
\end{equation*}
The above set contains the set 
\begin{equation} \label{Another phase set}
\diag( \lambda^{n_l-k(n_l)}, \lambda^{\lbrace n_l \frac{\log \gamma}{\log \lambda} \rbrace})   \cdot  \lbrace (c,d)-(x,y): (c,d)\in F , c\in \psi_{I^{k(n_l)-n_l} _i} (K), d\in P_2 (F^c \cap S^{k(n_l)-n_l} _i ) \rbrace,
\end{equation}
where the partial slice $S^{k(n_l)-n_l} _i$ was defined in \eqref{Equation approx slice}. Thus, $M_{n_l}$ contains the set in Equation \eqref{Another phase set} for every $l$. 

Note that for every $c\in \psi_{I^{k(n_l)-n_l} _i} (K)$, $F^c \cap S^{k(n_l)-n_l} _i  \neq \emptyset$. This follows by Lemma \ref{Lemma approx slice are uniform}. 

\textbf{Step 6 - A product set approximating a subset of $M_{n_l}$} We now claim that the Hausdorff distance between the set in equation \eqref{Another phase set} and the set 
\begin{equation} \label{New phase set}
\diag( \lambda^{n_l-k(n_l)}, \lambda^{\lbrace n_l \frac{\log \gamma}{\log \lambda} \rbrace})   \cdot (\psi_{I^{k(n_l)-n_l} _i} (K) -x)\times (P_2 (S^{k(n_l)-n_l} _i) -y)  
\end{equation}
is less than $2 \cdot \lambda^{k(n_l)-n_l}\cdot \text{diam}(F)$. 

Indeed, it is clear that the set in equation \eqref{Another phase set} is included in the set in equation \eqref{New phase set}. On the other hand, let $(c,d) \in \psi_{I^{k(n_l)-n_l} _i} (K)\times P_2 (S^{k(n_l)-n_l} _i)$. Then there is some $w\in P_1 (F)$ and some cylinder  $\phi_J (F)$ appearing  in the union $S^{k(n_l)-n_l} _i$ such that $(w,d)\in \phi_J (F)$. By Lemma \ref{Lemma approx slice are uniform},  the set $\phi_{J} (F)$ intersects $F^c$, so there is some $(c,d')\in \phi_J(F)$, and in particular $d' \in P_2 (F^c \cap S^{k(n_l)-n_l} _i )$. Since the diameter of $\phi_{J} (F)$ is at most $2 \lambda^{k(n_l)-n_l}\cdot \text{diam}(F)$, we see that
\begin{equation*}
|| \diag( \lambda^{n_l-k(n_l)}, \lambda^{\lbrace n_l \frac{\log \gamma}{\log \lambda} \rbrace})   \cdot \left( (c,d)-(x,y) \right) - \diag( \lambda^{n_l-k(n_l)}, \lambda^{\lbrace n_l \frac{\log \gamma}{\log \lambda} \rbrace})   \cdot \left( (c,d')-(x,y) \right)|| 
\end{equation*}
\begin{equation*}
\leq \lambda^{\lbrace n_l \frac{\log \gamma}{\log \lambda} \rbrace} | d-d'| \leq 2 \cdot  \text{diam}(\phi_J (F)) \leq 2\cdot   \lambda^{k(n_l)-n_l}\cdot \text{diam}(F).
\end{equation*}
This proves the asserted bound on the Hausdorff distance between the two sets, and  the claim follows.

\textbf{Step 7 - Taking limits} By the claim in step 6, the sequence of sets defined  in equation \eqref{Another phase set} (for $l\in \mathbb{N}$) has the same accumulation points in the Hausdorff metric as the sequence of  sets defined in equation \eqref{New phase set}. Thus, any limit in the Hausdorff metric of a sub-sequence of the sequence of sets defined in \eqref{New phase set} is contained in a limit set of $M_{n_l}$, by the arguments of Step 5 (see the last two lines of Step 5).

Now, note that the sets in \eqref{New phase set} are in fact equal to the sets
\begin{equation*}
( K +t_l) \times \lambda^{\lbrace n_l \frac{\log \gamma}{\log \lambda} \rbrace} ( P_2 (S^{k(n_l)-n_l} _i) +t_l ') , \quad t_l, t_l ' \in [-1,1], \quad \forall l\in \mathbb{N}.
\end{equation*}
Also note that these are subsets of $Q$. Thus, any limit of a subsequence of the sequence of sets above is contained within a limit of the sequence $M_n$.

\textbf{Step 8 - Proof of Proposition \ref{Proposition}} combining the results of steps 7 and 3, we see that: on the one hand, by Step 3 (Claim \ref{Claim homogneous}), every limit of the sequence of sets $M_{n_l}$ is contained within $\lambda^{-b} F +t$ for some fixed $b\in \mathbb{N}$ and $t\in B(0,\lambda^{-b})$.

Now, recall from Step 4 that $\lbrace n_l \frac{\log \gamma}{\log \lambda} \rbrace \rightarrow \alpha$. Also, from Step 4 part 2 we have
\begin{equation*}
\lim_{l\rightarrow \infty} S_i ^{k(n_l) - n_l} = S
\end{equation*}
and since $P_2$ is continuous in the Hausdorff metric (see e.g. Proposition \ref{Proposition - Hausdorff metric0} part 3) we have
\begin{equation*}
\lim_{l\rightarrow \infty} P_2(S_i ^{k(n_l) - n_l}) = P_2 (S).
\end{equation*}

So, on the other hand, by Step 7,  every accumulation point in the Hausdorff metric of the sequence of  sets 
\begin{equation*}
( K +t_l) \times \lambda^{\lbrace n_l \frac{\log \gamma}{\log \lambda} \rbrace} ( P_2 (S^{k(n_l)-n_l} _i) +t_l '), \quad t_l, t_l ' \in [-1,1], \quad l\in \mathbb{N}
\end{equation*}
is contained within a limit of $M_{n_l}$ (and is also contained within $Q$).  Every such limit set has the form
\begin{equation*}
\begin{pmatrix}
1 & 0 \\
0 & \alpha
\end{pmatrix}
(K \times P_2 (S_i)) + t'', t''\in Q.
\end{equation*}

Thus, for some  $\alpha\in [\gamma,1]$ and $b\in \mathbb{N}$ there is a $t'$ in the difference set $[-1,1]^2 - B(0,\lambda^{-b})$ such that
\begin{equation*}
\begin{pmatrix}
1 & 0 \\
0 & \alpha
\end{pmatrix}
(K \times P_2 (S_i)) + t' \subseteq \lambda^{-b} F
\end{equation*}
So
\begin{equation} \label{New embedding}
\begin{pmatrix}
\lambda^b & 0 \\
0 & \lambda^b \alpha
\end{pmatrix}
(K \times P_2 (S_i)) + t'' \subseteq  F, \quad t'' \in [-\lambda^b, \lambda^b]^2 - B(0,1).
\end{equation}
Note that the set $K \times P_2 (S_i)$ supports  the product measure $P_1 \mu \times P_2( \mu_{[x]}|_{S_i})$ introduced in step 1 (recall that that $\mu_{[x]} ( S_i) >0$ by our choice of $i$, see step 4). 
$$ $$
This concludes the proof of Proposition \ref{Proposition}. We proceed to prove part (2) of Theorem \ref{Theorem embeddings}, continuing from where we stopped.

$$ $$ 

\textbf{Step 9 - The set $F'$} Denote the set on the left hand side of \eqref{New embedding} by $F'$, so
\begin{equation*}
F' = \begin{pmatrix}
\lambda^b & 0 \\
0 & \lambda^b \alpha
\end{pmatrix}
(K \times P_2 (S_i)) + t'' .
\end{equation*}
Since $F' \subseteq F$, we have $g(F') \subseteq g(F) \subseteq F$, for the same self embedding $g$ we began with  (after the change of coordinates we preformed in the beginning of this section).

\textbf{Step 10 - The sets $T_n$} We now repeat the same argument we used to prove Proposition \ref{Proposition}, using this new set $F'$. Let $(x,y)\in F'$. Recall that $k_n = [n \frac{\log \gamma}{\log \lambda}]$. Denote
\begin{equation*}
T_n =  \left( \lambda^{-k(n)} \cdot [g^n (F') - g^n (x,y)] \right) \cap Q.
\end{equation*}
Note that $F' \subseteq F$ and for every $n$, $g^n (F) \subseteq F$, so $g^n (F') \subseteq g^n (F) \subseteq F$. In particular, $g^n (x,y) \in F$. Thus, by essentially the argument proving Claim \ref{Claim homogneous} in Step 3, we see that for the same $b\in \mathbb{N}$ from step 3, for every $n$, 
\begin{equation*}
T_n \subseteq \lambda^{-b} F +t_n
\end{equation*}
for some $t_n \in B(0,\lambda^{-b})$. Thus, every accumulation point of the sequence of sets $T_n$ is contained within $\lambda^{-b} F +t'$ for some $t' \in B(0,\lambda^{-b})$. 

\textbf{Step 11 - an assumption towards a contradiction} Assume towards a contradiction that $\frac{\log \gamma}{\log \lambda} \notin \mathbb{Q}$. Let $\beta \in [\gamma,1]$. Find some subsequence so that $\lbrace n_l \frac{\log \gamma}{\log \lambda} \rbrace$ converges to $\beta$. Without the loss of generality, assume this happens already for the sequence $n$.  

\textbf{Step 12 - subsets of $T_n$} Let $(x',y') \in K\times P_2 (S_i)$ be such that $(\lambda^b x', \lambda^b \alpha y')+t'' = (x,y)$ (the same $(x,y)$ from Step 10). Then
\begin{equation*}
T_n = \left( \lambda^{-k(n)} \cdot [g^n (F') - g^n (x,y)] \right) \cap Q = \left( \lambda^{-k(n)} \cdot [ \diag(\lambda^n, \gamma^n)\cdot  (F') - \diag(\lambda^n, \gamma^n) \cdot 
 (x,y)] \right) \cap Q 
\end{equation*}
\begin{equation*}
=\left( \diag(\lambda^{n-k(n)}, \lambda^{ \lbrace n \frac{\log \gamma}{\log \lambda} \rbrace})  
 \cdot [F' - (x,y)] \right) \cap Q 
\end{equation*}
\begin{equation*}
=\footnote{This uses the definition of $F'$ and  that $(\lambda^b x', \lambda^b \alpha y')+t'' = (x,y)$} \left( \diag(\lambda^{n-k(n)}, \lambda^{ \lbrace n \frac{\log \gamma}{\log \lambda} \rbrace})  
 \cdot [\diag( \lambda^b, \lambda^b \alpha)\cdot 
(K \times P_2 (S_i)) + t'' - (\diag( \lambda^b, \lambda^b \alpha)\cdot 
(x',y') + t'')] \right) \cap Q 
\end{equation*}
\begin{equation} \label{Phase set new iteration}
= \left( \diag(\lambda^{n-k(n)+b}, \alpha \lambda^{ \lbrace n \frac{\log \gamma}{\log \lambda} \rbrace+b})  
 \cdot [ K \times P_2 (S_i)  - (x',y')] \right) \cap Q 
\end{equation}
Pick any $\psi_I \in \Psi(-n+k(n)-b,x')$, i.e  any cylinder of $\Psi=P_1 \Phi$ such that $|I|=k(n)-n-b$ and $x' \in \psi_I (K)$. Then the set in equation \eqref{Phase set new iteration} includes the set
\begin{equation*}
 \diag(\lambda^{n-k(n)+b}, \alpha \lambda^{ \lbrace n \frac{\log \gamma}{\log \lambda} \rbrace+b})
 \cdot [ \psi_I (K) \times P_2 (S_i)  - (x',y')] 
\end{equation*}
\begin{equation} \label{Sequence of sets}
=  \diag(1, \alpha \lambda^{ \lbrace n \frac{\log \gamma}{\log \lambda} \rbrace+b}) 
 \cdot [ (K+t_n) \times P_2 (S_i)  - (0,y')]  , \quad t_n \in [-1,1].
\end{equation}
Note that both the sets above are contained within $Q$. Taking a convergent subsequence of the sequence of  sets in \eqref{Sequence of sets}, we see that there exists $v\in [-1,1]^2$ such that it equals
\begin{equation*}
  \begin{pmatrix}
1 & 0 \\
0 & \alpha \lambda^{ \beta+b}
\end{pmatrix}
 \cdot  K \times P_2 (S_i)  +v, \quad v\in Q
\end{equation*}
recalling that $\lbrace n_l \frac{\log \gamma}{\log \lambda} \rbrace$ converges to $\beta$ (see step 11).

\textbf{Step 13 - A large set of embeddings of $K\times P_2 (S_i)$ into $F$} Taking tally of steps 10 and 12, we see that for every $\beta \in [\gamma,1]$ there is some $v \in [-1,1]^2$ and $t' \in B(0,\lambda^{-b})$ such that 
\begin{equation*}
\begin{pmatrix}
1 & 0 \\
0 & \alpha \lambda^{ \beta+b}
\end{pmatrix}
 \cdot  K \times P_2 (S_i)  +v \subseteq \lambda^{-b} F +t'.
\end{equation*}

\textbf{Step 14 - Proof of Part (2) of Theorem \ref{Theorem embeddings} for self homothetic sets} Finally, recall that the set $K \times P_2 (S_i)$ supports  the product measure $P_1 \mu \times P_2 (\mu_{[x]}|_{S_i})$ introduced in step 1 ($\mu_{[x]} (S_i) >0$ by our construction) . By step 13,  it follows that there is a compact set $X$ of embeddings of $K \times P_2 (S_i)$ into $F$ that has dimension at least $1$. Also, $P_1 (F) \times P_2 (S_i)$ supports the measure $P_1 \mu \times P_2 (\mu_{[x]} |_{S_i})$, that has dimension $\dim \mu = \dim F$ (see the remark at the end of Step 1).

Moreover, the measure $P_1 \mu \times P_2 (\mu_{[x]} |_{S_i})$ belongs to the family $\text{Pro}(F)$ introduced in \eqref{Equation for Pro(F)}. We will show that this measure is not supported on any affine line. Therefore, we may apply the inverse Theorem for product measures Theorem \ref{Theorem Inverse Theorem for product measures - finite group} for the measure $P_1 \mu \times P_2 (\mu_{[x]} |_{S_i})$ and some measure $\nu \in P(X)$ of entropy dimension at least $1$. Thus, the convolution measure $\nu.(P_1 \mu \times P_2 (\mu_{[x]} |_{S_i}))$ has upper entropy dimension at least $\dim_H F + \delta$ for some $\delta>0$. But it is supported on $F$ that has box dimension $\dim_H F$, a contradiction.

In order to apply Theorem \ref{Theorem Inverse Theorem for product measures - finite group}, we must verify that $P_1 \mu \times P_2 (\mu_{[x]} |_{S_i})$ is not supported on any affine line. This follows since both measures are not atomic:  
\begin{itemize}
\item   $P_1 \mu$ is not atomic, since otherwise $\mu$ gives positive mass to an affine line and therefore $F$ must sit on this affine line (see Lemma 2.6 in \cite{elekes2010self}), a contradiction.

\item   $\mu_{[x]}$ is not atomic. Otherwise, since $\mu_{[x]}$ is exact dimensional, $\dim \mu_{[x]}=0$. Recall that we chose $\mu_{[x]}$ to be generic with respect to the dimension conservation formula \eqref{DC}. Therefore, $\dim \mu_{[x]}=0$ implies that $\dim P_1 (F)=\dim F$. Now,  $P_1 (F)$ has the WSC, so $H^{s} (P_1 (F)) >0$ for $s=\dim F = \dim P_1 (F)$ (see e.g. \cite{Zerner1996weak}). However, $\lambda^m (P_1 (F),0)+t \subset F$, so $P_1 (F)$ admits an affine homothetic embedding into a horizontal slice of $F$. Since $\mu$ is the restriction of $H^s$ to $F$, $\mu$ gives this horizontal slice positive mass. Again, this is a contradiction, as we are assuming $F$ does not sit on an affine line.

\end{itemize}

\textbf{Step 15 - the case of $1<|G_\Phi|<\infty$} The case when $F$ has $1 < |G_\Phi| < \infty$ follows by a similar consideration. We denote by $g$ the self embedding of $F$. We may assume, as is explained in the beginning of this section, that the eigenvalues of $g$ are $0<\gamma<\lambda$. 
\begin{itemize}

\item First, we preform a coordinate change so that $P_1 (F)$ can be homothetically embdded into $F$, in a similar manner to the self homothetic case:  Let $U \in GL(\mathbb{R}^2)$ be the matrix diagonalizing the linear part of $g$.  We again  consider the set $U^{-1} (F)$ that is generated by the conjugated IFS $U^{-1} \circ \Phi \circ U$. We work with the metric $||v-w||_* := || U(v)-U(w)||_2$ (induced by the norm $||v||_* = ||U(v)||_2$ which is equivalent to the $||\cdot ||_2$) , so that the maps in this conjugated IFS are similarities with contraction ratio $\lambda$. Moreover, since $G_\Phi$ is finite, then so is the group generated by the orthogonal parts (orthogonal with respect to the new norm)  of the conjugated IFS $U^{-1} \circ \Phi \circ U$.

\item  With this change of coordinates, since our original map was
\begin{equation*}
g(z)= U \begin{pmatrix}
\lambda & 0 \\
0 & \gamma
\end{pmatrix}\cdot U^{-1} z +t'
\end{equation*}
we now have that the affine map $g(x,y)=(\lambda\cdot x, \gamma\cdot y)+t$ for $t=U^{-1} (t')$ is our self embedding of $F$ (after the change of coordinates), that is $g(F)\subseteq F$.

\item Let $\delta (F)$ be as in Theorem \ref{Theorem Inverse Theorem for product measures - finite group}. Find a  subset $F' \subseteq F$ such that $\dim F < \dim F' + \delta$ such that $F'$ has a generating self-homothetic IFS with the SSC and uniform contraction $\lambda^p$ for some $p\in \mathbb{N}$, and such that $F'$ does not sit on an affine line. Assume without the loss of generality that $p=1$.

\item Preform Step 1 on this set $F'$ and produce a generic $x \in P_1 (F')$ for the dimension conservation formula \eqref{DC} with respect to the natural self similar measure on $F'$.

\item Define $M_n$ as in \eqref{The set Mn} in Step 2 for this $F'$. Then we have:
\begin{Claim} 
Let $M\subseteq Q$ be an accumulation point of the sequence $M_n$ with respect to the Hausdorff metric. Then  $M \subseteq \lambda^{-b} F +t$,  for $b\in [1,m] \cap \mathbb{N}$, where $m$ depends only on $F$. and some $t\in B(0,\lambda^{-b})$. 
\end{Claim}

The proof is essentially the same as the proof of Claim \ref{Claim homogneous} with minor adjustments: for example (in the notation of claim \ref{Claim homogneous}), for the choice of $p$ note that $\rho$ is still positive but is possibly  different  according to this new metric, and we find $q$ so that $\lambda^{-q} > \max_{v\in Q} ||v||_*$. Also, at the very end of the proof we find that $M_n \subseteq \lambda^{-p-q} \cdot O(F) +t_n$ for $t_n \in B(0, \lambda^{-p-q})$ for $O$ that is the orthogonal part of some cylinder of $F$. As $G_\Phi$ is finite, there is some global $m\in \mathbb{N}$ such that $O$ is the orthogonal part of a cylinder of generation $1\leq n \leq m$. Thus, $O(F) \subseteq \lambda^{-n} \cdot F+t$, for some $t\in \mathbb{R}^2$, which yields this version of the Claim.

\item Preform the procedure carried out in Steps 3-13 to find a product set (related to the self homothetic $F'$) that embeds into $F$ in many ways (i.e. such that the set of embeddings has dimension at least $1$). There are minor modification that need to be made, mainly due to the fact that the distance between two $(y,0),(y',0)\in P_1 (F)$ is $||(y,0)-(y',0)||_*$ and in particular  is not necessarily equal to $|y-y'|$.   Note that $P_1 (F')$ has the WSC (we'll address this point later).

Specifically, for the arguments in step 5 (e.g. between equations \eqref{Miniset phase set} and \eqref{Another phase set}) to work we use the fact that there is some $k$ such that $\lambda^k \cdot ||U^{-1}|| \leq 1$. We then choose the generation of the cylinder we take after equation \eqref{Miniset phase set} to be $k(n_l)-n_l +k$ instead of $k(n_l)-n_l$, and take the corresponding partial approximate slice $S^{k(n_l)-n_l+k} _i$ in equation \eqref{Another phase set}. Let $y,y' \in \psi_{I^{k(n_l)-n_l+k} _i} (K)$ be such that
\begin{equation*}
\text{diam}_{||\cdot ||_2} (\psi_{I^{k(n_l)-n_l+k} _i} (K)) = |y-y'|=  ||(y,0)-(y',0)||_2.
\end{equation*}
Then
\begin{equation*}
|y-y'|=|| (y,0) - (y',0)||_2 \leq ||U^{-1}|| \cdot ||(y,0)-(y',0)||_*  \leq  ||U^{-1}||\cdot  \text{diam}_{||\cdot ||_*} (\psi_{I^{k(n_l)-n_l+k} _i} (K))
\end{equation*}
\begin{equation*}
 =  \lambda^{k(n_l)-n_l +k}\cdot ||U^{-1}|| \leq \lambda^{k(n_l)-n_l}.
\end{equation*}
Thus, $\lambda^{-k(n_l)+n_l} \cdot (\psi_{I^{k(n_l)-n_l+k} _i} (K) -x )\subseteq [-1,1]$, and the proof follows through (with the minor change  that in equation \eqref{New embedding}, in the first coordinate of the matrix we should have $\lambda^{b+k}$ instead of $\lambda^b$).

\item  Let  $X$ be the compact set in $G_2$ which is the finite union of the following sets: for our fixed $\alpha, \gamma \in (0,1]$ and $b\in \mathbb{N}\cap [1,m]$ and $k\in \mathbb{N}$ as in the previous two items,
\begin{equation*}
X_b = \lbrace g\in \mathcal{G}_2 : g(z)=Az+t, A =   \diag(\lambda^{k+b}, \alpha\lambda^{\beta+2b}), \beta \in [\gamma,1], t \in [-\lambda^b,\lambda^b]^2 - B(0,1) \rbrace
\end{equation*}

\item Assume for the moment that the product measure we got this way (that is supported on our product set from a previous item) is not supported on an affine line. Apply the inverse Theorem \ref{Theorem Inverse Theorem for product measures - finite group} (with the set $X$ as in the previous item) to obtain a measure on $F$ that has entropy dimension at least $\dim_H F' +\delta$. This forms a lower bound on the box dimension of $F$. But $F$'s box dimension is equal to $\dim_H F$ and $\dim_H F < \dim_H F' + \delta$, a contradiction.
\end{itemize}

A few remarks are in order. First, note that when moving to a subset $F' \subseteq F$ we lose the fact that an affine image $P_1 (F')$ is contained within $F'$, so Lemma \ref{WSC for projection} cannot be applied as is to see that $P_1 (F')$ has the WSC (which is crucial for the proof). However,  $\lambda^m (P_1 (F),0) + t \subset F$, and so  $\lambda^m (P_1 (F'),0)+t \subset F$. One may now apply the remark following Lemma \ref{WSC for projection} to see that $P_1 (F')$ does indeed have the WSC.

We also need to verify that our chosen product measures are not supported on affine lines. This follows essentially from the same reasons as in Step 14, with the following modification. First, we may chose $F' \subseteq F$ so that it does not sit on an affine line (since $F$ does not sit on an affine line). Thus, if $\mu'$ is the natural self similar measure (that has maximal dimension) on  $F'$, $P_1 \mu'$ cannot be atomic. 

Secondly,   by the remark following Lemma \ref{WSC for projection}, we see that the Assouad dimension $\dim_A P_1 (F) <1$ (since for every micorset of $P_1 (F)$, some affine homothetic image of it is contained within a horizontal slice of $F$). Therefore, by (\cite{FraserOrponen2015assouad}, Theorem 2.3), $P_1 (F)$ has  positive Hausdorff measure in its dimension. Thus, as in the self homothetic case, if $\mu$ is the natural self similar measure on $F$, then almost surely one has $0 < \dim \mu_{[x]}$. Therefore, if  $\dim F - \dim F'$ is small enough we can ensure (via the dimension conservation formula applied to both the self similar measures\footnote{Note that for the self similar measure $\mu$ the dimension conservation formula is valid by \cite{Falconer2014Jin}.} $\mu$ and $\mu'$) that  $\dim \mu' _{[x]}$ is generically positive. Thus, the corresponding product measures are not supported on lines, and we may apply Theorem \ref{Theorem Inverse Theorem for product measures - finite group}.

\section{Inverse Theorems for entropy} \label{Section inverse Theorems for entropy}
The objective of this section is to state the version of the inverse Theorem we are using, explain how it is related to Theorems \ref{Theorem Inverse Theorem main application}, and prove Theorem \ref{Theorem Inverse Theorem for product measures - finite group}. 

\subsection{Preliminaries}

Let us recall some of the definitions from (\cite{hochman2015self}, Section 2.2). For a measure $\mu \in P(\mathbb{R}^d)$ and $m\in \mathbb{N}$ we define
\begin{equation*}
H_m (\mu) = \frac{1}{m} H(\mu,D_m)
\end{equation*}
where $D_m$ is the generation $m$ dyadic partition of $\mathbb{R}^d$ and $H(\mu,D_m)$ is the Shannon entropy  (all these concepts were introduced in section \ref{Section entropy}). Next, let $V$ be a linear subspace of $\mathbb{R}^2$, and let $\mu \in P(\mathbb{R}^2)$. We say that a measure $\mu$ is $(V,m,\epsilon)$ saturated if
\begin{equation*}
H_m (\mu) \geq \dim V + H_m(P_{V^\perp} (\mu)) - \epsilon.
\end{equation*}
Also, for $A\subseteq \mathbb{R}^d$ and $\epsilon>0$, denote the $\epsilon$-neighbourhood of $A$ by $A_{(\epsilon)} = \lbrace x \in \mathbb{R}^d : d(x,A) < \epsilon\rbrace$.
Let $V \leq  \mathbb{R}^d$ be a linear subspace and $\epsilon > 0$. A measure $\mu$ is
$(V, \epsilon)$-concentrated if there is a translate $W$ of $V$ such that $\mu(W_{(\epsilon)}) \geq 1 - \epsilon$.

Let $x\in \mathbb{R}^d$. Let $D_n (x) \in D_n$   denote the unique level-n dyadic cell containing the point $x$. For $D \in D_n$ let $T_D : \mathbb{R}^d \rightarrow \mathbb{R}^d$ be the unique homothety mapping $D$ to $[0, 1)^d$. Recall that if $\mu \in  P(\mathbb{R}^d)$ then $T_D \mu$ is the push-forward of $\mu$ through $T_D$.

Now, for $\mu \in  P(\mathbb{R}^d)$) and a dyadic cell $D$ with $\mu (D) > 0$, the (raw)
D-component of $\mu$ is
\begin{equation*}
\mu_D = \frac{1}{\mu (D)} \mu|_D
\end{equation*}
and the (rescaled) D-component is
\begin{equation*}
\mu^D = \frac{1}{\mu (D)} T_D \mu|_D
\end{equation*}
For $x\in \mathbb{R}^d$ with $\mu (D(x)) > 0$ we write
\begin{equation*}
\mu_{x,n} = \mu_{D_n (x)}
\end{equation*}
\begin{equation*}
\mu^{x,n} = \mu^{D_n (x)}
\end{equation*}
These measures, as $x$ ranges over all possible values for which $\mu (D(x)) > 0$,
are called the level-$n$ components of $\mu$.

A random level-$n$ component, raw or rescaled, is the random measure $\mu_D$
or $\mu^D$, respectively, obtained by choosing $D \in D_n$ with probability $\mu(D)$;
equivalently, this is the random measure $\mu_{x,n}$ or $\mu^{x,n}$ with $x$ chosen
according to $\mu$.  For a finite set $I \subseteq \mathbb{N}$, a random level-$I$ component, raw or rescaled, is chosen by first choosing $n \in I$ uniformly, and then (conditionally independently on the choice of n) choosing a raw or rescaled level-n component,
respectively.

When the symbols $\mu^{x,i}, \mu_{x,i}$ appear inside an expression $\mathbb{P}$, they will always denote random variables drawn according to the component distributions defined above. The range of i will be specified as needed.

\begin{Lemma} \label{Lemma entropy of component measures}
Let $\mu,\nu \in P([0,1])$. Then the following statements are valid:
\begin{enumerate}
\item Any component (raw or rescaled) measure of $\mu \times \nu$ is a product of components of $\mu$ and $\nu$, respectively.

\item For any $m\in \mathbb{N}$,
\begin{equation*}
H_m (\mu \times \nu) = H_m (\mu) + H_m (\nu).
\end{equation*}

\item Let $V\in G(2,1)$ and let $m\in \mathbb{N}$. Then one of the following holds, with an error term independent of $V$:
\begin{equation*}
H_m(P_{V^\perp} (\mu \times \nu)) \geq H_m (\mu) + O(\frac{1}{m})
\end{equation*}
or
\begin{equation*}
H_m(P_{V^\perp} (\mu \times \nu)) \geq H_m (\nu) + O(\frac{1}{m})
\end{equation*}
\end{enumerate}
\end{Lemma}

\begin{proof}
Items (1) and (2) follow by a direct calculation, using the fact that  the dyadic partition of $\mathbb{R}^2$ is the product of the dyadic partitions of $\mathbb{R}$. We omit the details.

Next, let $t<0$ and $\mu \in P(\mathbb{R})$ and let $n\in \mathbb{N}$. Define the map $S_t :\mathbb{R} \rightarrow \mathbb{R}$ by $S_t (x) = 2^t x$. We first claim that
\begin{equation} \label{Equation St}
H(S_t \mu, D_n) = H(\mu, D_n) + O(t) + O(1).
\end{equation}
Indeed, by equation (6) in \cite{hochman2016some} we have
\begin{equation*}
H(S_t \mu, D_n) = H(\mu, D_{n+t}) + O(1)
\end{equation*}
where the $t$ on the right hand side is understood as $\lceil t \rceil$. Now,
\begin{equation*}
H(\mu, D_{n+t}) + O(1) = H(\mu,D_n \vee D_{n+t}) - H(\mu,D_n | D_{n+t}) + O(1) 
\end{equation*}
\begin{equation*}
=H(\mu, D_n) - O(t) + O(1)
\end{equation*}
where we use the facts that $t<0$ so $D_n$ refines $D_{n+t}$, and every $D\in D_{n+t}$ intersects  $2^{|t|}$ elements of $D_n$. 

Turning our attention to part (3), let $(a,b)\in S^1\cap V$, so $||(a,b)||=1$. Thus, 
\begin{equation*}
H_m(P_{V^\perp} (\mu \times \nu)) = H_m (S_{\log a} \mu \ast S_{\log b} \nu) + O(\frac{1}{m})
\end{equation*}
where we change coordinates to get a measure supported on $\mathbb{R}$, which gives rise to the term $O(\frac{1}{m})$. Now, according to Lemma 6 of \cite{hochman2016some} we have
\begin{equation*}
H_m (S_{\log a} \mu \ast S_{\log b} \nu) \geq \max \lbrace H_m (S_{\log a} \mu) , H_m (S_{\log b} \nu) \rbrace - O(\frac{1}{m}).
\end{equation*} 

Now, $||(a,b)||=1$ hence either $a \geq \frac{1}{\sqrt{2}}$ or $b \geq \frac{1}{\sqrt{2}}$. Suppose $a \geq \frac{1}{\sqrt{2}}$, so $| \log a | \leq \frac{1}{2}$ since we also have $a\leq 1$. Then by equation \eqref{Equation St}
\begin{equation*}
H_m (S_{\log a} \mu) = H_m(\mu) + \frac{O(|\log a|)}{m} + O(\frac{1}{m}) = H_m(\mu) + \frac{O(\frac{1}{2})}{m} + O(\frac{1}{m}) = H_m(\mu) + O(\frac{1}{m}).
\end{equation*}
If $b \geq \frac{1}{\sqrt{2}}$ we similarly obtain
\begin{equation*}
H_m (S_{\log b} \nu) = H_m(\nu) + O(\frac{1}{m}),
\end{equation*}
as required.
\end{proof}

\subsubsection{Furstenberg Galleries}
We recall the definition of a gallery in the sense of Furstenberg, which we already defined after Lemma \ref{WSC for projection}. They shall play a role in our analysis, since we shall be looking at component measures of measures, and these have supports that are mini-sets of the supports of the original measure.

Recall that if $K\subseteq [0,1]^2$ is a compact set, we say that a set  $A \subseteq [-1,1]^2$ is  a  miniset of $K$ if $A \subseteq (\gamma \cdot K + t)\cap [-1,1]^2$ for some $\gamma \geq 1, t\in \mathbb{R}$. A set $M$ is called a  microset of $K$ if $M$ is a limit in the  Hausdorff metric on subsets of $[-1,1]^2$ of minisets of $K$.   The set $K$ is called Furstenberg homogeneous if every microset of $K$ is a miniset of $K$. A family of compact sets $G$ of $[-1,1]^2$ is called a gallery if it is closed in the Hausdorff metric and for every $K\in G$, every mini-set of $K$ is in $G$.

Let $G$ be a gallery. For a set $A\in G$ and $m\in \mathbb{N}$  let 
\begin{equation*}
N_m (A) = \lbrace D \in D_m : A \cap D_m \neq \emptyset \rbrace
\end{equation*}
where $D_m$ is  the generation $m$ dyadic partition of $[-1,1]^2$. Next, let
\begin{equation*}
N_m (G) = \max_{A \in G} N_m (A).
\end{equation*}
Now, Recall from \cite{furstenberg2008ergodic} that the limit
\begin{equation} \label{Equation dimension of gallery}
\Delta (G) =  \lim_{m\rightarrow \infty} \frac{\log N_m (G)}{m} 
\end{equation}
exists and bounds from above the upper box dimension of any set in $G$. $\Delta (G)$ is called the dimension of the gallery $G$. By (\cite{furstenberg2008ergodic}, Section 5)  there exists a set $A \in G$ such that $\dim_H A \geq \Delta (G)$, and so $\dim_H A = \Delta (G)$. 

Let us now give a few examples of galleries, that shall be used in the subsequent sections:
\begin{enumerate}
\item Let $F$ be a self homothetic set with the SSC. Then $F$ is Furstenberg homogeneous.

\item More generally, let $F$ be a self similar set generated by an IFS $\Phi$ with $|G_\Phi|<\infty$ and a uniform contraction ratio. Then $F$ is Furstenberg homogeneous.

\item \label{P1(F) is gallery} Let $F$ be a self homothetic set with the SSC and uniform contraction ratio. Suppose that $\gamma\cdot (P_1 (F),0)+t \subseteq F$ for some $\gamma > 0 , t\in \mathbb{R}^2$.  Then $P_1 (F)$ is Furstenberg homogeneous, by Lemma \ref{WSC for projection}. In particular,
\begin{equation*}
G_{P_1 (F)} = \lbrace A: A \text{ is a miniset of }  P_1 (F)  \rbrace
\end{equation*}
is a gallery with dimension $\Delta (G_{P_1 (F)}) = \dim P_1 (F)$.

\item \label{Gallrey of vertical slices} Let $F$ be a self homothetic set with the SSC and a uniform contraction ratio. Then  the family 
\begin{equation*}
G_{F,e_2} = \lbrace A: A \text{ is a miniset of }  F \cap P_1 ^{-1} (x) \text{ for some } x\in \mathbb{R} \rbrace
\end{equation*}
is known to be a gallery, where $e_2 = (0,1)$. Moreover,  $\Delta (G_{F,e_1}) = \dim_H A$ for some $A\in G_{F,e_2}$, and an affine image of $A$ is contained within a vertical slice of $F$ (by the definition of the gallery). Therefore, by Corollary \ref{Corollary SSC implies dimension gap for slices}, $\Delta (G_{F,e_2}) = \dim A <1$. 

\item \label{Gallery of slices - general} More generally, let $F$ be a self similar set generated by an IFS $\Phi$  the SSC, uniform contraction ratio $\lambda$, and $|G_\Phi|<\infty$. Let $V = \text{span}(v) \in G(2,1)$. Then
\begin{equation*}
G_{F,v} = \lbrace A: A \text{ is a miniset of }  F\cap (V+t): t\in \mathbb{R}^2 \rbrace
\end{equation*}
is known to be a gallery. Moreover, as in the previous example, since the dimension of the gallery is attained by an element of the gallery, and this elements embeds as a subset of a slice of $F$, we see by Corollary \ref{Corollary SSC implies dimension gap for slices} that $\Delta( G_{F,v}) < 1$.
\end{enumerate}

\subsection{A version of the inverse Theorem for entropy}
We first quote one of the inverse Theorems proved by Hochman in \cite{hochman2015self}. 
\begin{theorem} (\cite{hochman2015self}, Corollary 2.15) \label{Theorem Inverse Theorem} Let $G < GL(\mathbb{R}^2)$ be any subgroup,  and Let $\nu \in P(G)$ and $\mu \in  P(\mathbb{R}^d)$ be measures of bounded support and suppose $\mu$ is not supported on a proper affine subspace. Then for every $\epsilon > 0$ and $m\in \mathbb{N}$ there is a $\delta > 0$ (depending on $\epsilon$, $m$ and supp $\nu$), such that for every large enough $k$ and every large enough $n$, either
\begin{equation} \label{Equation increase of entropy}
H_n(\nu .\mu ) > H_n(\mu) + \delta,
\end{equation}
or else, for an independently chosen pair of level-$k$ components $\widetilde{\mu}, \widetilde{\nu}$ of $\mu,\nu$, respectively, with probability $> 1-\epsilon$, there are subspaces $V_1, . . . , V_n < \mathbb{R}^d$ such that
\begin{equation}
\mathbb{P}_{0 \leq i \leq n} ( \widetilde{\mu}^{x,i} \text{is } (V_i,m,\epsilon) \text{-saturated }, \widetilde{\nu}^{y,i} .x \text{ is } (V_i,\epsilon) \text{concentrated} )>1-\epsilon.
\end{equation}
If in addition $\mu$ is not supported on a proper affine subspace, for $k, n$ large
enough,
\begin{equation*}
\frac{1}{n+1} \sum_{i=0} ^n \dim V_i > c H_n (\nu)
\end{equation*}
For a suitable constant $c$ (depending only on $d=2$ and the support of $\nu$).
\end{theorem}

\textbf{Remarks}
\begin{enumerate}
\item The Theorem remains valid when taking measures supported on subgroups of the affine group $G_2$ rather then the linear group $GL(\mathbb{R}^2)$ (this follows by the same argument proving Theorem \ref{Theorem Inverse Theorem} in \cite{hochman2015self}).

\item The constant $\delta$ above is uniform in the following sense: if we prescribe a compact set $X \subset G_2$ in advance, then we may produce $\delta$ that works for all measures $\nu$ supported on $X$ (when $\epsilon$ and $m$ are of course the same).

\item Theorem \ref{Theorem Inverse Theorem main application} is a consequence of Theorem \ref{Theorem Inverse Theorem}. This is essentially the proof of Theorem 1.5 in \cite{hochman2015self}, as appearing in section 6 there.  We refer the reader to \cite{hochman2015self} for more details.
\end{enumerate}

\subsection{Proof of Theorem \ref{Theorem Inverse Theorem for product measures - finite group}}

\subsubsection{A preliminary version of Theorem \ref{Theorem Inverse Theorem for product measures - finite group}}
We begin by proving a more restrictive version of Theorem \ref{Theorem Inverse Theorem for product measures - finite group}, that works for self homothetic sets. Its proof contains essentially all the key ingredients of the proof of the more general Theorem \ref{Theorem Inverse Theorem for product measures - finite group}. In fact, this version suffices for the proof of part (2) of Theorem \ref{Theorem embeddings} for self homothetic sets (it may be applied at Step 14 of Section \ref{Section part 3}).

 Recall the definition of the set of product measures $\text{Pro}(F)$ defined in \eqref{Equation for Pro(F)}.

\begin{theorem} \label{Theorem Inverse Theorem for product measures - self homo}
Let $F$ be a self similar set generated by an IFS $\Phi$ with a uniform contraction ratio $\lambda$, $G_\Phi=\lbrace \id \rbrace$ and the SSC. Assume $\gamma\cdot (P_1 (F),0)+t \subset F$ for some $\gamma >  0,t\in \mathbb{R}^2$. Let $\theta \in \text{Pro}(F)$ be such that it is not supported on any affine line, and let $\nu \in P(G_2)$ be a compactly supported measure  with $\overline{\dim}_e (\nu) \geq 1$.

Then there exist $\delta(F,\theta, \text{supp}(\nu))>0$ such that  
\begin{equation*}
\overline{\dim}_e (\nu. \theta) \geq \overline{\dim}_e (\theta) +\delta.
\end{equation*} 
\end{theorem}

\begin{proof}
Let $\mu$ be the natural self similar measure on $F$, and let $\lbrace \mu_{[x]} \rbrace$ denote its disintegration with respect to $P_1$. Then $\theta \in \text{Pro}(F)$ implies that for some $x\in DC(F,P_1)$ (recall \eqref{Equation for DC(X,P)}) there is a compact set $S\subseteq F^x$ with $\mu_{[x]} (S)>0$ such that $\theta= P_1 \mu \times P_2 (\mu_{[x]}|_{S})$. We also assume $\theta$ is not supported on any affine line. Moreover, since $\gamma\cdot (P_1 (F),0)+t \subseteq F$ and $F$ is Furstenberg homogeneous, by Lemma \ref{WSC for projection}, the self similar set $P_1 (F)$ is homogeneous in the sense of Furstenberg. In particular, the family  $G_{P_1 (F)}$ of all microsets of $P_1 (F)$ forms a gallery in the sense of Furstenberg (see Example \ref{P1(F) is gallery}). We may assume without the loss of generality that $S= \text{supp}(\mu_{[x]})$, so we omit  $S$ from $\theta$ in our notation.

Recall that the family 
\begin{equation*}
G_{F,e_2} = \lbrace A: A \text{ is a miniset of }  F \cap P_1 ^{-1} (x) \text{ for some } x\in \mathbb{R} \rbrace
\end{equation*}
also forms a gallery in the sense of Furstenberg, and $\Delta (G_{F,e_2})  <1$ (see Example \ref{Gallrey of vertical slices}). Also, since $P_1 (F)$ is Furstenberg homogeneous, $ \Delta (G_{P_1 (F)}) = \dim P_1 (F) <1$, where the latter assertion follows since an affine homothetic image of $P_1 (F)$ is contained within a horizontal slice of $F$, and using Corollary \ref{Corollary SSC implies dimension gap for slices}.

Now, let $\epsilon$ be small enough so that 
\begin{equation} \label{Equation choice of epsilon}
\max \lbrace \Delta (G_{F,e_2}) , \Delta (G_{P_1 (F)}) \rbrace < 1- \epsilon. 
\end{equation}
Let $\nu \in P(G_2)$ be a compactly supported measure with $\overline{\dim}_e (\nu) \geq 1$, and let $X$ denote the support of $\nu$. Applying Theorem \ref{Theorem Inverse Theorem}, for any $m \in \mathbb{N}$ we may produce $\delta(\epsilon,m,X)$ such that either equation \eqref{Equation increase of entropy} is satisfied, or the other option stated in Theorem \ref{Theorem Inverse Theorem} is satisfied. It suffices to find some $m\in \mathbb{N}$ such that \eqref{Equation increase of entropy} is satisfied for our product measure $\theta = P_1 \mu \times P_2 ( \mu_{[x]})$. 

Suppose towards a contradiction that for every $m\in \mathbb{N}$  equation \eqref{Equation increase of entropy} fails. Fix some  $m \in \mathbb{N}$. Then for all large enough $k$ and large $n$, we may find a linear subspace $V_m$ such that the rescaled component of a raw component of our product measure satisfies that
\begin{equation} \label{Equation failure}
\left( (P_1 \mu \times P_2 (\mu_{[x]}) )_{e,k} \right) ^{y,i}
\end{equation}
is $(\epsilon, m)$-saturated on $V_{m}$, where $1\leq i \leq n$ for some large $n$ and $\dim V_{m} >0$ (here we use the constant $c$ from Theorem \ref{Theorem Inverse Theorem}, and that $\nu$ has positive upper entropy dimension). Recall that the numbers $N_M (A)$, for $A$ in a gallery, were defined right before equation \eqref{Equation dimension of gallery}.

\begin{Lemma} \label{Lemma four options}
One of the following options must occur:
\begin{enumerate}
\item $\max_{A\in G_{F,e_2}} \log N_m (A)  \geq m\cdot (1 - O(\frac{1}{m})  - \epsilon)$

\item $\max_{A\in G_{P_1 (F)}} \log N_m (A)  \geq m\cdot (1 - O(\frac{1}{m})  - \epsilon)$. 

\item $\max_{A\in G_{F,e_2}} \log N_m (A)   \geq m\cdot (1 - \frac{\epsilon}{2})$

\item $\max_{A\in G_{P_1 (F)}} \log N_m (A)   \geq m\cdot (1 - \frac{\epsilon}{2})$ 
\end{enumerate}
where the error term $O(\frac{1}{m})$ is the same term in both (1) and (2)
\end{Lemma}

\begin{proof}
Let $V_m$ be the subspace of equation \eqref{Equation failure}. Then either $\dim V_{m}=1$  or $V_{m} = \mathbb{R}^2$. We shall see that the first first case corresponds to the first two cases above, and that the latter case corresponds to options (3) and (4) above.

Recall that we fixed $m\in \mathbb{N}$. Suppose first that $\dim V_{m}=1$. Recall (from Lemma \ref{Lemma entropy of component measures}) that a component measure of a product of component measures is a product of  component measures of its factors, so the measure from equation \eqref{Equation failure} satisfies
\begin{equation*}
\left( (P_1 \mu \times P_2 (\mu_{[x]}) )_{e,k} \right) ^{y,i} =  (P_1 \mu)^{x_1,k+i} \times (P_2 (\mu_{[x]}))^{x_2, k+i}, \text{ for some } x_1,x_2.
\end{equation*}
Also, recall that the entropy of a product measure is the sum of the entropies of each factor (again see Lemma \ref{Lemma entropy of component measures}).   Thus, we have by our assumption that this component measure is $(V_m , m, \epsilon)$-saturated on $V_m$,
\begin{equation*}
H_m ( (P_1 \mu)^{x_1,k+i} ) + H_m ((P_2 (\mu_{[x]}))^{x_2, k+i}) = H_m ( \left( (P_1 \mu \times P_2 (\mu_{[x]}) )_{e,k} \right) ^{y,i} )  
\end{equation*}
\begin{equation*}
\geq 1 + H_m (P_{V_{m} ^\perp} ( \left( (P_1 \mu \times P_2 (\mu_{[x]}) )_{e,k} \right) ^{y,i} )) -\epsilon
\end{equation*}
\begin{equation} \label{Equation saturation}
= 1+ H_m (P_{V_{m} ^\perp} ( (P_1 \mu)^{x_1,k+i} \times (P_2 (\mu_{[x]}))^{x_2, k+i}) -\epsilon.
\end{equation}

Now, by Lemma \ref{Lemma entropy of component measures} part (3), we have either
\begin{equation} \label{Equation insert}
H_m (P_{V_{m} ^\perp} ( (P_1 \mu)^{x_1,k+i} \times (P_2 (\mu_{[x]}))^{x_2, k+i}) \geq H_m ( (P_1 \mu)^{x_1,k+i}) - O(\frac{1}{m})
\end{equation}
or
\begin{equation} \label{Equation insert 2}
H_m (P_{V_{m} ^\perp} ( (P_1 \mu)^{x_1,k+i} \times (P_2 (\mu_{[x]}))^{x_2, k+i}) \geq H_m ((P_2 (\mu_{[x]}))^{x_2, k+i}) - O(\frac{1}{m})
\end{equation}
where the error term $O(\frac{1}{m})$ depends only on the dimension of the ambient space ($d=1$ in our case), and in particular does not depend on the measures convolved. 

If equation \eqref{Equation insert} is satisfied, we insert it  into equation \eqref{Equation saturation} and see that 
\begin{equation*}
H_m ( (P_1 \mu)^{x_1,k+i} ) + H_m ((P_2 (\mu_{[x]}))^{x_2, k+i}) \geq 1+ H_m (\pi_{V_{m} ^\perp} ( (P_1 \mu)^{x_1,k+i} \times (P_2(\mu_{[x]}))^{x_2, k+i}) -\epsilon 
\end{equation*}
\begin{equation*}
\geq 1 + H_m ( (P_1 \mu)^{x_1,k+i}) - O(\frac{1}{m})  - \epsilon.
\end{equation*}
Thus,
\begin{equation*} 
H_m ((P_2 (\mu_{[x]}))^{x_2, k+i}) \geq 1 - O(\frac{1}{m})  - \epsilon.
\end{equation*}
Also, note that since $\mu_{[x]}$ is supported on a parallel line to the $y$-axis, we have
\begin{equation} \label{Equation equal entropy}
H_m ((P_2(\mu_{[x]}))^{x_2, k+i}) = H_m ((\mu_{[x]})^{x_2 ', k+i}),
\end{equation}
where the $x_2 '$ is understood to represent a point on $P_1 ^{-1} (x)$ on the right hand side of \eqref{Equation equal entropy}. so we have
\begin{equation} \label{Equation last calcultion}
H_m ((\mu_{[x]})^{x_2 ', k+i}) \geq 1 - O(\frac{1}{m})  - \epsilon.
\end{equation}

Now, let $A_m \in G_{F,e_2}$ be the support of the measure $(\mu_{[x]})^{x_2 ', k+i}$ (it is supported on a mini-set in this gallery by definition). Therefore, by the standard properties of entropy, and by equation \eqref{Equation last calcultion}
\begin{equation*}
\log |\lbrace D \in D_m : D\cap A_m \neq \emptyset \rbrace| \geq  H( (\mu_{[x]})^{x_2 ', k+i}, D_m) = m \cdot H_m ((\mu_{[x]})^{x_2 ', k+i}) \geq m\cdot (1 - O(\frac{1}{m})  - \epsilon).
\end{equation*}
Therefore, we obtain the first option stated in the Lemma
\begin{equation*}
\max_{A\in G_{F^x}} \log N_m (A) \geq \log N_m (A_m) \geq m\cdot (1 - O(\frac{1}{m})  - \epsilon). 
\end{equation*}

If  equation \eqref{Equation insert 2} is satisfied, then we arrive, via an analogues argument, to the conclusion that for some component measure of $P_1\mu$,
\begin{equation*} 
H_m ((P_1 \mu)^{x_1,k+i}) \geq 1 - O(\frac{1}{m})  - \epsilon.
\end{equation*}
Therefore, as above
\begin{equation*}
\max_{A\in G_{P_1 (F)}} \log N_m (A) \geq \log N_m (A_m) \geq m\cdot (1 - O(\frac{1}{m})  - \epsilon). 
\end{equation*}
This is the second option in the Lemma.

We remain with the case when $V_{m}=\mathbb{R}^2$. In this case, we can conclude that (using equation \eqref{Equation equal entropy}),
\begin{equation*}
H_m ( (P_1 \mu)^{x_1,k+i} ) + H_m ((\mu_{[x]})^{x_2 ', k+i}) \geq 2 - \epsilon
\end{equation*}
via an analogue of equation \eqref{Equation saturation}. It follows that  either
\begin{equation*}
H_m ( (P_1 \mu)^{x_1,k+i} ) \geq 1 - \frac{\epsilon}{2}
\end{equation*}
or
\begin{equation*}
H_m ((\mu_{[x]})^{x_2 ', k+i}) \geq 1 - \frac{\epsilon}{2}.
\end{equation*}
Thus, either
\begin{equation*}
\max_{A\in G_{F,e_2}} \log N_m (A)   \geq m\cdot (1 - \frac{\epsilon}{2}). 
\end{equation*}
or
\begin{equation*}
\max_{A\in G_{P_1 (F)}} \log N_m (A)   \geq m\cdot (1 - \frac{\epsilon}{2}).  
\end{equation*}
which are options (3) and (4) in the Lemma.
\end{proof}

We have seen that for every $m\in \mathbb{N}$ such that equation \eqref{Equation increase of entropy} fails for $\delta(\epsilon,m,X)$, one of the  four options stated in Lemma \ref{Lemma four options} must hold. Since we are assuming, towards a contradiction, that equation \eqref{Equation increase of entropy} fails for every $m$,  one of the above options in the Lemma must happen for infinitely many $m$. If either option (1) or (3) occur for infinitely many $m$, then it follows that for this subsequence $m_k$ either
\begin{equation*}
\Delta (G_{F,e_2}) =  \lim_{m\rightarrow \infty} \frac{\log N_m (G_{F,e_2})}{m} = \lim_{k\rightarrow \infty} \frac{\log N_{m_k} (G_{F,e_2})}{m_k} \geq \lim_{k\rightarrow \infty} 1 - O(\frac{1}{m_k})  - \epsilon = 1-\epsilon.  
\end{equation*}
or
\begin{equation*}
\Delta (G_{F,e_2}) \geq 1- \frac{\epsilon}{2}
\end{equation*}
This contradicts our choice of $\epsilon$ in \eqref{Equation choice of epsilon}. If options (2) or (4) happen infinitely often, then similarly
\begin{equation*}
\Delta (G_{P_1 (F)}) \geq 1- \epsilon
\end{equation*}
or
\begin{equation*}
\Delta (G_{P_1 (F)}) \geq 1- \frac{\epsilon}{2}
\end{equation*}
which again contradicts our choice of $\epsilon$. 

We conclude that for our chosen $\epsilon$, there exists some $m$ such that for the $\delta(\epsilon,m,X)$ from Theorem \ref{Theorem Inverse Theorem}, equation \eqref{Equation increase of entropy} is satisfied. As required.
\end{proof}
\subsubsection{Proof of Theorem \ref{Theorem Inverse Theorem for product measures - finite group}}
Recall that we are now assuming that $F$ is a self similar set such that its given generating IFS has the SSC and $1\leq |G_\Phi| <\infty$. The idea here is to use essentially  the proof of the previous subsection, with some modifications.  Recall from the remark following Lemma \ref{WSC for projection} that since $F$ is Furstenberg homogeneous, and $\gamma\cdot (P_1 (F),0) + t \subseteq F$ then $P_1 (F)$ is Furstenberg homogeneous. Thus, $G_{P_1 (F)}$, the collection of all mini-sets of $P_1 (F)$,  is a gallery in this case as well. Also, by the same argument as in the proof of Theorem \ref{Theorem Inverse Theorem for product measures - self homo}, $\dim P_1 (F)<1$ (since it embeds as a slice of $F$). By Example \ref{Gallery of slices - general}, and our assumptions on $F$ and $\Phi$, the collections  of all the mini-sets of vertical slices of $F$ $G_{F,e_2}$ is also a Furstenberg gallery. Recall that the SSC assumption on $\Phi$ implies that $\Delta (G_{F,e_2})<1$.

Thus, we have $\Delta (G_{F,e_2}), \Delta (G_{P_1 (F)})<1$. So, let $\epsilon$ be small enough so that 
\begin{equation} 
\max \lbrace \Delta (G_{F,e_2}) , \Delta (G_{P_1 (F)}) \rbrace < 1- \epsilon.
\end{equation}
Let $X \subset G_2$ be a bounded (see Step 15 of the proof  of part (2) of Theorem \ref{Theorem embeddings} for the explicit $X$ we need).

Now, for every $m\in \mathbb{N}$ produce the $\delta$ corresponding to $\epsilon,m,X$ and any measure $\nu$  supported on $X$ (see the remarks after Theorem \ref{Theorem Inverse Theorem}) of entropy dimension at least $1$. We claim that there exists some $m\in \mathbb{N}$ such that:

Equation \eqref{Equation increase of entropy} in Theorem \ref{Theorem Inverse Theorem} is satisfied, when we apply the Theorem for $\nu$ and any member of the following family of product measures: For every  $F' \subseteq F$ such that $F'$ is a homothetic set, let $\mu$ be the natural self similar measure on $F'$, and let   $\theta = P_1 \mu \times P_2 (\mu_{[x]}|_{S}) \in \text{Pro}(F')$ (where $S$ has positive $\mu_{[x]}$ measure, see equation \eqref{Equation for Pro(F)}), such that $\theta$ is not supported on any affine line.

If this is not true, then for every $m\in \mathbb{N}$ we can find such $F' \subseteq F$ and such a  product measure $P_1 (\mu) \times P_2 (\mu_{[x]})$ such that equation \eqref{Equation increase of entropy} fails (we suppress the $S$ in the notation).  For every such $m$  we can apply Lemma \ref{Lemma four options} for the measure $P_1 (\mu) \times P_2 (\mu_{[x]})$ and deduce that one of the following options must hold:
\begin{enumerate}
\item $\max_{A\in G_{F' ,e_2}} \log N_m (A)  \geq m\cdot (1 - O(\frac{1}{m})  - \epsilon). $
\item $\max_{A\in G_{P_1 (F')}} \log N_m (A)  \geq m\cdot (1 - O(\frac{1}{m})  - \epsilon). $
\item $\max_{A\in G_{P_1 (F'))}} \log N_m (A)  \geq m\cdot (1   - \frac{\epsilon}{2}). $
\item $\max_{A\in G_{F' ,e_2}} \log N_m (A)  \geq m\cdot (1   - \frac{\epsilon}{2}). $
\end{enumerate}

It remains to note that $F' \subseteq F$ so we have
\begin{equation*}
\max_{A\in G _{F,e_2}} \log N_m (A)  \geq \max_{A\in G_{F',e_2}} \log N_m (A)  
\end{equation*} 
and since $P_1 (F') \subset P_1 (F)$ we have
\begin{equation*}
\max_{A\in G _{P_1 (F)}} \log N_m (A)  \geq \max_{A\in G_{P_1 (F')}} \log N_m (A).
\end{equation*}

Therefore, our assumption towards a contradiction means that one of  one of the following options must hold for infinitely many $m$:
\begin{enumerate}
\item $\max_{A\in G_{F,e_2}} \log N_m (A)  \geq m\cdot (1 - O(\frac{1}{m})  - \epsilon). $
\item $\max_{A\in G_{P_1 (F)}} \log N_m (A)  \geq m\cdot (1 - O(\frac{1}{m})  - \epsilon). $
\item $\max_{A\in G_{P_1 (F))}} \log N_m (A)  \geq m\cdot (1   - \frac{\epsilon}{2}). $
\item $\max_{A\in G_{F,e_2}} \log N_m (A)  \geq m\cdot (1   - \frac{\epsilon}{2}). $
\end{enumerate}
However, one sees that our choice of $\epsilon$ guarantees that this is impossible (see the end of the proof of Theorem \ref{Theorem Inverse Theorem for product measures - self homo}). The Theorem follows.

\bibliography{bib}{}

\begin{thebibliography}{10}

\bibitem{algom2016affine}
Amir Algom.
\newblock Affine embeddings of cantor sets on the line.
\newblock {\em To appear in Journal of Fractal Geometry}, 2017.

\bibitem{algom2016self}
Amir Algom and Michael Hochman.
\newblock Self embeddings of bedford-mcmullen carpets.
\newblock {\em To appear in Ergodic Theory and Dynamical Systems}, 2017.

\bibitem{bishop2013fractal}
Christopher~J. Bishop and Yuval Peres.
\newblock {\em Fractals in probability and analysis}, volume 162.
\newblock Cambridge University Press, 2016.

\bibitem{bonk2013quasi}
Mario Bonk and Sergei Merenkov.
\newblock Quasisymmetric rigidity of square {S}ierpi\'{n}ski carpets.
\newblock {\em Ann. of Math. (2)}, 177(2):591--643, 2013.

\bibitem{deng2011bilipschitz}
Juan Deng, Zhi-Ying Wen, Ying Xiong, and Li-Feng Xi.
\newblock Bilipschitz embedding of self-similar sets.
\newblock {\em J. Anal. Math.}, 114:63--97, 2011.

\bibitem{Deng2013equivalence}
Qi-Rong Deng and Ka-Sing Lau.
\newblock On the equivalence of homogeneous iterated function systems.
\newblock {\em Nonlinearity}, 26(10):2767--2775, 2013.

\bibitem{Deng2017structure}
Qi-Rong Deng and Ka-Sing Lau.
\newblock Structure of the class of iterated function systems that generate the
  same self-similar set.
\newblock {\em J. Fractal Geom.}, 4(1):43--71, 2017.

\bibitem{einsiedler2011ergodic}
Manfred Einsiedler and Thomas Ward.
\newblock {\em Ergodic theory with a view towards number theory}.
\newblock Graduate Texts in Mathematics, volume 259. Springer-Verlag London,
  2011.

\bibitem{elekes2010self}
M{\'a}rton Elekes, Tam{\'a}s Keleti, and Andr{\'a}s M{\'a}th{\'e}.
\newblock Self-similar and self-affine sets: measure of the intersection of two
  copies.
\newblock {\em Ergodic Theory Dynam. Systems}, 30(2):399--440, 2010.

\bibitem{falconer1992lipschitz}
K.~J. Falconer and D.~T. Marsh.
\newblock On the {L}ipschitz equivalence of {C}antor sets.
\newblock {\em Mathematika}, 39(2):223--233, 1992.

\bibitem{falconer1997techniques}
Kenneth Falconer.
\newblock {\em Techniques in fractal geometry}.
\newblock John Wiley \& Sons, Ltd., Chichester, 1997.

\bibitem{falconer1986geometry}
Kenneth~J Falconer.
\newblock {\em The geometry of fractal sets}, volume~85.
\newblock Cambridge university press, 1986.

\bibitem{Falconer2014Jin}
Kenneth~J. Falconer and Xiong Jin.
\newblock Exact dimensionality and projections of random self-similar measures
  and sets.
\newblock {\em J. Lond. Math. Soc. (2)}, 90(2):388--412, 2014.

\bibitem{farkas2015dimension}
{\'A}bel Farkas.
\newblock Dimension approximation of attractors of graph directed ifss by
  self-similar sets.
\newblock {\em arXiv preprint arXiv:1505.05746}, 2015.

\bibitem{Farkas2014projections}
{\'A}bel Farkas.
\newblock Projections of self-similar sets with no separation condition.
\newblock {\em Israel J. Math.}, 214(1):67--107, 2016.

\bibitem{feng2014affine}
De-Jun Feng, Wen Huang, and Hui Rao.
\newblock Affine embeddings and intersections of {C}antor sets.
\newblock {\em J. Math. Pures Appl. (9)}, 102(6):1062--1079, 2014.

\bibitem{feng2009structures}
De-Jun Feng and Yang Wang.
\newblock On the structures of generating iterated function systems of {C}antor
  sets.
\newblock {\em Adv. Math.}, 222(6):1964--1981, 2009.

\bibitem{feng2016affine}
De-Jun Feng and Ying Xiong.
\newblock Affine embeddings of cantor sets and dimension of $\alpha \beta
  $-sets.
\newblock {\em To appear in Israel J. Math.}, 2017.

\bibitem{Fraser2015Assouad}
J.~M. Fraser, A.~M. Henderson, E.~J. Olson, and J.~C. Robinson.
\newblock On the {A}ssouad dimension of self-similar sets with overlaps.
\newblock {\em Adv. Math.}, 273:188--214, 2015.

\bibitem{fraser2014assouad}
Jonathan Fraser.
\newblock Assouad type dimensions and homogeneity of fractals.
\newblock {\em Transactions of the American Mathematical Society},
  366(12):6687--6733, 2014.

\bibitem{FraserOrponen2015assouad}
Jonathan~M Fraser and Tuomas Orponen.
\newblock The assouad dimensions of projections of planar sets.
\newblock {\em arXiv preprint arXiv:1509.01128}, 2015.

\bibitem{furstenberg2008ergodic}
Hillel Furstenberg.
\newblock Ergodic fractal measures and dimension conservation.
\newblock {\em Ergodic Theory and Dynamical Systems}, 28(02):405--422, 2008.

\bibitem{hochman2015self}
Michael Hochman.
\newblock On self-similar sets with overlaps and inverse theorems for entropy
  in $\mathbb {R}^d$.
\newblock {\em To appear in Memoires of the AMS}.

\bibitem{hochman2016some}
Michael Hochman.
\newblock Some problems on the boundary of fractal geometry and additive
  combinatorics.
\newblock {\em To appear in Proceedings of FARF 3}.

\bibitem{hutchinson1981fractals}
John~E. Hutchinson.
\newblock Fractals and self-similarity.
\newblock {\em Indiana Univ. Math. J.}, 30(5):713--747, 1981.

\bibitem{kaenmaki2015weak}
Antti K\"aenm\"aki and Eino Rossi.
\newblock Weak separation condition, {A}ssouad dimension, and {F}urstenberg
  homogeneity.
\newblock {\em Ann. Acad. Sci. Fenn. Math.}, 41(1):465--490, 2016.

\bibitem{Lau1999weak}
Ka-Sing Lau and Sze-Man Ngai.
\newblock Multifractal measures and a weak separation condition.
\newblock {\em Adv. Math.}, 141(1):45--96, 1999.

\bibitem{mattila1999geometry}
Pertti Mattila.
\newblock {\em Geometry of sets and measures in {E}uclidean spaces}, volume~44
  of {\em Cambridge Studies in Advanced Mathematics}.
\newblock Cambridge University Press, Cambridge, 1995.
\newblock Fractals and rectifiability.

\bibitem{shmerkin2016furstenberg}
Pablo Shmerkin.
\newblock On furstenberg's intersection conjecture, self-similar measures, and
  the $ l^{q} $ norms of convolutions.
\newblock {\em To appear in Annals of Mathematics}, 2018.

\bibitem{wu2016proof}
Meng Wu.
\newblock A proof of furstenberg's conjecture on the intersections of $\times p
  $ and $\times q $-invariant sets.
\newblock {\em arXiv preprint arXiv:1609.08053}, 2016.

\bibitem{Zerner1996weak}
Martin P.~W. Zerner.
\newblock Weak separation properties for self-similar sets.
\newblock {\em Proc. Amer. Math. Soc.}, 124(11):3529--3539, 1996.

\end{thebibliography}
\bibliographystyle{plain}

\end{document}